\documentclass[oneside]{amsart}
\usepackage{amssymb}
  \usepackage{hyperref}
\usepackage[a4paper]{geometry}
\usepackage{dsfont}

\theoremstyle{plain}
\newtheorem{thm}{Theorem}[section] 
\newtheorem{prop}[thm]{Proposition} 
\newtheorem{lem}[thm]{Lemma} 
\newtheorem{cor}[thm]{Corollary}
\theoremstyle{remark} 
\newtheorem*{rem}{Remark} 
\theoremstyle{definition}

\numberwithin{equation}{section}

\renewcommand*{\div}{\operatorname{div}}
\newcommand*{\curl}{\operatorname{curl}}

\newcommand*{\supp}{\operatorname{supp}}
\newcommand*{\loc}{\mathrm{loc}}

\usepackage[textsize=tiny,textwidth=27mm]{todonotes}

\usepackage{graphicx,pifont}

\newcommand*{\bydef}{\overset{\rm def}{=}}
\newcommand*{\norm}[1]{\left\Vert #1\right\Vert}

\begin{document}

\title[Axisymmetric Incompressible Viscous Plasmas]{Axisymmetric Incompressible Viscous Plasmas: Global Well-Posedness and Asymptotics}

\author[1]{Diogo Ars\'enio}
\author[2]{Zineb Hassainia}
\author[2]{Haroune Houamed}
\address{New York University Abu Dhabi \\
Abu Dhabi \\
United Arab Emirates}
\email{diogo.arsenio@nyu.edu, zh14@nyu.edu, haroune.houamed@nyu.edu}

\date{\today}

\begin{abstract}
	This paper is devoted to the global analysis of the three-dimensional axisymmetric Navier--Stokes--Maxwell equations. More precisely, we are able to prove that, for large values of the speed of light $c\in (c_0, \infty)$,  for some threshold $c_0>0$ depending only on the initial data, the system in question admits a unique global solution. The ensuing bounds on the solutions are uniform with respect to the speed of light, which allows us to study the singular regime $c\rightarrow \infty$ and  rigorously derive the limiting   viscous magnetohydrodynamic (MHD) system in the axisymmetric setting.
	
	The strategy of our proofs draws insight from recent results on the two-dimensional incompressible Euler--Maxwell system to exploit the dissipative--dispersive structure of Maxwell's system in the axisymmetric setting. Furthermore, a detailed analysis of the asymptotic regime $c\to\infty$ allows us to derive a robust nonlinear energy estimate which holds uniformly in $c$. As a byproduct of such refined uniform estimates, we are able to describe the global strong convergence of solutions toward the MHD system.
	
	This collection of results seemingly establishes the first available global well-posedness of three-dimensional viscous plasmas, where the electric and magnetic fields are governed by the complete Maxwell equations, for large initial data as $c\to\infty$.
\end{abstract}

\maketitle

\setcounter{tocdepth}{1}
\tableofcontents


\section{Introduction and main results}
In this paper, we consider the       incompressible Navier--Stokes--Maxwell equations
\begin{equation}\label{Maxwell:system:*2}
	\begin{cases}
		\begin{aligned}
			\text{\tiny(Navier--Stokes's equation)}&&&\partial_t u +u \cdot\nabla u   = \nu \Delta u- \nabla p + j \times B, &\div u =0,&
			\\
			\text{\tiny(Amp\`ere's equation)}&&&\frac{1}{c} \partial_t E - \nabla \times B =- j , &\div E = 0,&
			\\
			\text{\tiny(Faraday's equation)}&&&\frac{1}{c} \partial_t B + \nabla \times E  = 0 , &\div B = 0,   &
			\\
			\text{\tiny(Ohm's law)}&&&j= \sigma \big( cE + P(u \times B)\big), &\div j = 0,&
		\end{aligned}
	\end{cases}
\end{equation}
 completed with a divergence-free  initial data $$ (u,E,B)|_{t=0} =  (u_0,E_0,B_0).$$  Here, $u$, $E$, $B$ and $j$ are defined on the whole space $ \mathbb{R}^+\times \mathbb{R}^d$,  with $d\in \{ 2,3\}$, and take their values in $\mathbb{R}^3$. The operator 
  $P \bydef \text{Id} - \nabla \Delta^{-1} \div $ denotes   Leray's projector on divergence-free vector fields.   Note that the divergence-free condition $\div B=0$ is not  a constraint  and \eqref{Maxwell:system:*2} is therefore not overdetermined.
  Instead,  it is a property   propagated by the flow provided that  it holds initially. 
  
The Navier--Stokes--Maxwell equations describe the evolution of a viscous plasma,  a charged gas or an electrically conducting fluid subject to the self-induced electromagnetic Lorentz force $j\times B$.

     As for the physical meaning of the quantities in \eqref{Maxwell:system:*2},  the electric and magnetic fields are denoted by    $ E$ and $ B$, respectively, while the electric current is denoted by     $j$. Additionally, the positive constants  $\nu$, $\sigma$  and $c$ stand for the fluid viscosity, the electrical conductivity and the speed of light, respectively. We refer the interested reader to \cite{bis-book, D-book} for further details on the underlying physical theories   about plasma modeling and to \cite{as} for a mathematical derivation  of \eqref{Maxwell:system:*2} through the analysis of the
viscous incompressible hydrodynamic regimes of Vlasov--Maxwell--Boltzmann systems.

The only known a priori  global information on solutions of \eqref{Maxwell:system:*2} is given by the \textit{energy inequality}
\begin{equation}\label{energy-inequa}
	\norm {\left( u , E ,B\right)(t) }_{L^2}^2+2\nu\int_0^t \norm {\nabla u(\tau)}_{L^2}^2 d\tau
	+\frac{2}{\sigma}\int_0^t \norm {j(\tau)}_{L^2}^2 d\tau \leq \mathcal{E}_0^2,
\end{equation}
for all $t\geq 0$, where we denote
\begin{equation*}
	\mathcal{E}_0 \bydef \norm {(u_0,E_0,B_0)}_{L^2}.
\end{equation*}
Note that \eqref{energy-inequa} is in fact an equality for smooth solutions.

Observe that, at least formally, taking the limit of $c$ to infinity in \eqref{Maxwell:system:*2} yields the well-known magneto-hydrodynamic (MHD) system
\begin{equation}\label{MHD*}
	\begin{cases}
		\begin{aligned}
		&	 \partial_t u +u \cdot\nabla u   = \nu \Delta u- \nabla   p   +B\cdot \nabla B  , 
			\\
			&   \partial_t B  + u\cdot \nabla B = \frac{1}{\sigma} \Delta B + B\cdot \nabla u .
		\end{aligned}
	\end{cases}\tag{MHD}
\end{equation}
Similarly to \eqref{Maxwell:system:*2}, solutions to the   system of equations \eqref{MHD*} enjoy the energy inequality 
\begin{equation}\label{energy-inequa:MHD}
	\norm {\left( u  ,B  \right)(t) }_{L^2}^2+2\nu\int_0^t \norm {\nabla u(\tau)}_{L^2}^2 d\tau
	+\frac{2}{\sigma}\int_0^t \norm {\nabla B(\tau)}_{L^2}^2 d\tau \leq\norm {\left( u_0 ,B_0\right)}_{L^2}^2,
\end{equation}
for any $t\geq 0$.

Although   Leray-type weak solutions for \eqref{MHD*} satisfying \eqref{energy-inequa:MHD}  are known to exist, this remains unknown for \eqref{Maxwell:system:*2}, even for large values of $c$.

  The issue comes from the fact that standard compactness methods, usually used to prove the existence of weak solutions, raise the following question: Given two sequences $(j_n)_{n\in \mathbb{N}}$, $ ( B_n)_{n\in \mathbb{N}}$, bounded    in  $ L^2_{t,x}$ and $ L^\infty_tL^2_x$ and converging weakly to $ j$ and $B$, respectively, does it hold that 
$ j_n\times B_n$ converges to $j\times B$, at least in the sense of distributions?
This remains unsettled as there is no known phenomenon which prevents the build-up of high frequencies in $ j_n\times B_n  $, as $n$ tends to infinity. 

One way to overcome this lack of compactness consists in propagating some  additional regularity on the magnetic field. Based on this idea, the first attempt to construct a (unique) global solution to the two dimensional Navier--Stokes--Maxwell system \eqref{Maxwell:system:*2} is due to Masmoudi \cite{MN} who proved that, given any $s>0$, if $u_0\in L^2(\mathbb{R}^2)$ and $(E_0,B_0)\in  H^s(\mathbb{R}^2)$, then there is a unique global weak solution to \eqref{Maxwell:system:*2}. Apart from   the energy estimate, the solution constructed therein does not satisfy any  uniform bound with respect to the speed of light $c$.

 Later on,  the first author, Ibrahim and Masmoudi \cite{aim15} established a  conditional convergence result  which entails   the convergence of weak solutions of \eqref{Maxwell:system:*2} to \eqref{MHD*}. 
 
  Crucial progress was then achieved in  \cite{ag20}    by the first author and Gallagher by showing  the persistence of the  $H^s$-regularity of the electromagnetic field, uniformly with respect to the speed of light $c$.

As for the inviscid version of \eqref{Maxwell:system:*2}, i.e., the incompressible Euler--Maxwell system, the first global results   on that model  were established recently by  the first and the third authors in  \cite{ah,ah2}. In that context, owing to   known results on Euler equations, see for instance \cite{DS:1963}, working at the level of the energy \eqref{energy-inequa}, with $\nu=0,$  is  insufficient to ensure  the global well-posedness of the system. One   is  therefore  compelled to study the existence and   uniqueness of weak solutions in a higher-regularity space. 

The main contribution from \cite{ah} is then the construction of a unique global solution of the Euler--Maxwell equations  \eqref{Maxwell:system:*2} (with $\nu=0$) in the spirit of  Yudovich's work \cite{Yudovich1}, where  the electromagnetic field has some Sobolev regularity $ H^s(\mathbb{R}^2), $ $s\in (\frac{7}{4}, 2)$.
 Moreover, it is shown therein that the solution is uniformly  bounded with respect to $c$ in   adequate spaces. The other work by the same authors \cite{ah2} further  establishes   the strong convergence of that solution, as $c$ goes to $\infty$.  We should point out that the results in \cite{ah,ah2}, for the   Euler--Maxwell system, only hold  under the assumption that $(u,E,B)$ has the \textit{two-dimensional normal structure}
 \begin{equation}\label{structure:2dim}
	u(t,x)=
	\begin{pmatrix}
		u_1(t,x)\\u_2(t,x)\\0
	\end{pmatrix},
	\qquad
	E(t,x)=
	\begin{pmatrix}
		E_1(t,x)\\E_2(t,x)\\0
	\end{pmatrix}
	\qquad\text{and}\qquad
	B(t,x)=
	\begin{pmatrix}
		0\\0\\b(t,x)
	\end{pmatrix}.
\end{equation}
 Extending the arguments from \cite{ah,ah2} to general structures is  a challenging open problem. As we shall see later on, a novelty of the present work is to permit the consideration of a new   three-dimensional structure in the context of plasmas, namely, the axisymmetric structure.
 
   From now on, we are going to  focus on the three dimensional case by first considering   the work by Ibrahim and Keraani \cite{IK2011} where it is proved that the system \eqref{Maxwell:system:*2} is globally well-posed provided that the initial data $(u_0,E_0,B_0)$ are small enough in 
 $$ \dot{B}_{2,1}^\frac{1}{2} \times \dot{H}^\frac{1}{2}\times \dot{H}^\frac{1}{2}  (\mathbb{R}^3).$$
 This smallness condition is slightly weakened in the work of Germain, Ibrahim and Masmoudi \cite{gim}, where the initial data lies in the space
 $$ \dot{H}^\frac{1}{2}\times \dot{H}^\frac{1}{2}\times \dot{H}^\frac{1}{2}  (\mathbb{R}^3).$$
 Note that the scaling of these spaces is at the critical level of the three-dimensional Navier--Stokes equations (i.e., in the case $E\equiv B\equiv 0$). Accordingly, we do not expect to maintain the uniqueness of the solution to \eqref{Maxwell:system:*2}  if the initial velocity  lies in spaces that scale below $\dot{H}^{\frac{1}{2}}$.
 
  Beside \cite{gim, IK2011},   the three-dimensional Navier--Stokes--Maxwell system \eqref{Maxwell:system:*2} was studied in several papers, for instance  \cite{ag20, YG20, YG22, YGZ16}. However, to the best of our knowledge, none of the existing results   provide enough information to study the regime $c\rightarrow \infty$ in the three-dimensional case.  
 
Thus, an important novelty of our paper is achieved by shedding   light on the  regime $c\rightarrow \infty$  and extending the techniques recently developed by the first and the third authors in \cite{ah,ah2} to the three-dimensional axisymmetric setting, which we introduce next.

 \subsection{The axisymmetric structure}\label{axi-subsection}
Throughout this paper, we assume that the velocity and the electromagnetic fields are axisymmetric, $u$ and $E$ are without swirl, whereas, $B$ has pure swirl. That is to say 
\begin{equation}\label{structure:axi}
u(x)=u_r(r,z) e_r + u_z(r,z) e_z, \quad  E(x)=E_r(r,z) e_r + E_z(r,z) e_z, \quad B(x)= B_\theta (r,z)e_\theta, 
\end{equation}
where $(r,\theta,z)$ denotes the usual cylindrical coordinates and $ (e_r,e_\theta,e_z)$ is   the corresponding orthonormal basis (see Section \ref{section:axi and para product} for the explicit  definition of these variables).
The careful reader may now notice that   this assumption shares some similarity with   the \textit{two-dimensional normal structure} \eqref{structure:2dim}. Indeed, in \eqref{structure:axi}, the magnetic field $B$  remains orthogonal to $u$ and $E$.  

This key observation will allow us to somewhat simplify the equations. However, this structure will also require some extra attention when performing estimates in Besov spaces, for the cylindrical basis $ (e_r,e_\theta,e_z)$ depends on the position variables.

  Let us now state some crucial properties provided by   \eqref{structure:axi}. First of all, note that the vorticity $$\omega\bydef \curl u$$ is acting only in the direction $e_\theta$, that is, $$\omega = \omega_\theta\, e_\theta \bydef \big(\partial_z u_r - \partial_r u_z \big) e_\theta.$$ On the other hand, straightforward computations show  that
\begin{equation*}
\nabla \times (j\times B) =  B \cdot \nabla j - j \cdot \nabla B = j_r\frac{B}{r}  -  j \cdot \nabla B.
\end{equation*}  
   Thus, taking the curl of the momentum equation in \eqref{Maxwell:system:*2} yields   
  \begin{equation}\label{omega-equa}
  \partial_t \omega_\theta + u\cdot\nabla \omega_{\theta} - \nu\big(\Delta - \frac{1}{r^2} \big)\omega_\theta= \frac{u_r}{r} \omega_\theta +  j_r\frac{B_\theta}{r}  -   j \cdot \nabla B _\theta .
  \end{equation}
  Then, by introducing $$\Omega\bydef \frac{\omega_{\theta}}{r}, \qquad \Gamma\bydef \frac{B_\theta}{r},$$ we further obtain the   equation  
    \begin{equation}\label{Omega-equa}
  \partial_t \Omega  + u\cdot\nabla \Omega  - \nu \big(\Delta+ \frac{\partial_r}{r} \big)\Omega =      j \cdot \nabla  \Gamma  .
  \end{equation}
  Note that the last equation is very   similar to the vorticity equation of the  two-dimensional Euler--Maxwell system treated in \cite{ah}.

  An equivalent formulation of \eqref{omega-equa} and \eqref{Omega-equa} is   obtained  by employing     Amp\`ere's equation   to eliminate  the current density $j$ and   deduce that
  \begin{equation}\label{omega-equa-form2}
  \partial_t \omega_\theta + u\cdot\nabla \omega_{\theta} - \nu\big(\Delta - \frac{1}{r^2} \big)\omega_\theta= \frac{u_r}{r} \omega_\theta -\partial_z \big( \Gamma B_\theta\big)-  \frac{1}{c}\partial_t E _r\frac{B_\theta}{r}  + \frac{1}{c}\partial_t E   \cdot \nabla B _\theta ,
  \end{equation} 
  and
    \begin{equation}\label{Omega-equa-form2}
  \partial_t \Omega  + u\cdot\nabla \Omega - \nu \big(\Delta+ \frac{\partial_r}{r} \big)\Omega =  -\partial_z \big( \Gamma^2\big)  -  \frac{1}{c}\partial_t E \cdot \nabla \Gamma   .
  \end{equation}
 Furthermore, by virtue of  \eqref{structure:axi}, combining Amp\`ere and Faraday's equations yields   
  \begin{equation*}
  \frac{1}{c^2} \partial_{tt} B + \partial_t B + u\cdot\nabla B -\frac{1}{\sigma} \Delta B= \frac{u_r}{r} B,
  \end{equation*}
  and therefore
   \begin{equation*}
  \frac{1}{c^2} \partial_{tt} \Gamma + \partial_t \Gamma + u\cdot\nabla \Gamma -\frac{1}{\sigma}\big(\Delta+ \frac{\partial_r}{r} \big) \Gamma=0.
  \end{equation*}
  
   The key observation behind the global estimates in the axisymmetric setting when $E\equiv B\equiv 0$ is the fact that the stretching term $\omega\cdot \nabla u$ is reduced to $\frac{u_r}{r}\omega$.  As shown in  \cite{SY94}, this term can be controlled by the estimate\footnote{$L^{p,q}$ stands for the usual  Lorentz spaces.}  
   \begin{equation}\label{ur/r-interpolation}
\norm {\frac{u_r}{r}}_{L^\infty}\lesssim \norm {\Omega}_{L^{3,1}}\lesssim \norm {\Omega}_{L^2}^\frac{1}{2}  \norm {\nabla \Omega}_{L^2}^\frac{1}{2},
\end{equation}  
    which then yields a   global bound on the velocity in sub-critical spaces. Indeed, this is due to  the fact that $\Omega$ obeys a transport--diffusion equation (resp.\ transport equation) in the case of the Navier--Stokes equations (resp.\ Euler equations), which can be used to deduce      a global bound for $\Omega$ in $L^\infty_tL^2\cap L^2_t\dot{H}^1$   (resp.\ in $L^\infty_tL^{3,1}  $),   by means of standard energy estimates.  Similar results can be extended to the  MHD equations (see Theorem \ref{Thm:MHD}, below).

We conclude by pointing out that the axisymmetric structure has  been extensively studied in a variety of fluid models, such as the Navier--Stokes \cite{HFDZ17,HFDZ19,YL22,LY23A,LX23}, Euler \cite{D07}, Boussinesq \cite{HHK11,DH21,DWZ18,HHZ20,HHZ23,HR10,HZ20} and MHD \cite{YZ18,QYZ17,LZ15,QY22,LY18,WZ22} systems with partial or full dissipation.

   \subsection{Aims and   main results}\label{Section-aims}
   We intend to show that the global well-posedness of the three dimensional  system  \eqref{Maxwell:system:*2}  holds  whenever $c $ is large enough. We shall prove our results uniformly with respect to $c$ and then   derive the  MHD equations by studying  the singular limit $c\rightarrow \infty$.    We refer to   Section \ref{section:axi and para product} for the definition of all  functional spaces.
 
\begin{thm}[Global well-posedness]\label{Thm:1}
Let   $(u_0^c,E_0^c,B_0^c)_{c>0}$ be a family of divergence-free axisymmetric  initial data   such that  $ u_0^c$ and $E_0^c$ are without swirl, whereas $B_0^c$ has pure swirl. Assume further that 
$$ (u_0^c,E_0^c,B_0^c) \in \left( H^1\times   H^{\frac{3}{2}}\times   H^{\frac{3}{2}} \right)(\mathbb{R}^3), \qquad c^{-1} ( E_0^c,B_0^c) \in \dot{B}^{\frac{5}{2}}_{2,1}  (\mathbb{R}^3),$$ 
$$  \Omega_0^c \bydef \frac{\omega_0^c}{r}  \in L^2(\mathbb{R}^3) , \qquad  j_0^c \bydef \sigma( cE_0^c + P(u_0^c\times B_0^c)  ) \in  H^\frac{1}{2}(\mathbb{R}^3), $$ 
uniformly in $c>0$. Then,  there is a constant $c_0>0$,   depending only on the size of the initial data,   such that,  for any $c\in (c_0,\infty)$, there is a unique     global axisymmetric solution $(u^c,E^c,B^c)$ of the three dimensional Navier--Stokes--Maxwell equations \eqref{Maxwell:system:*2}, with $\nu>0$,     such that $ u^c$ and $E^c$ are without swirl and $B^c$ has pure swirl.  This solution  enjoys the bounds 
 \begin{equation*}
 \begin{gathered}
  u^c \in L^\infty(\mathbb{R}^+; H^1), \qquad \nabla u^c\in  L^2(\mathbb{R}^+; H^1), \qquad \Omega^c\bydef \frac{\omega _{\theta}^c}{r} \in L^\infty(\mathbb{R}^+; L^2) \cap L^2(\mathbb{R}^+; \dot{H} ^1),
  \\  
  (E^c,B^c)\in L^\infty(\mathbb{R}^+; H^\frac{3}{2}), \qquad  c^{-1}(E^c,B^c)\in L^\infty(\mathbb{R}^+; \dot{B}^\frac{5}{2}_{2,1}),
  \\
  (E^c,B^c) \in L^2(\mathbb{R}^+; \dot{B} ^\frac{5}{2}_{2,1}),
  \quad cE^c\in   L^2(\mathbb{R}^+; H^\frac{3}{2}),
  \quad B^c\in   L^2(\mathbb{R}^+; \dot H^1),
  \quad j^c\in \bigcap_{p=2}^\infty L^p(\mathbb{R}^+; H^\frac{1}{2}),
   \end{gathered}
 \end{equation*}
 uniformly in $c\in (c_0,\infty)$.
If, moreover, 
$$( E_0^c,B_0^c) \in \dot{B}^{\frac{5}{2}}_{2,1}  (\mathbb{R}^3),$$
uniformly in $c>0$, then     the bound 
$$(E^c,B^c)\in L^\infty(\mathbb{R}^+; \dot{B}^\frac{5}{2}_{2,1})$$
holds uniformly in $c \in (c_0,\infty).$
\end{thm}

The complete proof of Theorem \ref{Thm:1} is given in Section \ref{section:closing:ES}.

\begin{rem}
	By setting $E\equiv B \equiv 0$ in the above theorem, observe that we recover a classical global well-posedness result for the axisymmetric Navier--Stokes equations (see, for instance, Theorem 10.4 in \cite{lemarie:2016}).
\end{rem}

We also wish to understand the singular regime $c\rightarrow \infty$ in \eqref{Maxwell:system:*2}.    The next corollary establishes a well-posedness result for the limiting system \eqref{MHD*} by considering the weak limit of solutions constructed in  Theorem \ref{Thm:1}.

\begin{cor}[Weak singular limit]\label{Thm:MHD}
Let $(u_0,B_0)$ be an axisymmetric divergence-free vector field,  where $u_0$ is without swirl and $B_0$ has pure swirl. Further assume  that 
\begin{equation*}
\begin{aligned}
(u_0, B_0) \in \left( H^1 \times  H^\frac{3}{2}\right)(\mathbb{R}^3), \quad \frac{\omega_0}{r} \in L^2(\mathbb{R}^3).
\end{aligned}
\end{equation*}
Then, there is a unique axisymmetric solution to \eqref{MHD*}, with $\nu>0$ and initial data $(u_0,B_0)$, enjoying the bounds 
\begin{equation*}
\begin{aligned}
(u,B) \in  L^\infty \left( \mathbb{R}^+;  H^1(\mathbb{R}^3) \times H^\frac{3}{2}(\mathbb{R}^3)\right)  , \quad (\nabla u,  \nabla  B) \in   L^2 \left( \mathbb{R}^+;  H^1(\mathbb{R}^3) \times B_{2,1}^\frac{3}{2}(\mathbb{R}^3)\right)   ,
\end{aligned}
\end{equation*} 
and
\begin{equation*}
\begin{aligned}
\frac{\omega_\theta}{r}   \in L^\infty ( \mathbb{R}^+;L^2(\mathbb{R}^3))\cap L^2 ( \mathbb{R}^+; \dot{H}^1(\mathbb{R}^3)).
\end{aligned}
\end{equation*}
\end{cor}

\begin{proof}
We build   global solutions by making use of  the bounds from the  preceding theorem and employing standard compactness methods.  To that end, let $(u_0^{c_n},E_0^{c_n},B_0^{c_n})_{n\in \mathbb{N}}$ be any family of divergence-free vector fields satisfying the assumptions of Theorem \ref{Thm:1} and converging  to $(u_0,0,B_0)$, at least in the sense of distributions. In particular, by Theorem \ref{Thm:1}, there exists a unique family $(u^{c_n},E^{c_n},B^{c_n})_{n\in \mathbb{N}}$ of solutions to \eqref{Maxwell:system:*2}, which obeys the bounds 
\begin{equation*}
\begin{aligned}
(u^{c_n},B^{c_n}) \in  L^\infty \left( \mathbb{R}^+;  H^1(\mathbb{R}^3) \times H^\frac{3}{2}(\mathbb{R}^3)\right)  , \quad (\nabla u^{c_n},  \nabla  B^{c_n}) \in   L^2 \left( \mathbb{R}^+;  H^1(\mathbb{R}^3) \times B_{2,1}^\frac{3}{2}(\mathbb{R}^3)\right)   ,
\end{aligned}
\end{equation*} 
and
\begin{equation*}
\begin{aligned}
\frac{\omega_\theta^{c_n}}{r}   \in L^\infty ( \mathbb{R}^+;L^2(\mathbb{R}^3))\cap L^2 ( \mathbb{R}^+; \dot{H}^1(\mathbb{R}^3)), \qquad c_nE^{c_n} \in L^2(\mathbb{R}^+; H^\frac{3}{2}),
\end{aligned}
\end{equation*}
uniformly in $n \in \mathbb{N}$.

 Thus, by the Banach--Alaoglu theorem, up to   extraction of a subsequence (which is not distinguished, for simplicity), it holds that 
$$(u^{c_n},E^{c_n},B^{c_n}) \overset{n\rightarrow \infty}{\longrightarrow} (u ,0 ,B ) , \quad \text { in } \, \mathcal{D}'(\mathbb{R}^+\times \mathbb{R}^3),$$
where $(u,B)$ is in the same functional spaces as $(u^{c_n},B^{c_n})$.  
 
 In fact, one can also show that $ (\partial_t u^{c_n} , \partial_t B^{c_n})_{n\in \mathbb{N}}$ 
 is uniformly  bounded in $L^2_{t,x,\loc}$, which implies,  by a classical compactness result by Aubin and Lions (see \cite{s87}  for a thorough discussion of such compactness results and, in particular, Section 9 therein, for convenient results which are easily applicable to our setting), that $(u^{c_n},B^{c_n})_{n\in \mathbb{N}}$ is relatively compact in $L^2_{t,x,\loc}$.
 
  Therefore, by taking the limit in \eqref{Maxwell:system:*2} in the sense of distributions, and exploiting the strong compactness of $(u^{c_n},B^{c_n})_{n\in \mathbb{N}}$ to show the weak stability of nonlinear terms, it is readily seen that $(u,B) $ is a weak solution of \eqref{MHD*}, thereby completing the existence proof.
  
   As for the  uniqueness of solutions, we   only need to note that  it follows directly from weak--strong uniqueness principles for fluid dynamical models (for instance, see \cite[Section 3.2.3]{as}  and \cite{GP02}).
\end{proof}  

\begin{rem}
The solution constructed in the previous corollary enjoys the  refined bound
\begin{equation}\label{RMK:HF}
	B\in \widetilde{L}^\infty\big( \mathbb{R}^+; \dot{B}^\frac{3}{2}_{2,2}(\mathbb{R}^3)\big).
\end{equation}
Indeed, by virtue of the bounds from Corollary \ref{Thm:MHD}, one can show that 
$ u\cdot \nabla B$  and $B\cdot \nabla u$ belong to  $ L^2( \mathbb{R}^+; \dot{H}^\frac{1}{2} (\mathbb{R}^3)) $. Therefore,  standard parabolic regularity estimates applied directly to the heat equation satisfied by $B$ in \eqref{MHD*}   give  the  bound above. 
 In particular, this bound guarantees that
\begin{equation}\label{0-limit:HF}
 \lim_{n \rightarrow \infty}  \sup_{t\in [0,\infty)}\| \mathds{1}_{\{|D|\geq  \Theta_n \}}  B  (t)\|_{\dot{H}^\frac{3}{2} } =0 ,   
\end{equation}  
for  any sequence $ (\Theta_n)_{n\geq 1}$ satisfying $$ \lim_{n\rightarrow \infty} \Theta_n = \infty.$$ 
Further note that  the weaker bound 
$$ B\in L^\infty \big( \mathbb{R}^+;   \dot{H} ^\frac{3}{2}(\mathbb{R}^3)\big)$$
would not be    enough  to establish \eqref{0-limit:HF}. This    will be important, later on, in the proofs of our main results.
\end{rem} 

We state now our second main theorem. It concerns the strong convergence, as $c\rightarrow \infty$, of the solution  given in Theorem \ref{Thm:1} toward the solution constructed in Corollary \ref{Thm:MHD}.

\begin{thm}[Strong singular limit]\label{Thm:CV}
Let   $(u_0^c,E_0^c,B_0^c)_{c>0}$ be a family of   initial data satisfying the assumptions in Theorem \ref{Thm:1} and denote by $(u^c,E^c,B^c)_{c>c_0}$ the corresponding unique global solution to \eqref{Maxwell:system:*2} given by the same theorem. Further consider   divergence-free vector fields  $(u_0,B_0)$   such that  
$$(u_0,B_0) \in  \left(  H^1\times H^\frac{3}{2}\right) (\mathbb{R}^3), \quad \frac{\omega_0}{r} \in L^2(\mathbb{R}^3) ,$$  
and  assume that
$$\lim_{c \rightarrow \infty} \norm {(u_0^c,B_0^c) - (u_0,B_0) }_{H^1\times L^2}=0. $$
 Finally, let   $(u,B)$ be  the  unique global solution of \eqref{MHD*}, given in Corollary \ref{Thm:MHD}, associated to the data $ (u_0,B_0)$. 
 
 Then $ (u^c,B^c)$ converges strongly to $(u,B)$,  as $c\rightarrow \infty$. More specifically, it holds that 
\begin{equation} \label{CV-u:statement}
\begin{aligned}
\lim_{c \rightarrow \infty}\left( \sup_{t\in [0,\infty)}\| (u^c-u) (t)\|_{\dot{H}^1 }   +  \int_0^ \infty \| (u^c-u) (t) \|_{ \dot{H}^{2} }^2 dt \right)  =0 ,
\end{aligned}
\end{equation}
and, for all $s\in [0,\frac{3}{2})$, 
\begin{equation}\label{CV-B:statement}
\begin{aligned}
\lim_{c \rightarrow \infty}\left( \sup_{t\in [0,\infty)}\| (B^c-B) (t)\|_{\dot{H}^s  }  +  \int_0^ \infty \| (B^c-B) (t) \|_{ \dot{H}^{s+ 1} }^2 dt \right)  =0 .
\end{aligned}
\end{equation}
If, moreover, we assume  that
\begin{equation*}
\lim_{c \rightarrow \infty} \norm {(E_0^c,B_0^c)-(0,B_0)}_{\dot{H}^\frac{3}{2}}=0
\end{equation*}  
and   
\begin{equation}\label{assumption:high:frequencies}
 \lim _{c\rightarrow \infty} \left( c^{-1} \| (E^c_0,B^c_0) \|_{  \dot{B}^\frac{5}{2}_{2,1}} \right) =0,
\end{equation} 
then \eqref{CV-B:statement} holds for $s=\frac{3}{2}$, as well.
\end{thm}

The complete proof of Theorem \ref{Thm:CV} is given in Section \ref{section:closing:CV}.

 \begin{rem} It is possible to quantify \eqref{CV-u:statement} and \eqref{CV-B:statement} with a rate of convergence $O(c^{-\alpha})$, for some $\alpha>0$, provided that the initial data satisfy a similar algebraic rate  of convergence. 
 \end{rem}
 
 \begin{rem}
 Observe that the initial data in Theorem \ref{Thm:1} are required to enjoy the regularity $ \dot{B}^{\frac{5}{2}}_{2,1}$. However, this is not needed  in the limiting system obtained in the regime $c\rightarrow \infty$,  as reflected in   Corollary \ref{Thm:MHD}.

  Moreover, the growth of the global solution from Theorem \ref{Thm:1} in that space is   of order $c$, which is consistent with the fact that it is uniformly bounded in the space $\dot{H}^\frac{3}{2}$.  Indeed, heuristically, we notice from our proofs, later on, that each spatial derivative has the same dimension as the speed of light $c$, i.e.,
 $$c^{-1} \dot{B}^{\frac{5}{2}}_{2,1}  \sim \dot{H}^{\frac{3}{2}} ,$$
 which means that the solutions have   a  comparable size in each space.  This is natural in view of the fact that waves produced by Maxwell's system have a characteristic speed $c$.
 
%

In conclusion, we supplement our remark with a typical example of   data which summarizes the preceding observations  and   fulfills   the assumptions of Theorem \ref{Thm:1} and Theorem \ref{Thm:CV}.  To that end, let $ \varphi\in C^\infty_c(\mathbb{R}^3)$ and introduce the standard mollifier 
$$ \varphi_c(\cdot) \bydef c^{3} \varphi(c\,\cdot).$$
Let $ (u_0,B_0)$ be the initial data from Corollary \ref{Thm:MHD} and $ E_0$ be any divergence-free profile in $ H^\frac{3}{2}.$  Then the sequence of data defined by
\begin{equation}\label{example1}
(u_0^c,E_0^c,B_0^c)  \bydef  \varphi_c * (u_0,0,B_0)
\end{equation}   
is suitable and satisfies all the assumptions in Theorem \ref{Thm:1} and Theorem \ref{Thm:CV}, as long as $c>c_0$, where $c_0>0$ only depends  on the size of  $u_0$, $E_0$, $B_0$ and $\varphi.$ 
 \end{rem}
 \subsection{Challenges and   ingredients of proof}\label{strategy of the proof} 
Here, we provide the reader with a short roadmap that sheds   light on our strategy and main ingredients towards proving our results. We already made it clear in the introduction that the construction of a (unique global) solution to \eqref{Maxwell:system:*2} at the level of the energy \eqref{energy-inequa} is an outstanding open problem. Hence, we seek      solutions in higher-regularity spaces. 

Notice that one cannot expect to have a better understanding of the Cauchy problem for   \eqref{Maxwell:system:*2} compared to the Navier--Stokes equations (i.e., the case  $E\equiv B \equiv 0$).  In the three-dimensional case, one of the well known settings where the Navier--Stokes equations are globally well-posed is when the initial data obey an axisymmetric geometric condition, which, roughly speaking, reduces the dimension from three to two, in some sense. In this work, we therefore make the choice of restricting ourselves to the case of axisymmetric data.

We first clarify the key idea leading to an estimate of the velocity field, which can be performed on   \eqref{omega-equa}--\eqref{Omega-equa} or \eqref{omega-equa-form2}--\eqref{Omega-equa-form2}. We argue now that it is better to work on   \eqref{omega-equa-form2}--\eqref{Omega-equa-form2}. The reasons behind this choice can be summarized in two crucial points. The first one is that the terms containing $\frac{1}{c}\partial_t E$ can be seen as an error   that should vanish and be discarded from the system when $c\rightarrow \infty$, at least formally. Observe that this formal limit yields exactly the equations for the vorticity in the axisymmetric case, which have been studied in several papers (for instance, see \cite{ZH}). Accordingly, we believe that   \eqref{omega-equa-form2}--\eqref{Omega-equa-form2} are more suitable to study the equations in the regime $c\rightarrow \infty$. We shall come back to this limit, later on, to comment on the formal claim \begin{equation}\label{E->0}
\lim_{c \rightarrow \infty} \frac{1}{c}\partial_t E =0.
\end{equation}   The second reason why we choose to consider \eqref{omega-equa-form2}--\eqref{Omega-equa-form2}   is more technical and it relies on the fact that  the  $L^2$ energy estimate \eqref{energy-inequa}  provides us with   weak information in dimension three, unlike the two dimensional case. Indeed, our alternative choice would be to consider \eqref{omega-equa}--\eqref{Omega-equa}
and to   estimate $j$ in $L^2_{t,x}$, which  is     the only    global bound on $j$ that can be extracted from the energy estimate. By doing so, one can can only establish an estimate of $(\omega,\Omega)$ in $L^\infty_t L^2_x\cap L^2_t\dot{H}^1_x$, which is linear in terms of the   $L^2_t B^{\frac{5}{2}}_{2,1}$ norm of $B$. This is   comparable  to the techniques from \cite{ah} where, in two dimensions, it is shown that such    estimates can be closed. However, in three dimensions, if one goes into the details of   low-frequency estimates, one notices that the interpolation argument used in \cite{ah} to weaken the power of   $B$ in some crucial norms will likely not work. It seems that this method  would  require a global control of $j$ in $L^2_t L^3_x$, which is not available. 

Thus, our main estimate  on the velocity field is based on \eqref{omega-equa-form2}--\eqref{Omega-equa-form2} and is given in Proposition \ref{prop.velocity-ES}, below, where we show that
 \begin{equation*}
	 \begin{aligned}
	    &\norm{(\omega,\Omega)}_{L^\infty_tL^2 \cap L^2_t\dot{H}^1 }
	     \\
		 &\qquad\lesssim   \left( \norm{(\omega_0,\Omega_0)}_{L^2}   +  \norm { \frac{1}{c}\partial_t E}_{L^2_t\dot{H}^\frac{1}{2}  }   \norm {B}_{L^\infty_tH^2   }     + \norm {\Gamma}_{L^\infty_t L^3 }\norm {(B,\Gamma)}_{L^2_t \dot{H}^1 } \right)
	     \exp \big(C\mathcal{E}^2_0\big),
	 \end{aligned}
 \end{equation*} 
 for some $C>0$. Now, in order to use this bound, we need  the following ingredients:
\begin{enumerate} 
\item An asymptotically  vanishing estimate for $\frac{1}{c} \partial_t E$ of the form
$$ \frac{1}{c}\norm {\partial_t E}_{L^2_t \dot{H}^\frac{1}{2}  } \lesssim c^{-\alpha} F\Big(\norm{(u,E,B)}_X\Big),  $$
for some $\alpha>\frac{1}{2}$, some (nonlinear) function $F$ and a suitable functional space $X$.
\item A bound   of the type
$$  \norm {B}_{ L^\infty_tH^2} \lesssim c^{\frac{1}{2}} F\Big(\norm{(u,E,B)}_X\Big).$$ 
\item An asymptotically global estimate for $ B$ in $L^2_t\dot{H}^1$ and $\Gamma$ in $L^\infty_t L^3  \cap L^2_t \dot{H}^1 $   of the form 
$$\norm B_{ L^2_t \dot{H}^1 } +  \norm { \Gamma}_{L^\infty_t L^3  \cap L^2_t \dot{H}^1  } \leq C_0 + c^{-\beta} F\Big(\norm{(u,E,B)}_X\Big),$$
for some $\beta>0$ and $C_0>0$ depending  only on the initial data.
\end{enumerate}
The first bound above is inspired by the results from \cite[Section 3]{ah2}, whereas the third one is  obtained  in the spirit of \cite[Section 3.6]{ah}. 

The precise proof of these claims is the subject of Section \ref{section:asymptotics} where we will build on the arguments  from \cite{ah,ah2}. For simplicity, the reader can think of the space $X$ in the foregoing estimates as a combination of the spaces appearing in the statement of Theorem \ref{Thm:1}. The precise construction of $X$ is also detailed  in Section \ref{section:asymptotics}. 

 The next step in our strategy is to study the Maxwell system 
\begin{equation}\label{Maxwell:system:*****}
		\begin{cases}
			\begin{aligned}
				\frac{1}{c} \partial_t E - \nabla \times B + \sigma c E & =- \sigma P( u \times B),
				\\
				\frac{1}{c} \partial_t B + \nabla \times E & = 0,
				\\
				\div u=\div B  = \div E& =0 .
			\end{aligned}
		\end{cases}
	\end{equation} 
The relevant   estimates for $(E,B)$  with a general forcing term have recently been established   in  \cite{ah}.   They are reproduced in Lemmas \ref{cor:maxwell} and   \ref{cor:parabolic:maxwell}. Nevertheless, some refinements  in the analysis of \eqref{Maxwell:system:*****} are required in order  to obtain adequate bounds for the electromagnetic field which are compatible  with the       asymptotic behavior of $ \frac{1}{c}\partial_t E$, $B$ and $\Gamma$. 

As a starting point, we will be facing the problem of estimating products of the form 
$$\norm { P(u \times B)}_{B^s_{p,q}},$$
with $s>\frac{3}{p}$, where  the regularity of $u$ is restricted to $ L^\infty_t \dot{H} ^1 \cap L^2_t \dot{H}^2$. Note that a similar issue was overcome in \cite{ah} by exploiting  the two dimensional normal structure \eqref{structure:2dim}, which is not valid in our context. Here, we show that  a similar analysis can be performed in the axisymmetric case, by using the fact that $B$ remains orthogonal to $E$ and $u$.  Accordingly, in Lemma \ref{paradifferential:1}, we provide  a  general result  refining   the classical paraproduct estimates in a framework which covers the axisymmetric setting.

 The ability to prove    Theorem \ref{Thm:1} further requires a precise understanding of the  damping phenomenon in  \eqref{Maxwell:system:*****}, which is obtained by studying the parabolic--hyperbolic properties of fields, their interaction and their behavior    relatively to $c$. This analysis is laid out in detail in  Section \ref{section.EB:ES}.

 With the above ingredients, we are then in a position to conclude Theorem \ref{Thm:1}. Thus, in Section \ref{Section:closing the estimates}, we gather all the estimates   to establish a   nonlinear bound of the form
$$ \norm {(u,E,B)}_{X} \leq C_0 + \mathcal{P}\left(  \norm {(u,E,B)}_{X} \right),$$
where $\mathcal{P}$ is a polynomial whose coefficients vanish as $ c\rightarrow \infty$. The conclusion of the   global estimates is then a straightforward application of the abstract Lemma \ref{fix-point:lem} which ensures that 
$$  \norm {(u,E,B)}_{X} \leq 2C_0,$$
as long as $c$ is larger than some power of $ C_0$.  

Note that the proof of Theorem \ref{Thm:1} is divided into two parts. The first one deals with the case that we call \textit{rough profiles,}  in the spirit of the example of initial data given in \eqref{example1}. In that case, the initial electromagnetic field $(E_0^c,B_0^c)$ can have a $ \dot{B}^{\frac{ 5}{2}}_{2,1}$ norm growing at most like $c$. In the second part of the proof, in the case of \textit{regular profiles,} we show that if, moreover, the corresponding $ \dot{B}^{\frac{5}{2}}_{2,1}$ norm of $(E_0^c,B_0^c)$ does not blow up as $c$ goes to infinity, then,     the  $ \dot{B}^{\frac{ 5}{2}}_{2,1}$ norm of $(E^c(t),B^c(t))$ remains bounded, for any $t\in [0,\infty)$. The second case will be proved with a kind of   bootstrap argument and it can be seen   as the  result of    persistence of   initial regularity.

Our second main result (Theorem \ref{Thm:CV}) establishes the convergence toward \eqref{MHD*} of the solution constructed in Theorem \ref{Thm:1}. It is to be emphasized that a fundamental ingredient in the proof of Theorem \ref{Thm:CV} hinges upon the understanding of the limit \eqref{E->0}. This is crucial to obtain a  convergence result in the whole domain $ [0,\infty)\times \mathbb{R}^3$. 

The proof of   Theorem \ref{Thm:CV} will be done in several steps. Firstly, we prove the  convergence  in the $L^2$ energy space by performing  adequate   stability estimates. Subsequently, by interpolation, it follows that the convergence  of  velocities and   magnetic fields  holds in $L^\infty_t \dot{H}^s \cap L^2_t \dot{H}^{s+1}$ and  $L^\infty_t \dot{H}^{\frac{3s}{2}} \cap L^2_t \dot{H}^{\frac{3s}{2}+1}$, respectively, for any $s\in [0,1)$. 

Therefore, the convergence of   velocity fields in the  endpoint space $L^\infty_t \dot{H}^1 \cap L^2_t \dot{H}^{2}$ will be achieved by an energy estimate in $\dot{H}^1$ and by making use of the stability results from the previous step. Here, it is important to mention that the validity of \eqref{E->0} in $L^2_t\dot{H}^\frac{1}{2}$ (see Proposition \ref{decay-E}) is crucial.

 On the other hand, the convergence of the magnetic fields in the endpoint space  $L^\infty_t \dot{H}^{\frac{3}{2}} \cap L^2_t \dot{H}^{\frac{3}{2}+1}$ will be established by a different approach, for a standard $\dot{H}^\frac{3}{2} $ energy estimate would require the validity of  \eqref{E->0}   in $L^2_t\dot{H}^\frac{3}{2}$  which is not available from Proposition \ref{decay-E}. More precisely, the   convergence of   magnetic fields in  that space  follows from extrapolation compactness techniques, introduced and utilized by the first and third authors in \cite{ah2}. 
 
 In summary, the  proof of convergence of magnetic fields in the energy space of $ \dot{H}^{\frac{3}{2}}$  is split into two main steps. In the first step, we treat the convergence of  frequencies that are localized in a ball whose radius grows as the speed of light increases. By suitably choosing that radius, the convergence of low frequencies then follows as a direct consequence of the convergence in the $L^2$ energy space.
 In the second step, by exploiting the assumption \eqref{assumption:high:frequencies}, we take care of the remaining high frequencies by building on the refined analysis of Maxwell equations laid out in Section \ref{section.EB:ES}.
The combination of these ideas eventually leads to the completion of all proofs.

\subsection{Notation}

All definitions and basic properties of  functional spaces  utilized throughout the paper are introduced in Section \ref{section:axi and para product}.

Furthermore, the letter $C$ will often denote a universal (possibly large) constant that is independent of the variables of the problem, and which is also allowed to change from one line to the next. 

 For simplicity, we will also be using 
$ A\lesssim B$ instead of $A\leq CB.$ Moreover, when needed, in order to specify the dependence of some estimates on some parameters, we will occasionally utilize
$ A\lesssim_s B$ to insist on the fact that the generic constant $C$ might depend on a parameter $s$.

\section{The axisymmetric structure, Hardy inequalities and paradifferential calculus   }
\label{section:axi and para product}

In this section, we establish several lemmas which shed light on crucial features of the axisymmetric structure, which are similar to the properties of the two-dimensional normal structure which is exploited in \cite{ah,ah2}. In particular, this structure will be  employed to obtain useful improvements on the classical paradifferential product laws and  will serve in our main a priori estimates, later on.

First of all, we  recall that  a vector field $F : \mathbb{R} ^3\to\mathbb{R}^3$ is  axisymmetric if it has the form
\begin{equation*}
F(x_1,x_2,x_3)= F_r(r,z)e_r+F_\theta(r,z)e_\theta+F_z(r,z)e_z,
\end{equation*}
where the triple $(r,\theta,z)$ denotes the usual cylindrical coordinates defined by the relations
\begin{equation*}
	x_1=r\cos\theta,\quad x_2=r\sin\theta,\quad x_3=z,
\end{equation*}
and $(e_r,e_\theta,e_z) $ is the corresponding cylindrical orthonormal basis
$$ e_r= \Big(\frac{x_h}{r},0\Big), \quad e_\theta= \Big(\frac{{x_h^\perp}}{r},0\Big), \quad e_\theta= (0,0,1).$$
Here, the index $h$ is used to refer to the horizontal components
$$x_h \bydef (x_1,x_2), \quad x_h^\perp \bydef (-x_2,x_1).$$
Thus, axisymmetry is characterized by the property that the components $(F_r,F_\theta,F_z)$ are independent of $\theta$. In other words, the field $F$ is axisymmetric if and only if it satisfies $R\circ F=F\circ R$, for all rotations $R$ around the $z$-axis.

Furthermore, we say that the vector field $F : \mathbb{R} ^3\to\mathbb{R}^3$ is  axisymmetric  without swirl if   it has the form
\begin{equation*}
F(x_1,x_2,x_3)= F_r(r,z)e_r+F_z(r,z)e_z,
\end{equation*}
and that it is axisymmetric with pure swirl if it can be represented as
\begin{equation*}
F(x_1,x_2,x_3)= F_\theta (r,z)e_\theta.
\end{equation*}
Observe that an axisymmetric field with pure swirl is always divergence free, which follows from a straightforward computation.

We take some time now to carefully introduce the functional spaces which we use in this article and some related notation. To that end, we first consider the Littlewood--Paley decomposition
\begin{equation*}
	\sum_{k\in\mathbb{Z}}\Delta_{ k}f=f
\end{equation*}
of a tempered distribution modulo polynomials $f\in\mathcal{S}'/\mathcal{P}(\mathbb{R}^d)$, in any dimension $d\geq 1$, where the operator $\Delta_k$ is the classical frequency truncation which restricts the support of the Fourier transform
\begin{equation*}
	\mathcal{F}f\left(\xi\right)=\hat f(\xi)\bydef\int_{\mathbb{R}^d} e^{- i \xi \cdot x} f(x) dx
\end{equation*}
to the set $\{2^{k-1}\leq |\xi|\leq 2^{k+1}\}$, for each $k\in\mathbb{Z}$. Recall that the space $\mathcal{S}'/\mathcal{P}(\mathbb{R}^d)$ is isomorphic to the space $\mathcal{S}_0'(\mathbb{R}^d)$ of tempered distribution restricted to the subspace $\mathcal{S}_0(\mathbb{R}^d)$, which is made up of all Schwartz functions $\varphi\in\mathcal{S}(\mathbb{R}^d)$ such that
\begin{equation*}
	\int_{\mathbb{R}^d}x^\alpha \varphi(x)dx=0,
\end{equation*}
for every mutli-index $\alpha\in\mathbb{N}^d$.  It is always possible to construct $\Delta_k$ so that it acts on tempered distributions through a convolution with a radial smooth function in $\mathcal{S}_0(\mathbb{R}^d)$, which is appropriately dilated by a factor $2^k$.

Then, for any $s \in \mathbb{R}$ and $1\leq p,q,r\leq \infty$, we define the homogeneous Besov space $\dot B^{s}_{p,q}\left(\mathbb{R}^d\right)$ and the homogeneous Chemin--Lerner space $\widetilde L^r \left( [0,T) ; B^{s}_{p,q}\left(\mathbb{R}^d\right) \right)$, with $T\in (0,\infty]$, as the subspaces of tempered distributions modulo polynomials $\mathcal{S}'/\mathcal{P}$ endowed with the respective norms
\begin{equation*}
	\begin{aligned}
		\left\|f\right\|_{\dot B^{s}_{p,q}\left(\mathbb{R}^d\right)}
		&\bydef
		\left(
		\sum_{k\in\mathbb{Z}} 2^{ksq}
		\left\|\Delta_{k}f\right\|_{L^p\left(\mathbb{R}^d\right)}^q\right)^\frac{1}{q},
		\\
		\left\|f\right\|_{ \widetilde L^r \left( [0,T) ; B^{s}_{p,q}\left(\mathbb{R}^d\right) \right)}
		&\bydef
		\left( \sum_{k\in\mathbb{Z}} 2^{ksq}
		\left\|\Delta_{k}f\right\|_{L^r\left([0,T);L^p\left(\mathbb{R}^d\right)\right) }^q\right)^\frac{1}{q},
	\end{aligned}
\end{equation*}
if $q<\infty$, and with the usual modifications if $q=\infty$.

We will also employ the homogeneous Sobolev spaces $\dot{W}^{s,p}\left(\mathbb{R}^d\right)\subset \mathcal{S}'/\mathcal{P}$ which are defined by the semi-norms
$$\norm {f }_{\dot{W}^{s,p}\left(\mathbb{R}^d\right)}  \bydef \big\||D|^s  f\big\|_{L^p\left(\mathbb{R}^d\right)}, $$
where $ |D|^s  $ is the Fourier multiplier corresponding to the symbol $ |\xi|^s$, with $s\in\mathbb{R}$ and $1<p<\infty$. When $p=q=2$, note that the Besov space $\dot B^s_{2,2}$ is equivalent to the homogeneous Sobolev space classically denoted by
\begin{equation*}
	\dot H^s\bydef\dot W^{s,2}.
\end{equation*}
Nonhomogeneous versions of these spaces are also defined in a similar way.

Finally, employing truncated Littlewood--Paley decompositions, one can show that the subspace $\mathcal{S}_0$ is dense in $\dot B^s_{p,q}$ and $\dot W^{s,p}$, for any $s\in \mathbb{R}$ and $1\leq p,q<\infty$, and that a similar statement holds for Chemin--Lerner spaces with suitable modifications.

We refer to \cite[Appendix A]{ah} for some more details and properties of Besov and Chemin--Lerner spaces in the same notation as in this article, and to \cite{bcd11,g14:2} for a comprehensive introduction to the subject of Littlewood--Paley decompositions and functional spaces.

We move on now to the main results of this section. Prior to presenting the paradifferential features of the axisymmetric structure, we establish the following version of Hardy's inequality which will be useful in our analysis, below.

\begin{lem}\label{lemma-f/x}
Fix the dimension $d\geq 2$.
Let $f$ be a smooth function satisfying, for all $ x' \in \mathbb{R}^{d-1}$, that
$$f(0,x')=0.$$ 
Then, it holds that 
\begin{equation*}
\norm {\frac{f}{x_1} }_{\dot W^{s,p}(\mathbb{R}^d)} \lesssim _{s,p}   \norm {f}_{\dot W^{s+1,p}(\mathbb{R}^d)},
\end{equation*}
for any $ p\in (1,\infty)$ and all $s> \frac{1}{p}-1$.
\end{lem}

\begin{proof}
By virtue of   the vanishing assumption on $f$, we can write, for any $(x_1,x') \in \mathbb{R}^*\times \mathbb{R}^{d-1}$,  that
\begin{equation*}
\frac{f(x_1,x')}{x_1}= \int_0^1 (\partial_{1} f)(\lambda x_1,x')d\lambda  .
\end{equation*}
It then follows that
\begin{equation*}
	\begin{aligned}
		\norm {\frac{f}{x_1} }_{\dot W^{s,p}} &\leq \int_0^1 \big\||\partial_1 f(\lambda \cdot, \cdot)\big\|_{\dot W^{s,p}}d\lambda
		\\
		&=\int_0^1 \big\| |D|^s  (\partial_{1} f)(\lambda \cdot, \cdot) \big\|_{L^p}  d\lambda\\
		&=   \int_0^1 \lambda^{-\frac{1}{p}} \big\| |D|^s  m_{\lambda,s}(D)      (\partial_{1} f)  \big\|_{L^p}  d\lambda\\
		&=  \int_0^1 \lambda^{-\frac{1}{p}} \big\|   m_{\lambda,s}(D)      \partial_{1} f  \big\|_{ \dot{W}^{s,p}}  d\lambda,
	\end{aligned}
\end{equation*}
where $m_{\lambda,s}(D) $ is the Fourier multiplier operator given, for any $\lambda \in (0,1)$ and $s\in\mathbb{R}$,  by
$$  m_{\lambda,s} (\xi)= \left(\frac{   \sqrt{|\lambda \xi_1 |^2 + |\xi'|^2} }{|\xi|}\right)^s.$$
Now, we claim that the multiplier norm of $m_{\lambda,s}(D)$ over $L^p$ is bounded by a constant multiple of $\max\{1,\lambda^s\}$, which leads to
\begin{equation*}
\begin{aligned}
\norm {\frac{f}{x_1} }_{\dot{W}^{s,p}}
&\lesssim    \int_0^1 \lambda^{-\frac{1}{p}}  \max\{ 1, \lambda ^s\}d\lambda
\norm  {    \partial_{1} f  }_{ \dot{W}^{s,p}}
\\
&\lesssim
\max \left\{ \frac{p}{p-1} , \frac{1}{s+1 - \frac{1}{p}} \right\}     \norm  {     f  }_{ \dot{W}^{s+1,p}} ,
\end{aligned}
\end{equation*}
provided that $ s> \frac{1}{p} -1$, thereby establishing the main estimate of the lemma.

In order to justify the boundedness of $m_{\lambda,s}(D)$ over $L^p$, it is sufficient, by the Marcinkiewicz--Mikhlin multiplier theorem (see \cite[Theorem 6.2.4]{g14}), to show that
\begin{equation}\label{multiplier:estimate}
	\sup_{\xi\in\mathbb{R}^d}\left|\xi^\alpha\partial_\xi^\alpha m_{\lambda,s}(\xi)\right|
	\lesssim \max\{ 1, \lambda ^s\},
\end{equation}
for all multi-indices $\alpha\in \{0,1\}^d$. To that end, we introduce the notation
\begin{equation*}
	\ell_{i,\lambda} (\xi) \bydef  \frac{\xi_i^2 }{  \lambda^2| \xi_1 |^2 + |\xi'|^2 }
\end{equation*}
and compute that
\begin{equation*}
	\begin{aligned}
		\xi_1\partial_{\xi_1} m_{\lambda,s} (\xi) &= s \ell_{1,1}(\xi)
		\Big( \lambda^2  m_{\lambda,s-2} (\xi) -  m_{\lambda,s} (\xi) \Big),
		\\
		\xi_i\partial_{\xi_i} m_{\lambda,s} (\xi) &= s  \left(\ell_{i,\lambda} (\xi) - \ell_{i,1} (\xi) \right)   m_{\lambda,s} (\xi) ,
		\\
		\xi_1\partial_{\xi_1}\ell_{j,\lambda}(\xi)&= 2 \left( \delta_{1j}    - \lambda^2 \ell_{1,\lambda}(\xi) \right) \ell_{j,\lambda}(\xi),
		\\
		\xi_i\partial_{\xi_i}\ell_{j,\lambda}(\xi)&= 2 \left( \delta_{ij}    - \ell_{i,\lambda}(\xi) \right) \ell_{j,\lambda}(\xi),
	\end{aligned}
\end{equation*} 
for any integers $i\in [2,d]$ and $j\in [1,d]$, where $\delta_{ij}$ denotes the usual Kronecker delta. Therefore, iterating the preceding calculations and observing that
 \begin{equation*}
 	\left|m_{\lambda,s} (\xi) \right|
	+\left|\lambda^{2}  m_{\lambda,s-2} (\xi) \right| \lesssim \max\{1, \lambda^s \},
	\qquad
	\left|\lambda^2\ell_{1,\lambda}\right|+\left|\ell_{i,\lambda}(\xi)\right|\lesssim 1,
\end{equation*}
for all integers $i\in [2,d]$, it is readily seen that \eqref{multiplier:estimate} holds true, which establishes that the operator norm of $m_{\lambda,s}(D)$ is controlled by $\max\{1,\lambda^s\}$. This completes the proof of the lemma.
\end{proof}

\begin{rem}
Note that the assumption on the smoothness of the function $f$ in Lemma \ref{lemma-f/x} can in fact be relaxed.
Indeed, as described, for instance, in Theorem 6.6.1 from \cite {bl76}, the trace operator
$$ \mathrm{Tr} : \mathcal{S}_0(\mathbb{R}^d) \rightarrow \mathcal{S}_0(\mathbb{R}^{d-1}) $$
defined by
$$ (\operatorname{Tr}f )(x_1,x')\bydef f(0,x') ,\quad\text{where } (x_1,x')\in \mathbb{R}\times \mathbb{R}^{d-1},$$
has a well defined extension into a bounded operator from $ W^{s+1,p}(\mathbb{R}^d)$  into  $ B^{s +1- \frac{1}{p}}_{p,p}(\mathbb{R}^{d-1})$, for any $p\in (1,\infty)$, as soon as  the condition $ s> \frac{1}{p}-1$ is satisfied.
Hence, for $f\in W^{s+1,p}(\mathbb{R}^d)$,  the condition $f(0,x')\equiv 0 $ in Lemma \ref{lemma-f/x} can be replaced by $ \operatorname{Tr}f  = 0. $
\end{rem}

The next lemma is a variant of  Lemma \ref{lemma-f/x} for vector fields with an axisymmetric structure. It will be employed to control specific quantities involving  electromagnetic fields, later on.  The statement of the lemma below is written in the notation introduced at the beginning of this section.

\begin{lem}\label{lemma-f/x:2}
Let $E$ be an axisymmetric divergence-free vector field without swirl and $B$  be an axisymmetric   vector field with pure swirl.

Then, it holds that
\begin{equation*}
\norm {\frac{\nabla \times E}{r}}_{L^p} \lesssim  \norm E_{\dot W^{2,p}}
\end{equation*}
and
\begin{equation*}
\norm {\frac{B_\theta}{r}}_{\dot W^{s,p}} \lesssim_s \norm B_{\dot W^{s+1,p}},
\end{equation*}
for any $ p\in (1,\infty)$ and all $s> \frac{1}{p}-1$.
\end{lem}

\begin{rem}
	For any given axisymmetric vector field $E=E_re_r+E_ze_z$ with no swirl, a direct computation gives that the curl $\nabla\times E=(\partial_z E_r-\partial_r E_z)e_\theta$ is axisymmetric with pure swirl. Similarly, given an axisymmetric vector field $B=B_\theta e_\theta$ with pure swirl, another straightforward computation gives that the curl $\nabla\times B=-\partial_z B_\theta e_r+(\partial_rB_\theta+\frac 1r B_\theta)e_z$ is axisymmetric with no swirl.
\end{rem}

\begin{proof}
The bound on $B_\theta$ is a consequence of  Lemma \ref{lemma-f/x}. To see this, let us first assume that $B$ is smooth and has pure swirl. Then, we can write that
\begin{equation*}
	(B_1,B_2,0)=B_\theta e_\theta=B_\theta \left(\frac{-x_2}r,\frac{x_1}r,0\right)
\end{equation*}
to deduce that
\begin{equation*}
	B_1|_{x_2=0}\equiv 0, \qquad B_2|_{x_1=0}\equiv 0,
\end{equation*}
and
\begin{equation*}
\frac{B_\theta}{r}=-\frac{B_1}{x_2}= \frac{B_2}{x_1}, 
\end{equation*}
as soon as $r\neq 0.$
Therefore,  it is enough to estimate $\frac{B_1}{x_2}$ or $\frac{B_2}{x_1}$.

Then, an application of Lemma \ref{lemma-f/x} yields that
\begin{equation*}
\norm {\frac{B_\theta}{r}}_{\dot W^{s,p}} = \norm {\frac{B_2}{x_1}}_{\dot W^{s,p}} \lesssim _s \norm{B_2}_{\dot W^{s+1,p}},
\end{equation*} 
thereby establishing the desired bound on $B_\theta$ in the case of a smooth vector field. The general nonsmooth case is then obtained by a standard approximation argument.

As for the bound on $E$, it will follow from a remarkable identity valid for axisymmetric divergence-free vector fields with no swirl. Indeed, the fact that $E$ is axisymmetric without swirl allows us to write that
\begin{equation*}
	\begin{aligned}
		\nabla\times E&=\left(\partial_zE_r-\partial_rE_z\right)e_\theta,
		\\
		e_r\cdot\nabla\left(\nabla\times E\right)&=\left(\partial_r\partial_z E_r-\partial_r^2E_z\right)e_\theta,
	\end{aligned}
\end{equation*}
which, when combined with the divergence-free condition for axisymmetric fields
\begin{equation*}
	\partial_rE_r+\frac 1r E_r+\partial_zE_z=0
\end{equation*}
leads to
\begin{equation*}
	\begin{aligned}
		\frac{\nabla\times E}r+e_r\cdot\nabla\left(\nabla\times E\right)
		&=\left(\frac 1r \partial_zE_r-\frac 1r\partial_rE_z+\partial_r\partial_z E_r-\partial_r^2E_z\right)e_\theta
		\\
		&=-\left(\partial_r^2E_z+\frac 1r\partial_rE_z+\partial_z^2E_z\right)e_\theta.
	\end{aligned}
\end{equation*}
Then, identifying the action of the Laplacian on axisymmetric functions to deduce that
\begin{equation*}
	\Delta E_z=\partial_r^2E_z+\frac 1r\partial_r E_z + \partial_z^2 E_z,
\end{equation*}
we arrive at the expression
\begin{equation*}
	\frac{\nabla\times E}r=-e_r\cdot\nabla\left(\nabla\times E\right)-(\Delta E_z)e_\theta.
\end{equation*}
The bound on $E$ therefore follows from a direct estimate in $L^p$ on the preceding identity, which completes the proof of the lemma.
\end{proof}

We conclude this section with a lemma that extends the range of parameters in the classical paraproduct laws by exploiting a geometric condition which is satisfied by the axisymmetric structure. These extended paradifferential estimates will be employed to obtain important a priori estimates for the Navier--Stokes--Maxwell system \eqref{Maxwell:system:*2}, later on.

\begin{lem}\label{paradifferential:1}
	Let $F,G:[0,T)\times\mathbb{R} ^3\to\mathbb{R}^3$ be such that $\div F=0$ and
	\begin{equation}\label{geometric:condition}
		\int_{\mathbb{R}^3}\varphi(x-y)\nabla\times F(t,y)dy\quad\text{and}\quad \int_{\mathbb{R}^3}\psi(x-y)G(t,y)dy
		\quad\text{are colinear,}
	\end{equation}
	for all $t\in[0,T)$, $x\in\mathbb{R}^3$, and any radially symmetric $\varphi,\psi \in\mathcal{S}_0(\mathbb{R}^3)$ (i.e., such that $\varphi(x)$ and $\psi(x)$ only depend on $|x|$).
	Further consider parameters in $[1,\infty]$ such that
	\begin{equation*}
		\frac 1a=\frac 1{a_1}+\frac 1{a_2}
		\qquad\text{and}\qquad
		\frac 1c=\frac 1{c_1}+\frac 1{c_2}.
	\end{equation*}
	Then, recalling that $P=(-\Delta)^{-1}\curl\mathrm{curl}$ denotes Leray's projector onto solenoidal vector fields, one has the product estimate
	\begin{equation} \label{paraproduct:1}
		\norm{P(F\times G)}_{\widetilde L^a([0,T);\dot B^{s+\eta-\frac{3}{2} }_{2,c} )}
		\lesssim
		\norm{F}_{\widetilde L^{a_1}([0,T);\dot B^s_{2,c_1} )}
		\norm{G}_{\widetilde L^{a_2}([0,T);\dot B^\eta_{2,c_2} )},
	\end{equation}
	for any $s\in(-\infty,\frac{3}{2} )$ and $\eta\in (-\infty,\frac{5}{2} )$ with $s+\eta>0$.
	
	Furthermore, in the endpoint case $s= \frac{3}{2}$, one has that
	\begin{equation}\label{paraproduct:2}
		\norm{P(F\times G)}_{\widetilde L^a([0,T); \dot B^{\eta}_{2,c} )}
		\lesssim
		\norm{F}_{L^{a_1}([0,T); L^\infty  ) \cap \widetilde L^{a_1}([0,T);\dot B^{\frac{3}{2}} _{2,\infty} )}
		\norm{G}_{\widetilde L^{a_2}([0,T)\dot B^\eta_{2,c} )},
	\end{equation}
	for any $\eta\in(-\frac{3}{2} , \frac{5}{2})$.
	
	 Finally, in the case $\eta=\frac{5}{2}$, it holds that
	\begin{equation} \label{paraproduct:3}
		\norm{P(F\times G)}_{\widetilde L^a([0,T); \dot B^{s+1}_{2,c}  )}
		\lesssim
		\norm{F}_{  \widetilde L^{a_1}([0,T);\dot B^{s}_{2,c}    )}
		\norm{G}_{\widetilde L^{a_2}([0,T);\dot B^ {\frac{5}{2}}_{2,1} )},
	\end{equation}
	as soon as $ s\in (-\frac{5}{2}, \frac{3}{2})$, and the case $ s = \frac{3}{2} $ is allowed provided that $c=1$.
\end{lem}

\begin{rem}
	The preceding result also holds for vector fields which are independent of time in classical Besov spaces (without any norm in time). Accordingly, the lemma also holds in the case of Besov-space-valued Lebesgue spaces. This means that removing the tildes in \eqref{paraproduct:1}, \eqref{paraproduct:2} and \eqref{paraproduct:3} produces valid estimates.
\end{rem}

\begin{rem}
	Consider an axisymmetric field $H:\mathbb{R}^3\to\mathbb{R}^3$. For any radially symmetric $\varphi\in\mathcal{S}_0$ and any rotation $R$ around the $z$-axis, we find that
	\begin{equation*}
		\varphi *H(Rx)
		=\int_{\mathbb{R}^3}\varphi(R(x-y))H(Ry)dy
		=\int_{\mathbb{R}^3}\varphi(x-y)RH(y)dy=R\varphi *H(x),
	\end{equation*}
	thereby showing that $\varphi* H$ is axisymmetric, too. If, furthermore, the field $H$ has pure swirl, then, employing that
	$S_x(y)\bydef y-2(y\cdot e_\theta(x))e_\theta(x)$, for any $x\neq 0$, is an isometry and satisfies that
	\begin{equation*}
		\frac 12\big(e_\theta(y)+e_\theta(S_x y)\big)=e_\theta(y)-\big(e_\theta(y)\cdot e_r(x)\big)e_r(x)=\big(e_\theta(y)\cdot e_\theta(x)\big)e_\theta (x),
	\end{equation*}
	we compute that
	\begin{equation*}
		\begin{aligned}
			\int_{\mathbb{R}^3}\varphi(x-y)H_\theta(y)e_\theta(y)dy
			&=\frac 12\int_{\mathbb{R}^3}\big(\varphi(x-y)H_\theta(y)e_\theta(y)+\varphi(x-S_xy)H_\theta(S_xy)e_\theta(S_xy)\big)dy
			\\
			&=\frac 12\int_{\mathbb{R}^3}\varphi(x-y)H_\theta(y)\big(e_\theta(y)+e_\theta(S_xy)\big)dy
			\\
			&=\left(\int_{\mathbb{R}^3}\varphi(x-y)H_\theta(y)\big(e_\theta(y)\cdot e_\theta(x)\big)dy\right) e_\theta(x),
		\end{aligned}
	\end{equation*}
	which establishes that $\varphi*H$ has pure swirl, as well.
\end{rem}

\begin{rem}
	The hypothesis \eqref{geometric:condition} is satisfied by axisymmetric divergence-free vector fields such that $F$ has no swirl and $G$ has pure swirl. Indeed, as previously emphasized, the curl of an axisymmetric vector field with no swirl is axisymmetric with pure swirl. Therefore, in this situation, it holds that $\nabla\times F$ and $G$ both have pure swirls, which implies, according to the preceding remark, that their convolutions with radial functions remain axisymmetric with pure swirl and are thus colinear.
\end{rem}

\begin{rem}\label{classical time-space:ES}
	Note  that  Lemma \ref{paradifferential:1} above is an extension of \cite[Lemma 3.4]{ah} to the three-dimensional setting. In particular, its significance  lies in the fact that it allows us to cover the range of parameters  $\eta\in[\frac{3}{2} ,\frac{5}{2})$. Indeed, without the geometric assumption \eqref{geometric:condition} on $F$ and $G$, the paradifferential estimates remain valid but may need to be restricted to parameters satisfying $\eta<\frac 32$.
\end{rem}

\begin{proof}
	We follow the method of proof of Lemma 3.4 from \cite{ah} and write Bony's decomposition
	\begin{equation*}
		F\times G=T_FG-T_GF+R(F,G),
	\end{equation*}
	where the paraproducts are defined by
	\begin{equation*}
		T_FG=\sum_{\substack{j,k\in\mathbb{Z}\\j-k<-2}}\Delta_jF\times\Delta_kG,
		\qquad
		T_GF
		=\sum_{\substack{j,k\in\mathbb{Z}\\j-k<-2}}\Delta_jG\times\Delta_kF
		=-\sum_{\substack{j,k\in\mathbb{Z}\\j-k>2}}\Delta_jF\times\Delta_kG,
	\end{equation*}
	and the remainder is given by
	\begin{equation*}
		R(F,G)=\sum_{\substack{j,k\in\mathbb{Z}\\|j-k|\leq 2}}\Delta_jF\times\Delta_kG,
	\end{equation*}
	to deduce that a direct application of classical paraproduct estimates on homogeneous Besov spaces (see \cite[Appendix A]{ah} or \cite{bcd11}, for instance), combined with the fact that $P$ is bounded over Besov spaces, yields the validity of \eqref{paraproduct:1} for parameters $s\in(-\infty,\frac{3}{2} )$ and $\eta\in (-\infty,\frac{3}{2} )$ with $s+\eta>0$, and the validity of \eqref{paraproduct:2} for parameters $\eta\in (-\frac 32,\frac{3}{2} )$.
	
	It is important to emphasize here that the restriction $\eta<\frac 32$ comes solely from the estimate of $T_GF$. Thus, in order to establish the validity of \eqref{paraproduct:1} and \eqref{paraproduct:2} for the full range of parameters, we only need to show now that
	\begin{equation} \label{paraproduct:7}
		\norm{P(T_GF)}_{\widetilde L^a([0,T);\dot B^{s+\eta-\frac{3}{2} }_{2,c} )}
		\lesssim
		\norm{F}_{\widetilde L^{a_1}([0,T);\dot B^s_{2,c_1} )}
		\norm{G}_{\widetilde L^{a_2}([0,T);\dot B^\eta_{2,c_2} )},
	\end{equation}
	for any $s\in\mathbb{R}$ and $\eta\in (-\infty,\frac{5}{2} )$, as a consequence of the divergence-free structure of $F$ and the geometric assumption \eqref{geometric:condition}.
	
	To that end, assuming first that $F$ and $G$ are smooth, we compute that
	\begin{equation*}
		\nabla\times(F\times G) = \nabla(F\cdot G)-G \times (\nabla \times F)-F\times( \nabla\times G) -2F\cdot \nabla G + F\div G -G\div F.
	\end{equation*}
	Thus, if $\nabla\times F$ and $G$ are colinear and $F$ is divergence free, we arrive at
	\begin{equation}\label{remarkable:paraproduct:0}
		P(F\times G) = (- \Delta)^{-1}\nabla \times \big(
		-F\times( \nabla\times G) -2F\cdot \nabla G + F\div G
		\big).
	\end{equation}
	Now, applying the same reasoning to $\Delta_k F$ and $\Delta_j G$, instead of $F$ and $G$, and utilizing \eqref{geometric:condition} to deduce that $\nabla\times \Delta_k F$ and $\Delta_j G$ are colinear, we obtain that
	\begin{equation}\label{remarkable:paraproduct}
		P(T_GF)
		=(- \Delta)^{-1}\nabla \times
		\Bigg(
		\sum_{\substack{j,k\in\mathbb{Z}\\j-k<-2}}
		\Big(\Delta_k F\times( \nabla\times \Delta_j G) +2\Delta_k F\cdot \nabla \Delta_j G - \Delta_k F\div \Delta_j G\Big)
		\Bigg).
	\end{equation}
	Therefore, applying classical paraproduct estimates (see \cite[Appendix A]{ah} or \cite{bcd11}), we conclude, for any $s\in\mathbb{R}$ and $\eta<\frac 52$, that
	\begin{equation*}
		\begin{aligned}
			&\norm{P(T_GF)}_{\widetilde L^a\dot B^{s+\eta-\frac{3}{2} }_{2,c}}
			\\
			&\qquad\lesssim
			\Bigg\|
			\sum_{\substack{j,k\in\mathbb{Z}\\j-k<-2}}
			\Big(\Delta_k F\times( \nabla\times \Delta_j G) +2\Delta_k F\cdot \nabla \Delta_j G - \Delta_k F\div \Delta_j G\Big)
			\Bigg\|_{\widetilde L^a\dot B^{s+\eta-\frac{5}{2} }_{2,c} }
			\\
			&\qquad\lesssim
			\norm{F}_{\widetilde L^{a_1}\dot B^s_{2,c_1} }
			\norm{\nabla G}_{\widetilde L^{a_2}\dot B^{\eta-\frac 52}_{\infty,c_2} }
			\lesssim
			\norm{F}_{\widetilde L^{a_1}\dot B^s_{2,c_1} }
			\norm{\nabla G}_{\widetilde L^{a_2}\dot B^{\eta-1}_{2,c_2} }
			\\
			&\qquad\lesssim
			\norm{F}_{\widetilde L^{a_1}\dot B^s_{2,c_1} }
			\norm{G}_{\widetilde L^{a_2}\dot B^{\eta}_{2,c_2} },
		\end{aligned}
	\end{equation*}
	which establishes \eqref{paraproduct:7}, thereby completing the proof of \eqref{paraproduct:1} and \eqref{paraproduct:2}.
	
	The justification of \eqref{paraproduct:3} is similar. Indeed, the classical paradifferential estimates apply directly to the paraproduct $T_FG$ and the remainder $R(F,G)$ in the range of parameters desribed in \eqref{paraproduct:3}. Thus, we see that \eqref{paraproduct:3} will follow from the justification of the paraproduct estimate
	\begin{equation} \label{paraproduct:8}
		\norm{P(T_GF)}_{\widetilde L^a([0,T);\dot B^{s+1 }_{2,c} )}
		\lesssim
		\norm{F}_{\widetilde L^{a_1}([0,T);\dot B^s_{2,c} )}
		\norm{G}_{\widetilde L^{a_2}([0,T);\dot B^{\frac 52}_{2,1} )},
	\end{equation}
	for any $s\in\mathbb{R}$. As before, in order to prove \eqref{paraproduct:8}, we apply classical paraproduct estimates to \eqref{remarkable:paraproduct}. This leads to
	\begin{equation*}
		\begin{aligned}
			\norm{P(T_GF)}_{\widetilde L^a\dot B^{s+1 }_{2,c} }\hspace{-20mm}&
			\\
			&\lesssim
			\Bigg\|
			\sum_{\substack{j,k\in\mathbb{Z}\\j-k<-2}}
			\Big(\Delta_k F\times( \nabla\times \Delta_j G) +2\Delta_k F\cdot \nabla \Delta_j G - \Delta_k F\div \Delta_j G\Big)
			\Bigg\|_{\widetilde L^a\dot B^{s }_{2,c} }
			\\
			&\lesssim
			\norm{F}_{\widetilde L^{a_1}\dot B^s_{2,c} }
			\norm{\nabla G}_{\widetilde L^{a_2}L^\infty }
			\lesssim
			\norm{F}_{\widetilde L^{a_1}\dot B^s_{2,c} }
			\norm{\nabla G}_{\widetilde L^{a_2}\dot B^{0}_{\infty,1} }
			\\
			&\lesssim
			\norm{F}_{\widetilde L^{a_1}\dot B^s_{2,c} }
			\norm{G}_{\widetilde L^{a_2}\dot B^{1}_{\infty,1} }
			\lesssim
			\norm{F}_{\widetilde L^{a_1}\dot B^s_{2,c} }
			\norm{G}_{\widetilde L^{a_2}\dot B^{\frac 52}_{2,1} },
		\end{aligned}
	\end{equation*}
	for all $s\in\mathbb{R}$, which establishes the validity of \eqref{paraproduct:3} and concludes of the proof of the lemma.
\end{proof}

\begin{rem}
	If, instead of \eqref{geometric:condition}, one merely assumes that $\nabla\times F$ and $G$ are colinear, in the sense that $G\times(\nabla\times F)=0$, then the paradifferential estimate \eqref{paraproduct:1} remains valid for any parameters $s\in(-\infty,\frac{3}{2} )$ and $\eta\in (-\infty,\frac{5}{2} )$ in the range $s+\eta>1$. This follows from applying Bony's decomposition to each term of \eqref{remarkable:paraproduct:0} and then estimating the resulting paraproducts and remainders as in the proof above. The restriction $s+\eta>1$ is then a consequence of the estimates of the remainders.
	Similarly, under the assumption that $\nabla\times F$ and $G$ are colinear, one can show that \eqref{paraproduct:2} and \eqref{paraproduct:3} remain valid in the respective ranges $\eta\in (-\frac 12, \frac 52)$ and $s\in (-\frac 32,\frac 32)$.
\end{rem}

\section{A priori estimates}\label{a priori ES}

Here, we establish    a priori estimates on smooth solutions of the Navier--Stokes--Maxwell system \eqref{Maxwell:system:*2}. The ensuing bounds will be employed to prove existence of global solutions to that system. For simplicity, from now on, we take $\nu= 1$ in \eqref{Maxwell:system:*2}. However, we emphasize that all estimates below hold for any $\nu >0$.
 
 We  recall first that   the only available global bound for smooth solutions of \eqref{Maxwell:system:*2}   corresponds to the $L^2$-energy estimate
\begin{equation}\label{L^2-Energy}
\norm {u }_{L^\infty_tL^2\cap L^2_t\dot{H}^1} + \norm {(E ,B )}_{L^\infty_tL^2} + \norm {j }_{L^2_tL^2} \lesssim \mathcal{E}_0 \bydef \norm {(u _0,E _0,B _0)}_{L^2} .
\end{equation}
As explained in the introduction, this bound does not seem to be   enough to construct any kind of solutions. Therefore,  we shall aim to propagate some higher regularity for $u$, $E$ and $B$.

In what follows, we recall that are using the notation
$$ \Omega \bydef \frac{\omega_\theta}{r} \qquad \text{and}  \qquad \Gamma \bydef \frac{B_\theta}{r},$$
where
$$ \omega_\theta =\omega \cdot e_\theta= \left( \nabla \times  u \right)\cdot e_\theta \qquad \text{and}  \qquad   B_\theta =B \cdot e_\theta.$$

\subsection{Controlling the velocity field}
In order to control higher regularities for the velocity field, we shall exploit the axisymmetric structure  and perform an energy estimate on the equations describing the evolution of $\omega_\theta$ and $\Omega$, which we derived in Section \ref{axi-subsection}. For convenience, we recall here that the equation for $\omega_\theta$ can be written as
 \begin{equation}\label{omega-equa-form2:B}
  \partial_t \omega_\theta + u\cdot\nabla \omega_{\theta} -\big(\Delta - \frac{1}{r^2} \big)\omega_\theta= \frac{u_r}{r} \omega_\theta -\partial_z \big( \Gamma B_\theta\big)-  \frac{1}{c}\partial_t E _r\frac{B_\theta}{r}  + \frac{1}{c}\partial_t E   \cdot \nabla B_\theta,
  \end{equation} 
  while the equation for $\Omega$ reads as
    \begin{equation}\label{Omega-equa-form2:B}
  \partial_t \Omega  + u\cdot\nabla \Omega -\big(\Delta+ \frac{\partial_r}{r} \big)\Omega =  -\partial_z \big( \Gamma^2\big)  -  \frac{1}{c}\partial_t E \cdot \nabla \Gamma .
  \end{equation}
  The following proposition provides a control on the velocity field in terms of some suitable norms of electromagnetic fields. This will be combined with the estimates from  Section \ref{section.EB:ES} to obtain uniform global bounds in Section \ref{section:closing:ES}, later on.

 \begin{prop}\label{prop.velocity-ES}
 Let $T\in (0,\infty]$ and  $ (u,E,B)$   be a smooth axisymmetric solution of \eqref{Maxwell:system:*2} 
 on $[0,T)$, where $u$ and $E$ have no swirl and $B$ has pure swirl. Then, there is a universal constant $C>0$  such that 
 \begin{equation*}
	 \begin{aligned}
	    &\norm{(\omega,\Omega)}_{L^\infty_tL^2 \cap L^2_t\dot{H}^1 }
	     \\
		 &\qquad\lesssim   \left( \norm{(\omega_0,\Omega_0)}_{L^2}   +  \norm { \frac{1}{c}\partial_t E}_{L^2_t\dot{H}^\frac{1}{2}  }   \norm {B}_{L^\infty_tH^2   }     + \norm {\Gamma}_{L^\infty_t L^3 }\norm {(B,\Gamma)}_{L^2_t \dot{H}^1 } \right)
	     \exp \big(C\mathcal{E}^2_0\big),
	 \end{aligned}
 \end{equation*} 
 where all the norms are taken over the time interval $[0,T)$.
 \end{prop}

 \begin{rem}Note that the right-hand side in the bound above does not exhibit any time growth, provided that the norms of the electromagnetic fields remain bounded on any time interval $[0,T)$. This is crucial and necessary to prove the global existence of solutions in Theorem \ref{Thm:1}.
 \end{rem}

 \begin{proof}
 Firstly, we observe that a straightforward computation relying on the fact $\omega$ is axisymmetric without swirl yields that
 \begin{equation*}
 	e_r\cdot\nabla\omega=(\partial_r\omega_\theta)e_\theta,
 	\qquad
 	e_\theta\cdot\nabla\omega=-\frac{\omega_\theta}{r}e_r=-\Omega e_r,
 	\qquad
 	e_z\cdot\nabla\omega=(\partial_z\omega_\theta)e_\theta.
 \end{equation*}
In particular, this implies that
 \begin{equation}\label{ID:00}
 	|\omega| = |\omega_\theta|
	\qquad\text{and}\qquad
	|\nabla \omega| \sim \left|(\partial_r \omega_\theta,\partial_z \omega_\theta)\right| + \left| \Omega\right|,
 \end{equation}
which will allow us to estimate $(\omega_\theta,\Omega)$ instead of $(\omega,\Omega)$.

 Now, multiplying \eqref{omega-equa-form2:B} by $\omega_\theta$, integrating in time and space, and then using  the divergence-free condition on $E$ and $u$ yields, for all $t\in [0,T)$, that
 \begin{equation}\label{EQ0}
 \begin{aligned}
 \frac{1}{2} \norm{\omega_\theta(t)}_{  L^2 }^2 & +  \norm{\nabla \omega_\theta}_{L^2_t L^2 }^2 +\norm{\frac{\omega_\theta}{r}}_{L^2_t L^2 }^2
 \\
 &\leq   \frac{1}{2} \norm{\omega_0}_{L^2}^2  + \int_{0}^t\norm{\frac{u_r(\tau)}{r}}_{L^\infty} \norm {\omega_\theta(\tau)}_{L^2}^2d\tau
 +\left| \int_{0}^t \int_{\mathbb{R}^3} \Gamma(\tau) B(\tau) \partial_z \omega_\theta(\tau)  dxd\tau \right|
\\
&\quad+\int_{0}^t \frac{1}{c}\norm {(\partial_t E)B(\tau) }_{L^2} \Big( \norm {\nabla \omega_\theta(\tau)}_{L^2} +  \norm {\frac{\omega_\theta}{r}(\tau)}_{L^2}\Big) d\tau,  
 \end{aligned}
\end{equation}
where the norms in time are taken over the interval $[0,t)$.

In order to estimate the second term in the right-hand side above, we make use of \eqref{ur/r-interpolation} to obtain, for any $\varepsilon>0$, that
\begin{equation}\label{EQ1}
	\begin{aligned}
		 \int_{0}^t\norm{\frac{u_r(\tau)}{r}}_{L^\infty}
		 &\norm {\omega_\theta(\tau)}_{L^2}^2d\tau
		 \\
		 & \lesssim    \int_{0}^t\norm {\Omega(\tau)}_{L^2}^\frac{1}{2}  \norm {\nabla \Omega(\tau)}_{L^2}^\frac{1}{2}\norm {\omega_\theta(\tau)}_{L^2}^2d\tau
		 \\
         & \leq    \varepsilon  \int_{0}^t\norm {\Omega(\tau)}_{L^2}  \norm {\nabla \Omega(\tau)}_{L^2} d\tau+  C_\varepsilon\int_{0}^t\norm {\omega_\theta(\tau)}_{L^2}^4d\tau\\
         & \leq   \varepsilon \Big(\norm {\frac{\omega_{\theta}}{r}}_{L^2_t L^2 }^2  +\norm { \Omega}_{L^2_t \dot{H}^1 } ^2\Big)+  C_\varepsilon\int_{0}^t\norm {u(\tau)}_{\dot{H}^1}^2\norm {\omega_\theta(\tau)}_{L^2}^2d\tau,  
	\end{aligned}
\end{equation} 
where we utilized the celebrated Biot--Savart estimate in the last line.

On the other hand, the estimate of the third term in the right-hand side of \eqref{EQ0} is obtained by employing   H\"older's inequality, again, followed by the Sobolev embedding $\dot{H}^1\hookrightarrow L^6(\mathbb{R}^3)$ to write, for any $\varepsilon>0$, that 
 \begin{equation} \label{EQ2}
 \begin{aligned}
\left| \int_{0}^t \int_{\mathbb{R}^3} \Gamma(\tau) B(\tau) \partial_z \omega_\theta(\tau)  dxd\tau \right|  & \leq    C_\varepsilon \int_{0}^t \norm {\Gamma B(\tau)}_{L^2}^2 d\tau +\varepsilon \norm {\partial_z \omega_\theta}_{L^2_t L^2 }^2   \\
 & \leq   C_\varepsilon \int_{0}^t \norm{ \Gamma(\tau)}_{L^3}^2\norm {B(\tau)}_{L^6}^2 d\tau + \varepsilon \norm {\nabla \omega_\theta}_{L^2_t L^2 }^2 \\
 & \leq   C_\varepsilon  \norm{ \Gamma }_{L^\infty_t L^3 }^2 \norm {B}_{L^2_t \dot{H}^1 }^2 + \varepsilon \norm {\nabla \omega_\theta}_{L^2_t L^2 }^2 . 
 \end{aligned}
\end{equation} 
Finally, the last integral in \eqref{EQ0} can be easily controlled by   similar arguments to obtain that
\begin{equation}
\begin{aligned} \label{EQ3}
\int_{0}^t \frac{1}{c}\norm {(\partial_t E)B(\tau) }_{L^2}
&\Big( \norm {\nabla \omega_\theta(\tau)}_{L^2} +  \norm {\frac{\omega_\theta}{r}(\tau)}_{L^2}\Big) d\tau
\\
&\leq     \frac{C_\varepsilon}{c^2}\norm {\partial_t E }_{L^2_t L^3 }^2 \norm B_{L^\infty_t L^6 }^2
+ \varepsilon\Big( \norm {\nabla \omega_\theta}_{L^2_t L^2 }^2 + \norm {\frac{\omega_\theta}{r}}_{L^2_t L^ 2}^2 \Big)
\\
&\leq     \frac{C_\varepsilon}{c^2}\norm {\partial_t E }_{L^2_t \dot{H}^\frac{1}{2} }^2 \norm B_{L^\infty_t \dot{H}^1 }^2
+ \varepsilon\Big( \norm {\nabla \omega_\theta}_{L^2_t L^2 }^2 + \norm {\frac{\omega_\theta}{r}}_{L^2_t L^ 2}^2 \Big).
\end{aligned}
\end{equation}

All in all, incorporating \eqref{EQ1}, \eqref{EQ2} and \eqref{EQ3} into \eqref{EQ0}, and choosing $\varepsilon$ small enough yields that
\begin{equation}  \label{EQ4}
\begin{aligned}
\norm{\omega(t)}_{ L^2 }^2
+ \norm{\nabla \omega_\theta}_{L^2_t L^2 }^2
&+\norm{\frac{\omega_\theta}{r}}_{L^2_t L^2 }^2
\\
&\leq   \norm{\omega_0}_{L^2}^2   + \frac{1}{4}   \norm { \Omega}_{L^2_t \dot{H}^1 } ^2+  C\int_{0}^t\norm {u(\tau)}_{\dot{H}^1}^2\norm {\omega (\tau)}_{L^2}^2d\tau    \\  
  &\quad   +  C  \norm{ \Gamma }_{L^\infty_t L^3 }^2 \norm {B}_{L^2_t \dot{H}^1 }^2  + \frac{C }{c^2}\norm {\partial_t E }_{L^2_t \dot{H}^\frac{1}{2} }^2 \norm B_{L^\infty_t \dot{H}^1 }^2 ,
\end{aligned} 
\end{equation} 
where we have utilized \eqref{ID:00}.

 Now, we show how to estimate  $\Omega$ in the right-hand side above. To that end, we first perform an $L^2$-energy estimate on \eqref{Omega-equa-form2:B} followed by a standard application of paraproduct laws to find, for any $t\in [0,T)$, that 
   \begin{eqnarray*}
\begin{aligned}
	\frac{1}{2}\norm{\Omega(t)}_{ L^2 }^2 + \norm{\Omega}_{ L^2_t \dot{H}^1 } ^2  &\leq \frac{1}{2}\norm{\Omega_0}_{L^2}^2     +    \norm{ \Gamma ^2  }_{L^2_tL^2} \norm {\Omega}_{L^2_t \dot{H}^1}  + \frac{1}{c}\norm {\partial_t E  \cdot \nabla \Gamma} _{L^2_t \dot H^{-1}} \norm {\Omega}_{L^2_t \dot H^1}\\
	    &\leq \frac{1}{2}\norm{\Omega_0}_{L^2}^2     +     C \norm{ \Gamma }_{L^4_t L^4 }^4  + \frac{C }{c^2}\norm {\partial_t E }_{L^2_t \dot{H}^ \frac{1}{2} }^2 \norm \Gamma _{L^\infty_t \dot{H}^1  }^2  + \frac{1}{2} \norm{\Omega}_{L^2_t \dot H^1}^2 ,
\end{aligned}
 \end{eqnarray*}
for some universal constant $C>0$.
Therefore, by further employing H\"older's and embedding inequalities, we obtain that 
 \begin{eqnarray*}
\begin{aligned}
	 \norm{\Omega(t)}_{ L^2 }^2 + \norm{\Omega}_{ L^2_t \dot{H}^1 } ^2   
	    &\leq \norm{\Omega_0}_{L^2}^2     +   C \norm{ \Gamma }_{L^\infty_t L^3 }^2 \norm {\Gamma}_{L^2_t \dot{H}^1 }^2  + \frac{C }{c^2}\norm {\partial_t E }_{L^2_t \dot{H}^ \frac{1}{2} }^2 \norm \Gamma _{L^\infty_t \dot{H}^1  }^2 .
\end{aligned}
 \end{eqnarray*}
 Hence, recalling that 
 $$\Gamma=\frac{B_\theta}{r},$$
 and   employing Lemma \ref{lemma-f/x:2}, we arrive at the bound  
 \begin{eqnarray}\label{EQ5}
\norm{\Omega(t)}_{ L^2 }^2 + \norm{\Omega}_{ L^2_t \dot{H}^1 } ^2 \leq    \norm{\Omega_0}_{L^2}^2     +   C \norm{ \Gamma }_{L^\infty_t L^3 }^2 \norm {\Gamma}_{L^2_t \dot{H}^1 }^2  +  \frac{C }{c^2}\norm {\partial_t E }_{L^2_t \dot{H}^ \frac{1}{2} }^2 \norm B _{L^\infty_t \dot{H}^2 }^2 .
 \end{eqnarray}
 At last,  summing \eqref{EQ4} and \eqref{EQ5}, and utilizing \eqref{ID:00}, we arrive at the estimate
\begin{equation*}
\begin{aligned}
\norm{\omega(t)}_{ L^2 }^2
+ \norm{ \omega}_{L^2_t \dot{H}^1 }^2
&+\norm{\Omega(t)}_{ L^2 }^2
+ \norm{\Omega}_{ L^2_t \dot{H}^1 } ^2
\\
&\lesssim   \norm{\omega_0}_{L^2}^2+\norm{\Omega_0}_{L^2}^2
+ \int_{0}^t\norm {u(\tau)}_{\dot{H}^1}^2\norm {\omega (\tau)}_{L^2}^2d\tau
\\
&\quad +
\left(\norm { \frac{1}{c}\partial_t E}_{L^2_t\dot{H}^\frac{1}{2}  }   \norm {B}_{L^\infty_tH^2   }     + \norm {\Gamma}_{L^\infty_t L^3 }\norm {(B,\Gamma)}_{L^2_t \dot{H}^1 } \right)^2.
\end{aligned} 
\end{equation*}

Finally, an application of the classical Gr\"onwall inequality yields that
 \begin{align*}
\norm{(\omega,\Omega)}_{L^\infty_tL^2 \cap L^2_t\dot{H}^1 }
 &\lesssim
  \left(\norm{(\omega_0,\Omega_0)}_{L^2}+\norm { \frac{1}{c}\partial_t E}_{L^2_t\dot{H}^\frac{1}{2}  }   \norm {B}_{L^\infty_tH^2   }     + \norm {\Gamma}_{L^\infty_t L^3 }\norm {(B,\Gamma)}_{L^2_t \dot{H}^1 } \right)
  \\
  &\quad\times\exp\left(C\int_{0}^t\norm {u(\tau)}_{\dot{H}^1}^2 d\tau\right),
 \end{align*}
which, in view of the energy inequality \eqref{L^2-Energy}, concludes the proof of the proposition.
\end{proof}

In view of    the   estimates on the velocity field given in the preceding proposition, we now need to establish the following bounds on the electromagnetic   field:
\begin{itemize}
\item An asymptotically  vanishing bound for $\frac{1}{c} \partial_t E$ of the form
\begin{equation}\label{Claim:01}
	\frac{1}{c}\norm {\partial_t E}_{L^2_t \dot{H}^\frac{1}{2}  } \lesssim c^{-\alpha} F\Big(\norm{(u,E,B)}_X\Big), 
\end{equation}  
for some $\alpha>0$, a (nonlinear) function $F$ and a suitable functional space $X$.
\item An asymptotically global bound for $ B$ in $L^2_t\dot{H}^1$ and $\Gamma$ in $L^\infty_t L^3  \cap L^2_t \dot{H}^1 $   of the form 
\begin{equation}\label{Claim:02}
	\norm B_{ L^2_t \dot{H}^1 } +  \norm { \Gamma}_{L^\infty_t L^3  \cap L^2_t \dot{H}^1  } \leq C_0 + c^{-\beta} F\Big(\norm{(u,E,B)}_X\Big),
\end{equation}  
for some $\beta>0$ and  $C_0>0$ depending  only  on the initial data.
\end{itemize}
 The complete justification of these bounds will be the subject of   Section \ref{section:asymptotics}, later on.

\subsection{Controlling the electromagnetic field}\label{section.EB:ES}

Here, we establish several estimates combining  the refined study of the dispersive properties of Maxwell's equations from \cite{ah} with the axisymmetric structure. The principal results of this part of the article are obtained in Sections \ref{subsubsection:HF} and \ref{subsubsection:LF}, below.

  The ensuing bounds will be further combined with the results from the previous section to arrive at a global control of solutions to \eqref{Maxwell:system:*2} in Section \ref{section:closing:ES}, later on. 

For the sake of clarity, we recall first the essential  results from  \cite{ah} on Strichartz estimates and maximal parabolic regularity for the three-dimensional damped Maxwell system which are useful in the present work.
  
\begin{lem}\cite[Corollary 2.12]{ah}\label{cor:maxwell}
	Consider a solution $(E,B) :[0,T)\times\mathbb{R}^3\to\mathbb{R}^6$ of the damped Maxwell system
	\begin{equation}\label{damped:Maxwell:system}
		\begin{cases}
			\begin{aligned}
				\frac{1}{c} \partial_t E - \nabla \times B + \sigma c E & = G,
				\\
				\frac{1}{c} \partial_t B + \nabla \times E & = 0,
				\\
				\div B & =0,
			\end{aligned}
		\end{cases}
	\end{equation}
	for some initial data $(E,B)(0,x)=(E_0,B_0)(x)$, where $\sigma> 0$ and $c>0$.
	
	For any exponent pairs $(q,r),(\tilde q,\tilde r)\in [1,\infty]\times[2,\infty)$ which are admissible in the sense that
	\begin{equation*}
		\frac 1q+\frac {1} r\geq \frac {1} 2
		\qquad\text{and}\qquad
		\frac 1{\tilde q}+\frac {1} {\tilde r}\geq \frac {1} 2,
	\end{equation*}
	and such that
	\begin{equation*}
		\frac 1q+\frac 1{\tilde q}\leq 1,
	\end{equation*}
	one has the high-frequency estimate
	\begin{equation*}
		\begin{aligned}
			2^{-j\left(1-\frac 2r\right)}&\norm{\Delta_j (PE,B)}_{L^q([0,T);  L^r)  }
			\\
			&\lesssim
			c^{\frac 12-\frac 1r-\frac 2q}
			\norm{\Delta_j (PE_0,B_0)}_{L^2 }
			+c^{2-\frac 1r-\frac 1{\tilde r}-\frac 2q-\frac 2{\tilde q}}
			2^{j\left(1-\frac 2{\tilde r}\right)}\norm{\Delta_j PG}_{L ^{\tilde q'}([0,T);   L^{\tilde r'}  )},
		\end{aligned}
	\end{equation*}
	for all $j\in\mathbb{Z}$ with $2^j\geq \sigma c$, and the low-frequency estimates
	\begin{equation*}
		\begin{aligned}
			2^{-j\left(\frac 32-\frac 3r\right)}&\norm{\Delta_j PE}_{L^q([0,T);  L^r ) }
			\\
			&\lesssim
			c^{-\frac 2q}\norm{\Delta_j PE_0}_{L^2 }
			+c^{-1}
			2^{j\left(1-\frac 2q\right)}\norm{\Delta_j B_0}_{L^2 }
			+c^{1-\frac 2q-\frac 2{\tilde q}}
			2^{j\left(\frac 32-\frac 3{\tilde r}\right)}\norm{\Delta_j PG}_{L ^{\tilde q'} ([0,T);  L^{\tilde r'})  }
		\end{aligned}
	\end{equation*}
	and
	\begin{equation*}
		\begin{aligned}
			2^{-j\left(\frac 32-\frac 3r-\frac 2q\right)}&\norm{\Delta_j B}_{L^q([0,T);   L^r)  }
			\\
			&\lesssim
			c^{-1}2^{j}
			\norm{\Delta_j PE_0}_{L^2 }
			+\norm{\Delta_j B_0}_{L^2 }
			+
			2^{j\left(\frac 52-\frac 3{\tilde r}-\frac 2{\tilde q}\right)}\norm{\Delta_j PG}_{L ^{\tilde q'}([0,T);   L^{\tilde r'}  )},
		\end{aligned}
	\end{equation*}
	for all $j\in\mathbb{Z}$ with $2^j\leq \sigma c$.
\end{lem}

\begin{lem}\cite[Corollary 2.14]{ah} \label{cor:parabolic:maxwell}
	Consider a solution $(E,B) :[0,T)\times\mathbb{R}^3\to\mathbb{R}^6$ of the damped Maxwell system \eqref{damped:Maxwell:system}, for some initial data $(E,B)(0,x)=(E_0,B_0)(x)$, where $\sigma> 0$ and $c>0$.

	For any $\chi\in C^\infty_c(\mathbb{R}^d)$ and $s\in\mathbb{R}$, one has the low-frequency estimates
	\begin{equation*}
		\begin{aligned}
			\norm{\chi(c^{-1}D)PE}_{L^m_t([0,T);\dot B^{s+\frac 2m}_{2,q})}
			&\lesssim
			c^{-\frac 2m}\norm{PE_0}_{\dot B^{s+\frac 2m}_{2,q}}+c^{-1}\norm{B_0}_{\dot B^{s+1}_{2,m}}
			\\
			&\quad+c^{-1+\frac 2r-\frac 2m}\norm{PG}_{L_t^r([0,T);\dot B_{2,q}^{s+\frac 2m})},
		\end{aligned}
	\end{equation*}
	for any $1<r\leq m<\infty$ and $1\leq q\leq \infty$, as well as
	\begin{equation*}
		\norm{\chi(c^{-1}D)B}_{L^m_t([0,T);\dot B^{s+\frac 2m}_{2,1})}
		\lesssim
		c^{-1}\norm{PE_0}_{\dot B^{s+1}_{2,m}}+\norm{B_0}_{\dot B^{s}_{2,m}}
		+\norm{PG}_{L_t^r([0,T);\dot B_{2,\infty}^{s-1+\frac 2r})},
	\end{equation*}
	for any $1<r<m<\infty$, and
	\begin{equation*}
		\norm{\chi(c^{-1}D)B}_{L^m_t([0,T);\dot B^{s+\frac 2m}_{2,q})}
		\lesssim
		c^{-1}\norm{PE_0}_{\dot B^{s+1}_{2,m}}+\norm{B_0}_{\dot B^{s}_{2,m}}
		+\norm{PG}_{L_t^m([0,T);\dot B_{2,q}^{s-1+\frac 2m})},
	\end{equation*}
	for any $1<m<\infty$ and $1\leq q\leq\infty$.
\end{lem}

Let us now be more precise about the source term $G$ which will be used in the application of the preceding two lemmas. Specifically, we will consider the Maxwell system 
\begin{equation}\label{Maxwell:system:*}
		\begin{cases}
			\begin{aligned}
				\frac{1}{c} \partial_t E - \nabla \times B + \sigma c E & =- \sigma P( u \times B), & \div E&=0,
				\\
				\frac{1}{c} \partial_t B + \nabla \times E & = 0, & \div B&=0,
				\\
				\div u & =0 .
			\end{aligned}
		\end{cases}
	\end{equation}
	
	Furthermore, in order to exploit the dichotomy between high and low frequencies featured in the estimates from the above lemmas, we consider the variants of Besov semi-norms
\begin{equation*}
	\left\|f\right\|_{\dot B^{s}_{p,q,<}}\bydef
	\left(
	\sum_{\substack{k\in\mathbb{Z}\\ 2^k< \sigma c}} 2^{ksq}
	\left\|\Delta_{k}f\right\|_{L^p}^q\right)^\frac{1}{q}
	\quad\text{and}\quad
	\left\|f\right\|_{\dot B^{s}_{p,q,>}}\bydef
	\left(
	\sum_{\substack{k\in\mathbb{Z}\\ 2^k\geq \sigma c}} 2^{ksq}
	\left\|\Delta_{k}f\right\|_{L^p}^q\right)^\frac{1}{q},
\end{equation*}
as well as the corresponding variants of Chemin--Lerner semi-norms
\begin{equation*}
	\left\|f\right\|_{\widetilde L^r_t\dot B^{s}_{p,q,<}}\bydef
	\left(
	\sum_{\substack{k\in\mathbb{Z}\\ 2^k< \sigma c}} 2^{ksq}
	\left\|\Delta_{k}f\right\|_{L^r_tL^p_x}^q\right)^\frac{1}{q}
	\quad\text{and}\quad
	\left\|f\right\|_{\widetilde L^r_t\dot B^{s}_{p,q,>}}\bydef
	\left(
	\sum_{\substack{k\in\mathbb{Z}\\ 2^k\geq \sigma c}} 2^{ksq}
	\left\|\Delta_{k}f\right\|_{L^r_tL^p_x}^q\right)^\frac{1}{q},
\end{equation*}
for any $s\in\mathbb{R}$ and $0<p,q,r\leq\infty$ (with obvious modifications if $q$ is infinite). These families of semi-norms have been introduced in \cite{ah}. We will utilize them extensively throughout the upcoming sections of our work.

	\subsubsection{Control of high-frequency electromagnetic waves}\label{subsubsection:HF}

	Here, we establish key bounds on   high frequencies of electromagnetic fields. Lemma \ref{lemma-high F} below combines the damped Strichartz estimates for high electromagnetic frequencies from Lemma \ref{cor:maxwell} with 
	the paradifferential product laws on axisymmetric vector fields from Lemma \ref{paradifferential:1}.  

	The method behind the proof of this lemma is similar to the one used in \cite{ah} (see Lemma 3.8, therein). However, here, we further refine the method by introducing an additional high--low frequency decomposition of the source term $P(u\times B)$ in \eqref{lemma-high F}. This will allow us to obtain stronger estimates (see \eqref{E_>:5/2}, below).
	
\begin{lem}\label{lemma-high F}
Let $T\in (0,\infty]$ and  $(E,B)$  be a smooth axisymmetric solution to \eqref{Maxwell:system:*}, defined on    $[0,T)$,  for some initial data $(E_0,B_0)$ and some axisymmetric divergence-free vector field $u$. Assume further that $ E$ and $u$ are both   without swirl and that $B$ has   pure swirl.

 Then, for any $s\in  (-\frac{5}{2}, \frac{5}{2})$, $n\in [1,\infty]$ and any $q\in [\frac{4}{3},\infty]$, it holds that 
  \begin{equation*}
  	 \norm{ (E,B) }_{\widetilde{L} ^q_t \dot{B}^{s}_{2,n,>} }  \lesssim 
  	           c^{-\frac{2}{q}}\norm { (E_0,B_0)}_{\dot{B}^s_{2,n,>}}  +c^{\frac{1}{2} - \frac{2}{q}  }\norm u_{L^\infty _t \dot{H} ^1 \cap L^2_t\dot{H} ^2  }   \norm{ B }_{\widetilde{L} ^2_t \dot{B}^{s}_{2,n}  }  .
  \end{equation*}
	Moreover, at the endpoint   $s=\frac{5}{2}$, we have that 
	\begin{equation*}
		 \norm{ (E,B) }_{\widetilde{L} ^q_t  \dot{B}^{\frac{5}{2} }_{2,1,>}  }  
		  \lesssim  
		   c^{-\frac{2}{q}}\norm { (E_0,B_0)}_{\dot{B}^\frac{5}{2} _{2,1,>}}  +c^{-\frac{1}{2} + \frac{2}{p} - \frac{2}{q}  }\norm u_{L^\infty_t \dot{H} ^1 \cap L^2_t \dot{H} ^2  }   \norm{ B }_{\widetilde{L} ^p_t\dot{B}^{ \frac{5}{2}}_{2,1}  }, 
	\end{equation*} 
	as soon as  $   2 \leq p \leq  q \leq \infty$.
	Furthermore, in the case $p=2$, it holds that
	\begin{equation}\label{E_>:5/2}
		\begin{aligned}
			\norm{ (E,B) }_{\widetilde{L} ^q_t \dot{B}^{\frac{5}{2} }_{2,1,>}  }   
			&\lesssim   c^{-\frac{2}{q}}\norm { (E_0,B_0)}_{\dot{B}^\frac{5}{2} _{2,1,>}}  
			 \\
	        & \quad +c^{  \frac{1}{2} - \frac{2}{q}   }\norm u_{L^\infty _t \dot{H} ^1 \cap L^2_t\dot{H} ^2  }  \left(  \norm{ B }_{\widetilde{L} ^2_t \dot{B}^{ \frac{5}{2}}_{2,1,>}  } +  \norm{ B }_{L ^2_t \dot{B}^{ \frac{5}{2}}_{2,1,<} }    \right),
		\end{aligned}
	\end{equation} 
	for all $ q\in [2,\infty]$.
\end{lem}
\begin{proof}
    Applying     Lemma \ref{cor:maxwell} yields that
    \begin{equation*}
		\begin{aligned}
			\norm{\Delta_j (E,B)}_{L^q_t L^2 } 
			\lesssim
			c^{ -\frac 2q}	 
			\norm{\Delta_j (E_0,B_0)}_{L^2}  +c^{-1+2(\frac 1{ p}-\frac 1q)}
			 \norm{\Delta_j  P\big(u \times B\big)}_{L_t^{ p} L^{2} },
		\end{aligned}
	\end{equation*}
	for all $j\in\mathbb{Z}$, with $2^j\geq \sigma c$, and   any $1\leq p\leq q \leq \infty$. It then follows, for any $s\in \mathbb{R}$ and any $n\in [1,\infty]$, that
\begin{equation}\label{high-F-EB}
	\norm{ (E,B)}_{\widetilde{L} ^q_t \dot{B}^{s}_{2,n,>}  }  \lesssim    c^{-\frac{2}{q}}\norm { (E_0,B_0) }_{\dot{B}^s_{2,n,>}}   +c^{-1+2(\frac 1{ p }-\frac 1q)} \norm{   P\big(u \times B \big)}_{\widetilde{L} _t^{p } \dot{B}^{s}_{2,n }  }  .
\end{equation}

	   Thus,   choosing $p= \frac{4}{3}$  
and   utilizing  \eqref{paraproduct:2} we find, for all $s\in (-\frac{3}{2},\frac{5}{2})$,  $n\in [1,\infty]$ and any $q\in [\frac{4}{3}, \infty]$, that
\begin{equation*}
 \norm{ (E,B) }_{\widetilde{L} ^q_t \dot{B}^{s}_{2,n,>}  }   \lesssim   c^{-\frac{2}{q}}\norm { (E_0,B_0)}_{\dot{B}^s_{2,n,>}}
 +c^{ \frac{1}{2}  - \frac{2}{q} }\norm u_{L^4_t L^\infty \cap\widetilde{L} ^4_t \dot{B}^{ \frac{3}{2}}_{2, \infty}  }   \norm{ B }_{\widetilde{L} ^2_t \dot{B}^{s}_{2,n}  }  .
\end{equation*} 
	Hence, we conclude the proof of the first claim in the lemma by employing   the embedding 
	$$L^\infty _t \dot{H} ^1 \cap L^2 _t \dot{H} ^2   \hookrightarrow L ^4_t \dot{B}^{\frac{3}{2}}_{2,1}  \hookrightarrow L^4_t L^\infty \cap\widetilde{L} ^4_t \dot{B}^{ \frac{3}{2}}_{2, \infty}   . 
	$$ 
	
	We now turn  to the endpoint case $(s,n)= (\frac{5}{2},1)$, which corresponds to the second bound in the statement of the lemma.  	
	The   natural attempt to estimate the product $P(u\times B)$ in that case would be by applying \eqref{paraproduct:3}, which is the corresponding extension of \eqref{paraproduct:2}. Doing so would lead to the control
		\begin{equation}\label{first-try:ES}
		\begin{aligned}
		 \norm{ (E,B) }_{\widetilde{L} ^q_t \dot{B}^{\frac 52}_{2,1,>}  }
		 &\lesssim   c^{-\frac{2}{q}}\norm { (E_0,B_0)}_{\dot{B}^\frac52_{2,1,>}}
		 +  c^{-1+2\left(   \frac{1}{4}+ \frac{1}{p}  -\frac 1q \right)}
		 \norm{   P\big(u \times B \big)}_{\widetilde{L} _t^{\left( \frac{1}{4}+ \frac{1}{p} \right)^{-1} } \dot{B}^{\frac 52}_{2,1 }  } \\
	 &\lesssim   c^{-\frac{2}{q}}\norm { (E_0,B_0)}_{\dot{B}^\frac 52_{2,1,>}}  +c^{- \frac{1}{2} - \frac{2}{q} + \frac{2}{p}  }
	 \norm u_{\widetilde{L}^4_t \dot{B}^{\frac{3}{2}}_{2,1}  }   \norm{ B }_{\widetilde{L} ^p_t \dot{B}^{\frac52}_{2,1}  } ,
		\end{aligned}
		\end{equation} 
		as soon as 
		 $$ 1\leq \frac{4p}{p+4} \leq q \leq \infty.$$

	Observe then that, in order to conclude, the preceding estimate would require a stronger control on the velocity field $u$, for even though one has that
	\begin{equation*}
		\widetilde L^\infty _t \dot{H} ^1 \cap L^2 _t \dot{H} ^2
		\hookrightarrow   \widetilde{L}^4_t \dot{B}^{\frac{3}{2}}_{2,1},
	\end{equation*}
	it is unclear whether the embedding
	$$L^\infty _t \dot{H} ^1 \cap L^2 _t \dot{H} ^2   \hookrightarrow   \widetilde{L}^4_t \dot{B}^{\frac{3}{2}}_{2,1}   $$
	holds or not. Accordingly, \eqref{first-try:ES} does not seem to be useful.
	 
	  Instead, we have to prove the desired estimate by  ``hand'' (that is, by cooking up a suitable interpolation argument). To that end, we first introduce the decomposition 
	$$ u= u_\ell + u_{h} \bydef \left(\mathds{ 1}_{|D|<  \frac{\sigma c}{2} }+ \mathds{ 1}_{|D|\geq  \frac{\sigma c }{2} }  \right) u    . $$
	Then,    by splitting the source term $P(u\times B)$ according to the latter decomposition of $u$ and by applying \eqref{high-F-EB} to each Maxwell system corresponding to the source terms $P(u_\ell\times B)$ and $P(u_h\times B)$, with different values of $p$, one obtains that
	\begin{equation} \label{AAA}
	\begin{aligned}
	 \norm{ (E,B)}_{\widetilde{L} ^q_t \dot{B}^{s}_{2,1,>}  }  &\lesssim     c^{-\frac{2}{q}}\norm { (E_0,B_0) }_{\dot{B}^s_{2,1,>}}   +c^{-1 +  \frac{2}{p} - \frac{2}{q}    } \norm{   P\big(u_\ell \times B \big)}_{\widetilde{L} _t^{p } \dot{B}^{s }_{2,1,> }  }  \\
	  & \quad+c^{ \frac{2}{p}- \frac{2}{q}}  \norm{   P\big(u _{h}\times B \big)}_{\widetilde{L} _t^{ \left( \frac{1}{2} + \frac{1}{p} \right)^{-1} }\dot{B}^{s }_{2,1,> } } ,
	\end{aligned}
\end{equation}
	for all $2\leq p\leq q$ and all $s\in\mathbb{R}$.
	Therefore, applying the   product law estimate \eqref{paraproduct:3}, we infer that
	\begin{equation*}
	\begin{aligned} 
	 \norm{ (E,B)}_{\widetilde{L} ^q_t \dot{B}^{\frac 52}_{2,1,>}  }  &\lesssim     c^{-\frac{2}{q}}\norm { (E_0,B_0) }_{\dot{B}^\frac 52_{2,1,>}} \\
	 &\quad   +c^{-1 +  \frac{2}{p} - \frac{2}{q}    }  \norm {u_{\ell}}_{\widetilde{L}^ {\infty}_t \dot{B}^{\frac{3}{2}}_{2,1}  }   \norm{ B }_{\widetilde{L} ^p_t \dot{B}^{\frac 52}_{2,1}  } + c^{ \frac{2}{p}- \frac{2}{q}} \norm {u_{h}}_{\widetilde{L}^{2} _t \dot{B}^{\frac{3}{2}}_{2,1}  }   \norm{ B }_{\widetilde{L} ^p_t \dot{B}^{\frac 52}_{2,1}  }  .
	\end{aligned}
	\end{equation*} 
	
	Now, observing that 
	\begin{equation}\label{u:ell}
	\norm {u_{\ell}}_{\widetilde{L}^ {\infty}_t \dot{B}^{\frac{3}{2}}_{2,1}  }
	\lesssim
	c ^{\frac{1}{2}}
	\norm {u_{\ell}}_{{L}^ {\infty}_t \dot{B}^{1}_{2,\infty}  }   
	\lesssim
	c ^{\frac{1}{2}}\norm {u }_{L_t^{\infty}  \dot{H}^1  }  
	\end{equation}
	and 
	\begin{equation*}
	\norm {u_{h}}_{\widetilde{L}^ {2}_t \dot{B}^{\frac{3}{2}}_{2,1}  }   \lesssim  c ^{-\frac{1}{2}}\norm {u }_{L_t^2 \dot{H}^2  }   
	\end{equation*}
	leads to the desired control
	\begin{equation*}
	\begin{aligned}
	 \norm{ (E,B)}_{\widetilde{L} ^q_t \dot{B}^{\frac 52}_{2,1,>}  }
	 &\lesssim     c^{-\frac{2}{q}}\norm { (E_0,B_0) }_{\dot{B}^\frac 52_{2,1,>}}   +c^{ - \frac{1}{2} - \frac{2}{q}+\frac{2}{p}}  \norm u_{L_t^{\infty}  \dot{H}^1   \cap L_t^2 \dot{H}^2  }    \norm{ B }_{\widetilde{L} ^p_t \dot{B}^{\frac 52}_{2,1}  },
	\end{aligned}
\end{equation*}	   
for any $2\leq p\leq q$.
This takes care of the second bound in the statement of the lemma.

	 Finally, in order for us to justify the last estimate of the lemma (which is an improvement of the preceding bound in the case $p=2$), we need to further decompose $B$ in the source term $ P(u_\ell\times  B)$.
	 To that end, we write
	 $$ B= B_\ell + B_h \bydef \left(\mathds{ 1}_{|D|< \frac{\sigma c}{2} }+ \mathds{ 1}_{|D|\geq \frac{\sigma c}{2} }  \right) B ,     $$
	which allows to deduce from \eqref{AAA} that
	\begin{equation} \label{VVV}
		\begin{aligned}
			\norm{ (E,B)}_{\widetilde{L} ^q_t \dot{B}^{ \frac{5}{2}}_{2,1,>}  }
			&\lesssim     c^{-\frac{2}{q}}\norm { (E_0,B_0) }_{\dot{B}^ \frac{5}{2} _{2,1,>}}
			+c^{- \frac{2}{q}    } \norm{   P\big(u_\ell \times B_\ell \big)}_{\widetilde{L} _t^{2 } \dot{B}^{\frac{5}{2}}_{2,1,> }  }
			\\
			&\quad
			+  c^{  - \frac{2}{q}    } \norm{   P\big(u_\ell \times B_h \big)}_{\widetilde{L} _t^{2 } \dot{B}^{\frac{5}{2}}_{2,1 } }
			+c^{1 - \frac{2}{q}} \norm{   P\big(u _{h}\times B \big)}_{L _t^{1} \dot{B}^{\frac{5}{2}}_{2,1 }  } ,
		\end{aligned}
	\end{equation}
	for all $q\geq 2$, where we have used that
	$$  \widetilde{L} _t^{1} \dot{B}^{\frac{5}{2}}_{2,1 }=L _t^{1} \dot{B}^{\frac{5}{2}}_{2,1 }   .   $$

	  Now, observe that 
	  $$ \supp\left(  \widehat{ u_\ell \times B_\ell} \right) \subset \left\{ \xi : |\xi|< \sigma c \right\},$$
	which implies that
	\begin{equation*}
		\Delta_j\left(  u_\ell \times B_\ell \right)\equiv 0,
	\end{equation*}
	whenever $2^j\geq 2\sigma c$.
	Consequently, it follows that 
	    \begin{equation*}
	     \norm{   P\big(u_\ell \times B_\ell \big)}_{\widetilde{L} _t^{2 } \dot{B}^{\frac{5}{2}}_{2,1,> }  }
		=
		2^{\frac 52j}\norm{   P\Delta_j \big(u_\ell \times B_\ell \big)}_{{L} _t^{2 } L^2  }
		=
		\norm{   P\big(u_\ell \times B_\ell \big)}_{L _t^{2 } \dot{B}^{\frac{5}{2}}_{2,1 ,> }  } ,
	   \end{equation*}
	where $j$ is the only integer value which satisfies $\sigma c\leq 2^j<2\sigma c$.
	
	Hence,    by applying \eqref{paraproduct:3} for classical Besov-space-valued Lebesgue spaces, we obtain that
	\begin{equation*}
		\begin{aligned}
			\norm{   P\big(u_\ell \times B_\ell \big)}_{\widetilde{L} _t^{2 } \dot{B}^{\frac{5}{2}}_{2,1,> }  }
			&=
			\norm{   P\big(u_\ell \times B_\ell \big)}_{L _t^{2 } \dot{B}^{\frac{5}{2}}_{2,1,> }  }   
			\\
			&\lesssim \norm{    u_\ell   }_{L _t^{\infty } \dot{B}^{\frac{3}{2}}_{2,1   }  } \norm{   B_\ell }_{L _t^{2 } \dot{B}^{\frac{5}{2}}_{2,1    }  }
			\lesssim c^{\frac{1}{2}}\norm{    u   }_{L _t^{\infty } \dot{H}^1 } \norm{   B }_{L _t^{2 } \dot{B}^{\frac{5}{2}}_{2,1,<  }  },
		\end{aligned}
	\end{equation*} 
	where we used \eqref{u:ell}.
	Similarly, we find, by employing  \eqref{paraproduct:3} and \eqref{u:ell}, again, that
	\begin{equation*}
		\begin{aligned}
			\norm{   P\big(u_\ell \times B_h \big)}_{\widetilde{L} _t^{2 } \dot{B}^{\frac{5}{2}}_{2,1 }  }
			&\lesssim
			\norm{    u_\ell   }_{\widetilde{L}  _t^{\infty } \dot{B}^{\frac{3}{2}}_{2,1   }  }
			\norm{   B_h }_{\widetilde{L}  _t^{2 } \dot{B}^{\frac{5}{2}}_{2,1   }  }
			\\
			&\lesssim
			\norm{    u_\ell   }_{\widetilde{L}  _t^{\infty } \dot{B}^{\frac{3}{2}}_{2,1   }  }
			\left(
			\norm{   B  }_{\widetilde{L}  _t^{2 } \dot{B}^{\frac{5}{2}}_{2,1,>    }  }
			+ \sum_{\frac{\sigma c}4< 2^j<\sigma c} 2^{\frac52j}\norm{   \Delta_j B }_{L^2_t  L^2 }
			\right)
			\\
			&\lesssim
			c^{\frac{1}{2}}\norm{    u   }_{L _t^{\infty } \dot{H}^1 }
			\left(
			\norm{   B  }_{\widetilde{L}  _t^{2 } \dot{B}^{\frac{5}{2}}_{2,1,>    }  }
			+  \norm{   B }_{L^2_t  \dot{B}^{\frac{5}{2}}_{2,1  ,< }  }
			\right).
		\end{aligned}
	\end{equation*}
	
	As for the last term in \eqref{VVV},  applying \eqref{paraproduct:3},  again, in classical Besov-space-valued Lebesgue spaces, entails that
	\begin{equation*}
	\begin{aligned}
	\norm{   P\big(u _{h}\times B \big)}_{L _t^{1} \dot{B}^{\frac{5}{2}}_{2,1 }  }  &\lesssim   \norm{    u_h   }_{L _t^{2 } \dot{B}^ \frac{3}{2}_{2,1} }  \norm{   B  }_{L  _t^{2 } \dot{B}^{\frac{5}{2}}_{2,1     }  }  \\
	&\lesssim      c^{-\frac{1}{2}}\norm{    u   }_{L _t^{2 } \dot{H}^2 } \left( \norm{   B  }_{\widetilde{L}  _t^{2 }\dot{B}^{\frac{5}{2}}_{2,1,>    }  } +  \norm{   B }_{L^2_t  \dot{B}^{\frac{5}{2}}_{2,1  ,< }  } \right).
	\end{aligned}
	\end{equation*} 
	All in all, gathering the previous estimates and incorporating them into \eqref{VVV} yields  the control \eqref{E_>:5/2}, which completes the proof of the lemma.
\end{proof}

\subsubsection{Control of low-frequency electromagnetic waves}\label{subsubsection:LF}

It is clear from Lemmas \ref{cor:maxwell} and \ref{cor:parabolic:maxwell}  that solutions to the damped Maxwell system enjoy  various types of bounds in different regions of low and high frequencies. 

Here, we intend to shed light on the low-frequency control of electromagnetic fields solving \eqref{Maxwell:system:*}. In broad terms,  the low-frequency bounds on the electric field $E$ are similar to the same bounds in the high-frequency regime. However,   the low-frequency part of the magnetic field $B$ enjoys parabolic-type estimates, which are consistent with the limiting system \eqref{MHD*} as $c$ goes to infinity.   A more precise formulation of that principle is given in the next lemma.

 \begin{lem}\label{low:freq:estimates} 
Let $T\in (0,\infty]$ and  $(E,B)$  be a smooth axisymmetric solution to \eqref{Maxwell:system:*}  on    $[0,T)$,  for some initial data $(E_0,B_0)$ and some axisymmetric divergence-free vector field $u$. Assume further that $ E$ and $u$ are both   without swirl and that $B$ has  pure swirl.

	Then, for any $\alpha\in [0,1]$, $n\in [1,\infty]$ and $ (m,p )\in [2,\infty]^2 $, with 
	$$ \alpha + \frac{2}{m} \leq \frac{3}{2},$$
	 and for any $s< \frac{5}{2}$, with $ s+\alpha + \frac{2}{m}>0$, one has the following low-frequency estimates 	\begin{equation}\label{frequencies:1}
		\begin{aligned}
			\norm{E}_{L^q_t\dot B^{s+\alpha+ \frac{2}{m}- \frac{3}{2}}_{2,n,<} }
			&\lesssim
			c^{-\frac 2q}\norm{E_0}_{\dot B^{s+\alpha+ \frac{2}{m}- \frac{3}{2}}_{2,n,<}}
			+c^{-1}\norm{B_0}_{\dot B^{s+\alpha+ \frac{2}{m}- \frac{1}{2} - \frac{2}{q}}_{2,q,<}}
			\\
			&\quad+c^{2\left(\frac 1p + \frac{1}{m}-\frac 1q\right)-1}
			\norm{u}_{L ^{m}_t\dot B^{\alpha + \frac{2}{m}}_{2,1} } 
			\norm{B}_{L ^{p}_t\dot B^s_{2,n} },
				\end{aligned}
			\end{equation}
			as soon as 
 $1< \left( \frac{1}{m} + \frac{1}{p}\right)^{-1}  \leq  q < \infty   , $ 			
			as well as 
				\begin{equation}\label{frequencies:1-B***}
		\begin{aligned}
			\norm{B}_{L^q_t\dot B^{s+\alpha  - \frac{1}{2}+ \frac 2q-\frac 2p}_{2,1,<} }
			&\lesssim
			c^{-1}\norm{E_0}_{\dot B^{s+\alpha  + \frac{1}{2} -\frac 2p}_{2,q,<}}
			+\norm{B_0}_{\dot B^{s+\alpha  - \frac{1}{2} -\frac 2p}_{2,q,<}}
			\\
			&\quad+\norm{u}_{L ^{m}_t\dot B^{\alpha + \frac{2}{m}}_{2,1} } 
			\norm{B}_{L ^{p}_t\dot B^s_{2,\infty} },
		\end{aligned}
	\end{equation}
	whenever 
	 $1< \left( \frac{1}{m} + \frac{1}{p}\right)^{-1}  < q < \infty   .$ 		
	 	
	Moreover, in the case where $$ 1< \left( \frac{1}{m} + \frac{1}{p}\right)^{-1}= q < \infty   , $$  it holds that 
	\begin{equation}\label{frequencies:1b}
	\begin{aligned}
			\norm{B}_{L^q_t\dot B^{s+\alpha  - \frac{1}{2}+ \frac 2q-\frac 2p}_{2,n,<} }
			 &\lesssim
			c^{-1}\norm{E_0}_{\dot B^{s+\alpha  + \frac{1}{2} -\frac 2p}_{2,q,<}}
			+\norm{B_0}_{\dot B^{s+\alpha  - \frac{1}{2} -\frac 2p}_{2,q,<}} +\norm{u}_{L ^{m}_t\dot B^{\alpha + \frac{2}{m}}_{2,1} } 
			\norm{B}_{L ^{p}_t\dot B^s_{2,n} }.
			\end{aligned} 
	\end{equation}
	
	At last, at the endpoint $s=\frac{5}{2}$, we have
	 \begin{equation}\label{frequencies:2}
		\begin{aligned}
			\norm{E}_{L^q_t\dot B^{1+\alpha+ \frac{2}{m} }_{2,n,<} }
			&\lesssim
			c^{-\frac 2q}\norm{E_0}_{\dot B^{1+\alpha+ \frac{2}{m} }_{2,n,<}}
			+c^{-1}\norm{B_0}_{\dot B^{2+\alpha+ \frac{2}{m}  - \frac{2}{q}}_{2,q,<}}
			\\
			&\quad+c^{2\left(\frac 1p + \frac{1}{m}-\frac 1q\right)-1}
			\norm{u}_{L ^{m}_t\dot B^{\alpha + \frac{2}{m}}_{2,1} } 
			\norm{B}_{L ^{p}_t\dot B^{\frac{5}{2}} _{2,1} },
			\\
			\norm{B}_{L^q _t\dot B^{2+\alpha   + \frac 2q-\frac 2p}_{2,n,<} }
			&\lesssim
			c^{-1}\norm{E_0}_{\dot B^{3+\alpha    -\frac 2p}_{2,q,<}}
			+\norm{B_0}_{\dot B^{2+\alpha   -\frac 2p}_{2,q,<}} +\norm{u}_{L ^{m}_t\dot B^{\alpha + \frac{2}{m}}_{2,1} } 
			\norm{B}_{L ^{p}_t\dot B^{\frac{5}{2}} _{2,1} },
		\end{aligned}
	\end{equation} 
as long as  	 $1< \left( \frac{1}{m} + \frac{1}{p}\right)^{-1}  \leq  q < \infty   .$ 	 
\end{lem}

\begin{proof}
	Similarly to  \cite[Lemma 3.9]{ah}, the proof   hinges upon the combination of Lemmas \ref{paradifferential:1}, \ref{cor:maxwell}   and \ref{cor:parabolic:maxwell}. 
Thus, on the one hand, applying the low-frequency estimate  from Lemma \ref{cor:maxwell}, for parameter values such that
$$r=\tilde{r}=2$$
and
$$\tilde{q}^\prime=  \left( \frac{1}{m}+ \frac{1}{p}\right)^{-1},$$
yields, as soon as $1 \leq  \left( \frac{1}{m}+ \frac{1}{p}\right)^{-1} \leq q \leq \infty$, that
\begin{equation*} 
		\begin{aligned}
		2^{j \left( s+\alpha+ \frac{2}{m}- \frac{3}{2} \right)}	\norm{\Delta_j E}_{L^q_tL^2}
			&\lesssim
			c^{-\frac 2q}\norm{E_0}_{\dot B^{s+\alpha+ \frac{2}{m}- \frac{3}{2}}_{2,\infty,<}}
			+c^{-1}\norm{B_0}_{\dot B^{s+\alpha+ \frac{2}{m}- \frac{1}{2} - \frac{2}{q}}_{2,\infty,<}}
			\\
			&\quad+c^{2\left(\frac 1p + \frac{1}{m}-\frac 1q\right)-1}
			\norm{P(u\times B) }_{L_t^{ \left(\frac{1}{m}+ \frac{1}{p} \right)^{-1}} \left( \dot B^{s+\alpha+ \frac{2}{m}- \frac{3}{2}}_{2,\infty} \right)}   
			\end{aligned}
	\end{equation*}
	and
	\begin{equation}\label{R0} 
		\begin{aligned}
		2^{j \left( s+\alpha  - \frac{1}{2}-\frac 2p + \frac 2q \right)}	\norm{\Delta_j B}_{L^q_tL^2}
			&\lesssim
			c^{-1}\norm{E_0}_{\dot B^{s+\alpha  + \frac{1}{2}  -\frac 2p}_{2,\infty ,<}}
			+\norm{B_0}_{\dot B^{s+\alpha  - \frac{1}{2} -\frac 2p}_{2,\infty ,<}}\\
			&\quad+\norm{P(u\times B) }_{L_t^{ \left(\frac{1}{m}+ \frac{1}{p} \right)^{-1} }\left( \dot B^{s+ \alpha + \frac{2}{m} - \frac{3}{2}}_{2,\infty }\right)} ,
		\end{aligned}
	\end{equation}
	for all $j\in \mathbb{Z}$ with $\frac{\sigma c}{2} \leq 2^j < \sigma c$.
	
	 On the other hand, employing the first and the second estimates from Lemma \ref{cor:parabolic:maxwell} entails
		\begin{equation*}
		\begin{aligned}
			\norm{\mathds{1}_{\{2^j<\frac{\sigma c}2\}}2^{j \left( s+\alpha+ \frac{2}{m}- \frac{3}{2} \right)}\norm{\Delta_j E}_{L^2_x}}_{L^q_t\ell^n_j}
			&\lesssim
			c^{-\frac 2q}\norm{E_0}_{\dot B^{  s+\alpha+ \frac{2}{m}- \frac{3}{2}  }_{2,n,<}}
			+c^{-1}\norm{B_0}_{\dot B^{s+\alpha+ \frac{2}{m}- \frac{1}{2}-\frac 2q}_{2,q,<}}
			\\
			&\quad +c^{2\left(\frac 1p + \frac{1}{m}-\frac 1q\right)-1}
			\norm{P(u\times B) }_{L_t^{ \left(\frac{1}{m}+ \frac{1}{p} \right)^{-1}} \left( \dot B^{s+\alpha+ \frac{2}{m}- \frac{3}{2}}_{2,n} \right)}  ,
		\end{aligned}
	\end{equation*}
	for any $1< \left( \frac{1}{m} + \frac{1}{p} \right)^{-1} \leq q< \infty$ and $1\leq n\leq\infty$, as well as
	\begin{equation*}
		\begin{aligned}
			\norm{\mathds{1}_{\{2^j<\frac{\sigma c}2\}}2^{j \left( s+\alpha  - \frac{1}{2}+ \frac 2q-\frac 2p  \right)}
			\norm{\Delta_j B}_{L^2_x}}_{L^q_t\ell^1_j}
			&\lesssim
			c^{-1}\norm{E_0}_{\dot B^{s+\alpha+\frac{1}{2} -\frac 2p}_{2,q,<}}
			+\norm{B_0}_{\dot B^{s+\alpha-\frac{1}{2}-\frac 2p}_{2,q,<}}
			\\
			&\quad+\norm{P(u\times B)}_{L_t^{ \left(\frac{1}{m}+ \frac{1}{p} \right)^{-1}} \left( \dot B^{s+\alpha+ \frac{2}{m}- \frac{3}{2}}_{2, \infty} \right)},
		\end{aligned}
	\end{equation*}
	for any $1<  \left( \frac{1}{m} + \frac{1}{p} \right)^{-1}< q< \infty$.

	All in all, by combining the preceding estimates, we arrive at the conclusion that
	\begin{eqnarray*}
		\norm{E}_{L^q_t\dot B^{  s+\alpha+ \frac{2}{m}- \frac{3}{2} }_{2,n,<}}
		&\lesssim &
		c^{-\frac 2q}\norm{E_0}_{\dot B^{ s+\alpha+ \frac{2}{m}- \frac{3}{2} }_{2,n,<}}
		+c^{-1}\norm{B_0}_{\dot B^{ s+\alpha+ \frac{2}{m}- \frac{1}{2}- \frac{2}{q} }_{2,q,<}}\\
		&&+c^{2\left(\frac 1p + \frac{1}{m}-\frac 1q\right)-1} \norm{P(u\times B)}_{L_t^{ \left(\frac{1}{m}+ \frac{1}{p} \right)^{-1}}\dot B^{ s+\alpha+ \frac{2}{m}- \frac{3}{2} }_{2,n}},
	\end{eqnarray*}
	for any $1<  \left( \frac{1}{m} + \frac{1}{p} \right)^{-1}\leq q< \infty$ and $1\leq n\leq\infty$, as well as
	\begin{equation*}
			\norm{B}_{L^q_t\dot B^{s+\alpha  - \frac{1}{2}+ \frac 2q-\frac 2p }_{2,1,<}}
			\lesssim
			c^{-1}\norm{E_0}_{\dot B^{s+\alpha+\frac{1}{2} -\frac 2p}_{2,q,<}}
			+\norm{B_0}_{\dot B^{s+\alpha-\frac{1}{2}-\frac 2p}_{2,q,<}}
			+\norm{P(u\times B)}_{L_t^{ \left(\frac{1}{m}+ \frac{1}{p} \right)^{-1}} \left( \dot B^{s+\alpha+ \frac{2}{m}- \frac{3}{2}}_{2, \infty} \right)},
	\end{equation*}
	for any $1<  \left( \frac{1}{m} + \frac{1}{p} \right)^{-1}< q< \infty$.
	Therefore, applying the product estimates from Lemma \ref{paradifferential:1}   concludes the proof of \eqref{frequencies:1} and \eqref{frequencies:1-B***}.
	
	As for the case $ 1< \left( \frac{1}{m} + \frac{1}{p}\right)^{-1}= q < \infty   ,$ we apply  the third estimate from Lemma \ref{cor:parabolic:maxwell} instead of the second one to infer that 
	\begin{equation*}
		\begin{aligned}
			\norm{\mathds{1}_{\{2^j<\frac{\sigma c}2\}}2^{j \left( s+\alpha  - \frac{1}{2}+ \frac 2q-\frac 2p  \right)}
			\norm{\Delta_j B}_{L^2_x}}_{L^q_t\ell^n_j}
			&\lesssim
			c^{-1}\norm{E_0}_{\dot B^{s+\alpha+\frac{1}{2} -\frac 2p}_{2,q,<}}
			+\norm{B_0}_{\dot B^{s+\alpha-\frac{1}{2}-\frac 2p}_{2,q,<}}
			\\
			&\quad+\norm{P(u\times B)}_{L_t^{ \left(\frac{1}{m}+ \frac{1}{p} \right)^{-1}} \left( \dot B^{s+\alpha+ \frac{2}{m}- \frac{3}{2}}_{2, n} \right)}.
		\end{aligned}
	\end{equation*}
	Hence,   combining the latter estimate with \eqref{R0}, we obtain that
	\begin{equation*}
			\norm{B}_{L^q_t\dot B^{s+\alpha  - \frac{1}{2}+ \frac 2q-\frac 2p }_{2,n,<}}
			\lesssim
			c^{-1}\norm{E_0}_{\dot B^{s+\alpha+\frac{1}{2}-\frac 2p}_{2,q,<}}
			+\norm{B_0}_{\dot B^{s+\alpha-\frac{1}{2}-\frac 2p}_{2,q,<}}
			+\norm{P(u\times B)}_{L_t^{ \left(\frac{1}{m}+ \frac{1}{p} \right)^{-1}} \left( \dot B^{s+\alpha+ \frac{2}{m}- \frac{3}{2}}_{2, n} \right)}.
	\end{equation*}
	 Our claim \eqref{frequencies:1b} follows then by employing Lemma \ref{paradifferential:1}, again.
	 
	  At last, for the endpoint case $s=\frac{5}{2}$,  we deduce from the previous cases above, for any $n\in [1,\infty]$, that
	 \begin{eqnarray*}
		\norm{E}_{L^q_t\dot B^{  1+\alpha+ \frac{2}{m}  }_{2,n,<}}
		&\lesssim &
		c^{-\frac 2q}\norm{E_0}_{\dot B^{ 1+\alpha+ \frac{2}{m}  }_{2,n,<}}
		+c^{-1}\norm{B_0}_{\dot B^{ 2 + \alpha+ \frac{2}{m} - \frac{2}{q} }_{2,q,<}}\\
		&&+c^{2\left(\frac 1p + \frac{1}{m}-\frac 1q\right)-1} \norm{P(u\times B)}_{L_t^{ \left(\frac{1}{m}+ \frac{1}{p} \right)^{-1}}\dot B^{ 1+\alpha+ \frac{2}{m}  }_{2,n}},
	\end{eqnarray*}
	and
	 \begin{equation*}
			\norm{B}_{L^q_t\dot B^{2+\alpha   + \frac 2q-\frac 2p }_{2,n,<}}
			\lesssim
			c^{-1}\norm{E_0}_{\dot B^{3+\alpha -\frac 2p}_{2,q,<}}
			+\norm{B_0}_{\dot B^{2+\alpha -\frac 2p}_{2,q,<}}
			+\norm{P(u\times B)}_{L_t^{ \left(\frac{1}{m}+ \frac{1}{p} \right)^{-1}} \left( \dot B^{1+\alpha+ \frac{2}{m} }_{2, n} \right)},
	\end{equation*}
	as long as 
	$1<  \left( \frac{1}{m} + \frac{1}{p} \right)^{-1} \leq q< \infty$.
	Therefore, applying Lemma \ref{paradifferential:1}, again, to estimate the source term concludes the proof.
\end{proof}

\begin{rem}
Observe that combining \eqref{E_>:5/2} and \eqref{frequencies:2}, with  the  choice of parameters $$ q=p=2, \quad m=4, \quad \alpha=n= 1,$$
and    employing the interpolation inequality
$$ \norm u_{ L^4_t \dot{B}^\frac{3}{2}_{2,1}} \lesssim \norm u_{ L^\infty_t \dot{H}^1 \cap L^2 _t\dot{H}^2}$$
    produces the useful bound
\begin{equation}\label{E-AAA}
\begin{aligned}
\norm E_{L^2_t\dot{B}^\frac{5}{2}_{2,1} } &\lesssim c^{-1}  \norm {(E_0,B_0)}_{\dot{B}^\frac{5}{2}_{2,1}}  
 +  c^{-\frac{1}{2}}\norm{    u   }_{L  ^{\infty }_t\dot{H}^1  \cap L  ^{2 }_t\dot{H}^2 } \left( \norm{   B  }_{\widetilde{L}   ^{2 }_t\dot{B}^{\frac{5}{2}}_{2,1,>    } } +  \norm{   B }_{L^2  _t\dot{B}^{\frac{5}{2}}_{2,1  ,< }  } \right).
\end{aligned}
\end{equation}
The latter control of the electric field will come in handy, later on.
\end{rem}

\subsubsection{Persistence of regularity}

Observe that the case $q=\infty$ is missing in      the bounds from  Lemma \ref{low:freq:estimates}, above. Although it is possible to extend the results from Lemma \ref{low:freq:estimates} to cover that case (at the cost of more restrictive assumptions on the third summability index of Besov norms of the initial data), it is simpler to establish these missing bounds by utilizing   an elementary energy estimate  for   Maxwell's equations, which is the content of the next lemma.

 \begin{lem}\label{lemma.reg-persistence}Let  $T\in (0,\infty]$ and  $(E,B)$ be a smooth axisymmetric solution to \eqref{Maxwell:system:*} on $[0,T)$, for some divergence-free initial data $(E_0,B_0)$ and   vector field $u$. Assume further that $ E$ and $u$ are both   without swirl and that $B$ is a   vector field with pure swirl. 
 
  Then, for any  $s \in (-\frac{3}{2},  \frac{5}{2})$, $\epsilon>0 $ with $s+ \epsilon< \frac{5}{2}$ and for all $p,q\in [2,\infty]$ with $\frac{1}{2}= \frac{1}{p} + \frac{1}{q}$, it holds that
 \begin{equation*}
\norm { (E,B) }_{L^\infty_t \dot{H}^s}  +c \norm {  E}_{L^2_t  \dot{H}^s }  \lesssim      \norm { (E_0,B_0)}_{\dot{H}^s}   + \norm u_{ L^q_t\dot{H}^ {\frac{3}{2}- \epsilon} }\norm {B}_{L^p_t\dot{H}^{s+ \epsilon }  },
\end{equation*}
 on the time interval $[0,T)$.
Moreover, in the endpoint case $\epsilon=0$, we have that
 \begin{equation*}
\norm { (E,B) }_{L^\infty_t \dot{H}^s}  +c \norm {  E}_{L^2_t \dot{H}^s  }  \lesssim      \norm { (E_0,B_0)}_{\dot{H}^s}   + \norm u_{ L^q_t\dot{B}^{\frac{3}{2}}_{2,1}  }\norm {B}_{L^p_t\dot{H}^{s  }  } ,
\end{equation*}  
on the time interval $[0,T)$.
 \end{lem}

\begin{proof}
We begin with localizing \eqref{Maxwell:system:*} in frequencies by applying $\Delta_j$, for $ j\in \mathbb{Z}$. Then, by an $L^2$ energy estimate, we find, for any $j\in \mathbb{Z}$, that
\begin{equation*}
\begin{aligned}
\frac{1}{2c}\norm {\Delta_j(E,B)(t)}_{L^2}^2 + \sigma c \norm {\Delta_j E}_{L^2_t L^2 }^2 &\leq    \frac{1}{2c}\norm {\Delta_j(E_0,B_0)}_{L^2}^2 + \sigma \norm {\Delta_j P(u\times B)}_{L^2_t L^2 } \norm {\Delta_j E}_{L^2_t L^2 }\\
& \leq   \frac{1}{2c}\norm {\Delta_j(E_0,B_0)}_{L^2}^2 + \frac{\sigma}{2c} \norm {\Delta_j P(u\times B)}_{L^2_t L^2 } ^2 \\
& \quad+ \frac{\sigma c}{2} \norm {\Delta_j E}_{L^2_t L^2 }^2.
\end{aligned}
\end{equation*} 
Therefore, it follows that
\begin{equation*}
\norm { (E,B)(t)}_{\dot{H}^s}^2 + \sigma c^2 \norm {  E}_{L^2_t \dot{H}^s }^2 \leq     \norm { (E_0,B_0)}_{\dot{H}^s}^2 + \sigma \norm {  P(u\times B)}_{L^2_t  \dot{H}^s }^2,
\end{equation*}
for any $s \in (-\frac{3}{2}, \frac{5}{2})$.

Finally, we conclude the proof    by employing  Lemma \ref{paradifferential:1} to control the source term in the right-hand side, above.
\end{proof}

\section{Asymptotic analysis of electromagnetic fields}
\label{section:asymptotics}

	In this section, we are going to make \eqref{Claim:01} and \eqref{Claim:02} precise. In particular, we are going to study the convergence
	$$\frac{1}{c}\partial_tE \to 0, $$
	as  $c\to \infty$.
	
	  For  simplicity, we will drop   $\sigma$ from the equations \eqref{Maxwell:system:*2} by fixing its value $\sigma=1$, and we emphasize that the analysis we perform here holds for any non-negative value of that parameter.

	\subsection{Asymptotic analysis of Amp\`ere's equation}
	
	From \eqref{Maxwell:system:*2}, observe, at least formally, that $E$ vanishes when $c$ goes to infinity. Moreover,  Amp\`ere's equation allows us to obtain
$$  j \rightarrow \nabla \times B, $$
as $ c \to  \infty$, in some suitable weak sense.

	 In the next lemma, we establish a more precise   description of the preceding convergence in adequate functional spaces. This step is crucial in the proof of Theorem \ref{Thm:1} and will come in handy in Section \ref{section:closing:ES}.
	  
	 	 \begin{prop}\label{decay-E}
	Let $T\in (0,\infty]$ and $(u,E,B)$ be smooth solution to the Navier--Stokes--Maxwell equations   \eqref{Maxwell:system:*2}, defined on $[0,T)$, where $u$ and $E$ have no swirl and $B$ has pure swirl.
	
	 Then, on the time interval $[0,T)$, for all $s\in [0, \frac{1}{2}]$, $p\in [2,\infty]$ and for any $c>0$, it holds that
	\begin{equation*}
	\norm {\frac{1}{c}\partial_t E }_{L^p_t \dot{H}^s}=\norm {j - \nabla \times B }_{L^p_t \dot{H}^s}\lesssim
	c^ {-\frac{2}{p}}\norm { \nabla \times B_0 - j_0 }_{\dot{H}^{s}}
	+c^ {-\frac{2}{p}}\norm { E_0 }_{\dot{H}^{s+1}}
	+   c^{-\left(\frac{2}{p}+1\right)} \mathcal{A}_{s} (u,E,B ),
	\end{equation*} 
	where
  
	 \begin{equation*}
	 \begin{aligned}
\mathcal{A}_s(u,E,B ) &\bydef  
    \Big(       \norm {   u  }_{ L^ \infty _t  \dot{H}^1    \cap L^ 2_t \dot{H} ^2  } +  \mathcal{E}_0 \norm {   B}_{L^\infty_t \dot{H}^ \frac{3}{2}  }       \Big)    \Big(   \norm {   B}_{L^\infty_t \dot{H}^{ s+ \frac{3}{2}} }   +   c \norm {E}_{L^2_t\dot{H}^{s+\frac{3}{2}}}\Big) \\
    & \quad +   \norm {   u  }_{ L^ \infty _t  L^2  }^{\frac{1}{2}-s} \norm {   u  }_{ L^ \infty _t  \dot{H} ^1  }^{s+ \frac{1}{2}}  \norm {   u  }_{   L^ 2_t \dot{H} ^{2 }  }     \norm {   B}_{L^\infty_t \dot{H }^{\frac{3}{2} }  } .
    \end{aligned}
	 \end{equation*}
	 \end{prop}

	\begin{proof} 
The proof relies on a key idea from \cite[Proposition 3.3]{ah2}, which we adapt to the three-dimensional setting. 

We begin by       applying a time derivative  to Amp\`ere's equation to obtain the following damped wave equation for $ E$
\begin{equation*}
\frac{1}{c } \partial_{tt}E - c \Delta E+ c \, \partial_t E = -   \partial_t P (u\times B),
\end{equation*}
where we have   used Faraday's equation and Ohm's law, as well.

Then, we localize in frequencies by applying $\Delta_j$, for $j\in \mathbb{Z}$, and we perform an  $L^2$  energy estimate followed by H\"older's inequality to find, for all $t\in [0,T)$, that
\begin{equation*}
\begin{aligned}
\frac{1}{2} \frac{d}{dt }\Bigg( \frac{1}{c^4} \norm {\Delta_j \partial_t E (t) }_{L^2}^2  &+\frac{1}{c^2} \norm {\Delta_j \nabla  E(t) }_{L^2}^2 \Bigg)   + \frac{1}{c^2}\norm {\Delta_j \partial_t E(t)}_{L^2}^2 \\
& \qquad\qquad \leq \frac{1}{c^3} \left| \int_{\mathbb{R}^3}   \Delta_j \partial_t P (u\times B) \Delta_j \partial_t E(t,x) dx \right|\\
& \qquad \qquad \leq \frac{1}{c^3} \norm{   \Delta_j \partial_t P (u\times B)(t)}_{L^2} \norm{ \Delta_j \partial_t E(t)  }_{L^2}\\
&  \qquad \qquad\leq  \frac{1}{2 c^4} \norm{   \Delta_j \partial_t P (u\times B)(t)}_{L^2} ^2 + \frac{1}{2c^2} \norm{ \Delta_j \partial_t E(t)  }_{L^2}^2.
\end{aligned}
\end{equation*}

Therefore, we find,  for any $s\in \mathbb{R}$, that
\begin{equation*}
 \frac{d}{dt }\left( \frac{1}{c^4} \norm {  \partial_t E(t) }_{ \dot{H}^s}^2  +\frac{1}{c^2} \norm {    E (t) }_{\dot{H}^{s+ 1}}^2 \right) + \frac{1}{c^2}\norm {  \partial_t E(t)}_{\dot{H}^s}^2  \leq   \frac{1}{  c^4} \norm{   \partial_t P (u\times B)(t)}_{\dot{H}^s} ^2  .
\end{equation*}
Consequently, employing H\"older's inequality followed by Young's inequality for products, we obtain that
\begin{equation}\label{Energy***}
	\begin{aligned}
	  	\frac{1}{c^{2-\frac 2p}}\norm {  \partial_t E }_{L^p_t\dot{H}^s}
		&\lesssim
		\frac{1}{c^2}\norm {  \partial_t E }_{L^\infty_t\dot{H}^s}
	  	+
	    \frac{1}{c }\norm {  \partial_t E }_{L^2_t\dot{H}^s}
		\\
		&\lesssim   \frac{1}{c} \Big( \norm {\nabla \times B_0 - j_0}_{\dot{H}^s} + \norm {E_0}_{\dot{H}^{s+1}}  \Big) +  \frac{1}{  c^2} \norm{   \partial_t P (u\times B) }_{L^2_t\dot{H}^s}   ,
	\end{aligned}
\end{equation}
for every $p\in [2,\infty]$, where we have used Amp\`ere's equation   
 to express the initial data of $   \frac{1}{c }\partial_t E$ in terms of $B_0$ and $j_0$.
 
Hence, we are now left to control $ \partial_t P (u\times B)$. To this end, we shall first   transform  the time derivative into spatial derivatives by using the momentum and Faraday's equations from \eqref{Maxwell:system:*2}. Accordingly,  we obtain that
$$\partial_t P (u\times B)  =    P \left(\partial_t u  \times B  \right)+P \left( u  \times \partial_t B  \right) = - \sum_{i=1}^4 \mathcal{I}_i,$$
where
\begin{equation*}
	\begin{aligned}
		\mathcal{I}_1  &\bydef   P \Big(P\left( u \cdot \nabla u  \right) \times B  \Big),
		&
		\mathcal{I}_2 &\bydef  -P \Big(\Delta u  \times B  \Big),
		\\
		\mathcal{I}_3  &\bydef  -P \Big(P\left( j \times B \right) \times B  \Big) ,
		&
		\mathcal{I}_4  &\bydef   cP \Big(u   \times  (\nabla \times E )\Big).
	\end{aligned}
\end{equation*} 
We are now going to utilize \eqref{paraproduct:1} from  Lemma \ref{paradifferential:1} to estimate each term separately.

For $\mathcal{I}_1 $, we find that
\begin{equation*}
\begin{aligned}
  \norm {\mathcal{I}_1 }_{L^2_t   \dot{H}^{s} }   
   &\lesssim   \norm {u \cdot \nabla u  }_{ L^ 2 _t \dot{H}^ {s} }   \norm {  B}_{L^\infty_t \dot{H}^{ \frac{3}{2} }}\\
   &\lesssim    \norm {u  }_{ L^ \infty _t \dot{H}^{s+\frac{1}{2}} }  \norm { \nabla u  }_{ L^ 2_t \dot{H}^1  }\norm {   B}_{L^\infty_t \dot{H}^{ \frac{3}{2} }}.
\end{aligned}
\end{equation*}
Therefore, as long as $ s\in [0,\frac{1}{2}] $, we obtain, by interpolation, that
 \begin{equation*}
\begin{aligned}
  \norm {\mathcal{I}_1 }_{L^2_t   \dot{H}^{s} }   
 &\lesssim     \norm {   u  }_{ L^ \infty _t  L^2  }^{\frac{1}{2}-s} \norm {   u  }_{ L^ \infty _t  \dot{H} ^1  }^{s+ \frac{1}{2}}  \norm {   u  }_{   L^ 2_t \dot{H} ^{2 }  }     \norm {   B}_{L^\infty_t \dot{H }^{\frac{3}{2} }  }.
   \end{aligned}
\end{equation*}
As for $\mathcal{I}_2$, we  employ  Lemma \ref{paradifferential:1}, again, to find, for any $s\in [0,\frac{1}{2}]$, that
 \begin{equation*}
\begin{aligned}
  \norm {\mathcal{I}_2 }_{L^2_t   \dot{H}^{s} }   
   &\lesssim   \norm {\Delta u  }_{ L^ 2_t L^2 }  \norm {   B}_{L^\infty_t \dot{H}^{s+\frac{3}{2}} }\\
   &\lesssim     \norm {   u  }_{  L^ 2_t \dot{H}^2  }  \norm {   B}_{L^\infty_t \dot{H}^{s+\frac{3}{2}} }.
\end{aligned}
\end{equation*}
In order to estimate  $\mathcal{I}_3$, we    utilize Lemma \ref{paradifferential:1} twice to obtain, for any $s\in [0,\frac{1}{2}]$, that
 \begin{equation*}
\begin{aligned}
  \norm {\mathcal{I}_3 }_{L^2_t   \dot{H}^{s} }   
   &\lesssim   \norm {P \big(j      \times  B \big)  }_{ L^ 2_t L^2 }  \norm { B}_{L^\infty_t \dot{H}^{s+ \frac{3}{2}} }\\ 
    &\lesssim   \norm { j }_{ L^ 2_t L^2 }\norm{   B  }_{ L^ \infty _t \dot{H}^{\frac{3}{2} }  }  \norm {   B}_{L^\infty_t \dot{H}^{s+ \frac{3}{2}}}\\
   &\lesssim    \mathcal{E}_0  \norm {   B}_{L^\infty_t \dot{H}^\frac{3}{2}  }  \norm {   B}_{L^\infty_t \dot{H}^{s+ \frac{3}{2}}}  ,
\end{aligned}
\end{equation*}
where we have used the energy estimate \eqref{L^2-Energy} in the last inequality.
Finally, the control of the last term $\mathcal{I}_4$  is achieved by classical product laws, which give that
 \begin{equation*}
\begin{aligned}
  \norm {\mathcal{I}_4 }_{L^2_t   \dot{H}^{s} }    \lesssim    \norm  u_{L^\infty_t\dot{H}^1} \norm {cE}_{L^2_t\dot{H}^{s+\frac{3}{2}}} .
\end{aligned}
\end{equation*}

All in all, gathering the previous bounds yields   
\begin{equation}\label{u-times-B}
\norm{\partial_t  (u \times B)}_{ L^2_t\dot{H}^s} \lesssim \mathcal{A}_s(u,E,B),
\end{equation}
where $\mathcal{A}_s(u,E,B)$  is defined in the statement of the proposition, above.
The proof is then concluded by incorporating \eqref{u-times-B} into \eqref{Energy***}.
	\end{proof}

	\subsection{Almost-parabolic estimates on  the   magnetic field} \label{section:AGE}

	When considering the limiting   system \eqref{MHD*}, one can show, with standard energy estimates, that the quantities $ B$ and $\Gamma = \frac{B_\theta}{r}$ are globally controlled in $L^2_t\dot{H}^1$ and $L^\infty_t L^p \cap  L^2_t\dot{H}^1$, for all $p\in [1,\infty]$, respectively.
	
	Here, we establish an asymptotic version of these bounds  for \eqref{Maxwell:system:*2}, thereby  justifying \eqref{Claim:02}. 
	The key observation in the proof below consists in considering   the term $ \frac{1}{c^2} \partial_{tt} \Gamma $ in the  equation 
	 \begin{equation}\label{Gamma-equa***}
  \frac{1}{c^2} \partial_{tt} \Gamma + \partial_t \Gamma + u\cdot\nabla \Gamma -\big(\Delta+ \frac{\partial_r}{r} \big) \Gamma=0
  \end{equation} 
  as an error, for large values of $c$.
  Accordingly, one should treat \eqref{Gamma-equa***} as a parabolic equation with a vanishing source term, as $c\to \infty$. A more precise statement of that is given in the next proposition

 \begin{prop}\label{prop:almost-parabolic-ES}  Let $T\in \mathbb{R}^+\cup\{\infty\}$ and $(u,E,B) $  be a  smooth axisymmetric solution to \eqref{Maxwell:system:*2}, defined on $[0,T)$, where $u$ and $E$ have no swirl and $B$ has pure swirl.
 
  Then, it holds that
\begin{equation}\label{B-almost.GE}
\norm{ B}_{L^2_t\dot{H}^1 }\leq \mathcal{E}_0 + \norm{\frac{1}{c}\partial_t E}_{L^2_tL^2 }.
\end{equation}
Moreover,    $\Gamma  $ enjoys the bounds
\begin{equation}\label{Gamma-almost.GE-1}
\norm{ \Gamma}_{L^2_t\dot{H}^1 }\lesssim  \norm{\Gamma_0}_{L^2}    +c^{-1} \norm {  E_0 }_{\dot{H}^2}   +  c^{-1} \left(     \norm {B}_{L^\infty _t\dot{H}^{2  }   } + c\norm { E}_{L^2_t\dot{H}^2 } \right),
  \end{equation}
and, for any $p\in [2,\infty)$,
\begin{equation}\label{Gamma-almost.GE-2}
\norm{ \Gamma}_{L^\infty_tL^p }\lesssim  \norm { \Gamma_0}_{ L^p   }     +c^{-1} \norm {E_0}_{\dot{W}^{2,p}}   +  \norm { E}_{L^2_t\dot{W}^{2,p} },
\end{equation}
where all the time norms above are taken over the whole interval $[0,T).$
\end{prop}

\begin{rem}
	Due to Lemma \ref{lemma.reg-persistence} and Proposition \ref{decay-E}, observe that  the last   terms   in the right-hand side of \eqref{B-almost.GE} and  \eqref{Gamma-almost.GE-1}    can be seen as   errors for large values of $c$.

 Likewise,    by virtue of    \eqref{E-AAA}    and  the embedding
$$ \dot{B}^\frac{5}{2}_{2,1} \hookrightarrow   \dot{W}^{2,3}(\mathbb{R}^3),$$
note that \eqref{Gamma-almost.GE-2} implies, for $p=3$,   that 
\begin{equation*} 
		\norm{ \Gamma}_{L^\infty_tL^3 }
		\lesssim  \norm { \Gamma_0}_{ L^3   }
		+ c^{-1}  \norm {(E_0,B_0)}_{\dot{B}^\frac{5}{2}_{2,1}}  
		+  c^{-\frac{1}{2}}\norm{    u   }_{L  ^{\infty }_t\dot{H}^1  \cap L  ^{2 }_t\dot{H}^2 } \left( \norm{   B  }_{\widetilde{L}   ^{2 }_t\dot{B}^{\frac{5}{2}}_{2,1,>    } } +  \norm{   B }_{L^2  _t\dot{B}^{\frac{5}{2}}_{2,1  ,< }  } \right).
\end{equation*}
This bound will come in handy, later on.
\end{rem}

\begin{proof}
	It is readily seen that, by virtue of   the energy inequality \eqref{L^2-Energy},   Amp\`ere's equation from \eqref{Maxwell:system:*2} entails that
\begin{equation*}
	\norm{ B}_{L^2_t \dot{H}^1 } \leq \norm {j}_{L^2_t L^2 } +  \norm{\frac{1}{c}\partial_t E}_{L^2_t L^2 }
	\leq   \mathcal{E}_0 + \norm{\frac{1}{c}\partial_t E}_{L^2_t L^2 },
\end{equation*} 
thereby establishing \eqref{B-almost.GE}.

We now focus   on the  proof of \eqref{Gamma-almost.GE-1}  and   \eqref{Gamma-almost.GE-2}. To that end,    multiplying    \eqref{Gamma-equa***}     by $\Gamma|\Gamma|^{p-2}$,     integrating with respect to   space variables and   using the identities
\begin{equation*}
	- \int_{\mathbb{R}^3}(\Delta \Gamma) \Gamma|\Gamma|^{p-2}dx
	-\int_{\mathbb{R}^3} \frac{\partial_r\Gamma}{r} \Gamma|\Gamma|^{p-2}dx  =\underbrace{ \frac{4(p-1)}{p^2} \norm {\nabla (|\Gamma(t)|^{\frac{p}{2}})}_{L^2}^2  
+  \frac{2\pi}{p} \int_\mathbb{R}|\Gamma(t,r=0,z)|^p dz}_{\bydef X_p(t)}
\end{equation*}
and
\begin{equation*}
\int_{\mathbb{R}^3}(\partial_{tt} \Gamma) \Gamma|\Gamma|^{p-2} dx=\partial_{tt} \Big(\underbrace{ \frac{ 1}{p}  \norm {  \Gamma (t)}_{L^p}^p}_{\bydef  Y _p(t)} \Big)- \underbrace{ (p-1) \int_{\mathbb{R}^3}|\partial_t \Gamma|^2 |\Gamma|^{p-2} dx}_{\bydef    A_p(t)} ,
\end{equation*} 
yields, for all $t\in [0,T)$, that
\begin{equation}\label{Y''-equa}
\frac{1}{c^2}Y''_p(t) + Y'_p(t) + X_p(t) = \frac{1}{c^2}A_p(t).
\end{equation}
Therefore, integrating in time, we obtain that  
\begin{equation}\label{Y-equa}
Y_p(t) + \int_0^t X_p(\tau) d\tau = Y_p(0) + \frac{1}{c^2} Y_p'(0) - \frac{1}{c^2}Y_p'(t) + \frac{1}{c^2}\int_0^t A_p(\tau)d\tau.
\end{equation}

Now, we need  to take care of the term $ - \frac{1}{c^2}Y_p'(t)$, above. To that end, observing that \eqref{Y''-equa} can be rewritten as 
\begin{equation*}
\frac{1}{c^2}\frac{d}{dt}\Big( Y_p'(t)e^{c^2t}\Big)=  \Big(\frac{1}{c^2}A_p(t)-X_p(t), \Big)e^{c^2t},
\end{equation*}
which we then integrate with respect to the time variable, we find that
\begin{equation*}
-\frac{1}{c^2} Y'_p(t) = -\frac 1{c^2}Y_p'(0) e^{-c^2t} + \int_0^t e^{-c^2(t-\tau)} \Big(X_p(\tau)-\frac{1}{c^2}A_p(\tau) \Big) d\tau.
\end{equation*}
Hence, plugging the latter identity into \eqref{Y-equa} yields that 
\begin{equation}\label{Y-equa:2}
\begin{aligned}
Y_p(t) + \int_0^t X_p(\tau) d\tau &=  Y_p(0) + \frac{1-e^{-c^2t}}{c^2} Y_p'(0)   +  \int_0^t e^{-c^2(t-\tau)}  X_p(\tau)  d\tau \\
&  \quad+\frac{1}{c^2} \int_0^t A_p(\tau)d\tau  - \frac{1}{c^2}\int_0^t e^{-c^2(t-\tau)}A_p(\tau)  d\tau.
\end{aligned}
\end{equation}

Next, utilizing Faraday's equation from \eqref{Maxwell:system:*2}, we write, for any $t\in (0,T)$, that
\begin{equation}\label{dt-Gamma}
 \left(\partial_t\Gamma \right)  = \left(   \frac{\partial_tB }{r} \right)  \cdot e_\theta
 = -c \left(   \frac{ \nabla \times E }{r} \right)  \cdot e_\theta,
\end{equation} 
whereby
$$ \left(\partial_t\Gamma \right)|_{t=0} = -c \left(   \frac{ \nabla \times E_0 }{r} \right)  \cdot e_\theta.  $$  
 Accordingly, it follows that 
\begin{equation*}
	\begin{aligned}
		\frac{1}{c^2} |Y'_p(0)|=\frac{1}{c^2}\left| \Big( \int_{\mathbb{R}^3} (\partial_t \Gamma)\Gamma |\Gamma|^{p-2}dx\Big)_{|_{t=0}} \right|
		&\leq\frac{1}{c } \int_{\mathbb{R}^3} \frac{|\nabla \times E_0|}{r} |\Gamma_0|^{p-1}dx
		\\
         & \quad \leq  \frac{1}{c} \norm {\frac{ \nabla \times E_0}{r}}_{L^p}\norm {\Gamma_0}_{L^{p}}^{p-1}.
	\end{aligned}
\end{equation*} 
Thus, recalling the inequality
\begin{equation}\label{conv:inequa}
ab\leq \varepsilon^\alpha \frac{a^\alpha}{\alpha} + \varepsilon^{-\alpha'} \frac{b^{\alpha'}}{\alpha'}, 
\end{equation}
for any $\varepsilon>0$ and $\alpha \in (1,\infty)$,
 where   $\alpha^{\prime}$ denotes the conjugate of $\alpha $,
   and  by virtue of Lemma \ref{lemma-f/x:2}, we end up with  
   \begin{equation}\label{Y'-bound}
   \begin{aligned}
    \frac{1}{c^2}  |Y'_p(0)|  & \leq      \left(   \left(  \frac{1}{c } \norm {\frac{ \nabla \times E_0}{r}}_{L^p}\right)^{\frac{1}{p}} \norm {\Gamma_0}_{L^{p}}^{\frac{1}{p^\prime} }   \right)^ p \\
   & \leq       \left(  \frac{1}{pc}    \norm {\frac{ \nabla \times E_0}{r}}_{L^p} + \frac{1}{p^{\prime}} \norm {\Gamma_0}_{L^{p}}    \right)^ p\\
 & \lesssim      \Big(  \frac{1}{c} \norm { E_0}_{\dot{W}^{2,p}} + \norm {\Gamma_0}_{L^{p}}   \Big)^ p. 
   \end{aligned}
   \end{equation} 
Consequently, for $p=2$, we find that \eqref{Y-equa:2} entails that
\begin{equation*}
\norm {\Gamma(t)}_{  L^2 }^2   \lesssim    \norm{\Gamma_0}_{L^2}^2 +    c^{-2} \left(   \norm { E_0}_{\dot{H}^2}  +     \norm {\partial_t \Gamma}_{L^2_t L^2 } \right)^2
\end{equation*}
and
\begin{equation*}
  \norm \Gamma_{L^2_t \dot{H}^1 }^2 \lesssim  \norm{\Gamma_0}_{L^2}^2 +    c^{-2} \left(   \norm { E_0}_{\dot{H}^2}  +   \norm \Gamma_{L^\infty _t  \dot{H}^1 }  + \norm {\partial_t \Gamma}_{L^2_t L^2 } \right)^2.
\end{equation*}
Finally,   employing \eqref{dt-Gamma} 
and  Lemma \ref{lemma-f/x:2} in the last bound leads to \eqref{Gamma-almost.GE-1}, and \eqref{Gamma-almost.GE-2} in the case $p=2$.

Then, in order to deal with the range $p\in (2,\infty)$,
 we  deduce from \eqref{Y-equa:2} that
  \begin{equation*}
\norm { \Gamma}_{L^\infty_t L^p }^p
\lesssim  \norm { \Gamma_0}_{L^\infty_t(L^p)}^p + \frac{1}{c^2} |Y'_p(0)| + \frac{1}{c^2}  \int_0^t \norm {\partial_t\Gamma(\tau)}_{L^p}^2 \norm {\Gamma(\tau )}_{L^p}^{p-2}d\tau .
   \end{equation*} 
   Therefore, employing \eqref{Y'-bound} to control $Y'_p(0) $ yields 
    \begin{equation*}
	   \norm { \Gamma}_{L^\infty_t L^p }^p
	   \lesssim   \norm { \Gamma_0}_{L^\infty_t L^p }^p
	   +    \Big(  \frac{1}{ c} \norm { E_0}_{\dot{W}^{2,p}} + \norm {\Gamma_0}_{L^{p}}   \Big)^ p    + \frac{1}{c^2}    \norm {\partial_t\Gamma }_{L^2_t L^p }^2 \norm {\Gamma }_{L^\infty_t L^p }^{p-2}.
   \end{equation*} 
After that, we utilize \eqref{conv:inequa}
 to obtain that
   \begin{equation*}
\norm { \Gamma}_{L^\infty_t L^p }^p \lesssim  \norm { \Gamma_0}_{L^\infty_t L^p }^p + \Big(  \frac{1}{ c} \norm { E_0}_{\dot{W}^{2,p}} +  \norm {\Gamma_0}_{L^{p}}   \Big)^ p
+  \frac{1}{c^p}     \norm {\partial_t\Gamma }_{L^2_t L^p }^p.
   \end{equation*} 
At last,     by using \eqref{dt-Gamma}   and Lemma \ref{lemma-f/x:2}, again, we arrive at the desired bound \eqref{Gamma-almost.GE-2}.  This completes the proof of the proposition.
\end{proof}

	\section{Closing the estimates and proof  of Theorem \ref{Thm:1}}

	\label{section:closing:ES}

	 We are now    ready to establish the final global bounds which   lead  to  the existence of global solutions of the Navier--Stokes--Maxwell equations \eqref{Maxwell:system:*2}. This is going to be done, first, by gathering all the estimates from the previous sections to produce a nonlinear energy estimate. Then, under adequate assumptions in the regime $c\to \infty$, the nonlinear energy bound will allow us to deduce the desired global control, uniformly with respect to the speed of light.
	 
	  Throughout this section, for the sake of simplicity and clarity,    the subscript ``$t$'' that appears in     time-Lebesgue spaces $L^p_t$ should be understood as the endpoint of the time interval, i.e., for example, we will be using the notation
	 $$ \norm{f}_{L^p_t\dot{H}^s} \bydef \norm {f}_{L^p\left([0,t); \dot{H}^s(\mathbb{R}^3)\right)}.$$
	 Accordingly, all the quantities involving time-Lebesgue norms are continuous functions on the real half-line $[0,\infty)$, for all $p\in [1,\infty]$.

	 \subsection{Approximation scheme and compactness}\label{section:compactness}

Solutions of \eqref{Maxwell:system:*2} will be constructed by a standard     compactness method based on the smooth approximation of solutions. Although this is now classical in the literature, we briefly recall here the principle ideas in that method. This will allow us to justify all the formal computations in the derivation  of our global a priori bounds, below.

 We  follow  the approach laid out in \cite[Section  3.1]{ah}, for instance. Thus,  we begin with  approximating the Navier--Stokes--Maxwell equations \eqref{Maxwell:system:*2} by a new system of equations which has a unique smooth solution.
 
  An admissible approximation should preserve the structures satisfied by the original system such as, in our case, the energy inequality \eqref{energy-inequa} and the axisymmetric properties. A possible choice of that approximation is given, for $n\in \mathbb{N}$, by
	 \begin{equation*} 
	\begin{cases}
		\begin{aligned}
		& \partial_t u_n +(S_nu_n) \cdot\nabla u _n  = \nu \Delta u_n- \nabla p_n + j_n \times B_n, &\div u_n =0,&
			\\
			& \frac{1}{c} \partial_t E_n - \nabla \times B_n =-(S_n j_n) , &\div E_n = 0,&
			\\
			& \frac{1}{c} \partial_t B_n + \nabla \times E_n  = 0 , &\div B_n = 0,   &
			\\
		& j_n= \sigma \big( cE_n + S_nP(u_n \times B_n)\big), &\div j_n = 0,&
		\end{aligned}
	\end{cases}
\end{equation*}
	 supplemented with the initial data $ (u_n,E_n,B_n)|_{t=0} \bydef S_n( u_0,E_0,B_0) $, where $S_n$ is a  radial cutoff Fourier multiplier which restricts the frequencies to the set $\{|\xi| \leq 2^n \}$ and converges to the identity as $n\to\infty$.

	 Showing that the approximate system, for any fixed $n\in \mathbb{N}$, has a unique global solution is a routine procedure based on standard methods.
	Moreover, one can show that the corresponding solutions are smooth in time and space  and satisfy the energy inequality 
\begin{equation*} 
	\norm {\left( u_n , E_n ,B_n\right)(t) }_{L^2}^2+2\nu\int_0^t \norm {\nabla u_n(\tau)}_{L^2}^2 d\tau
	+\frac{2}{\sigma}\int_0^t \norm {j_n(\tau)}_{L^2}^2 d\tau = \norm {S_n( u_0,E_0,B_0)}_{L^2}^2 \leq \mathcal{E}_0^2,
\end{equation*}
where we recall that 
\begin{equation*}
	\mathcal{E}_0 = \norm {(u_0,E_0,B_0)}_{L^2}.
\end{equation*}
If, furthermore, the initial data are assumed to be axisymmetric, and the cutoff operator $S_n$ does not alter that structure (which is the case when $S_n$ can be characterized as a convolution with a radial function), then the approximate smooth solution remains axisymmetric for all times.

Noting, once again, that the energy inequality above is not sufficient to ensure  the stability, as $n\to\infty$, of the nonlinear term
$$ 
j_n \times B_n ,
$$
our strategy thus consists in  looking for better bounds in higher regularity spaces, uniformly with respect to the regularizing parameter $n$. In particular, we will obtain new bounds which will allow us to establish the strong relative compactness  of all vector field $u_n$, $E_n$, $B_n$ and $j_n$ in $L^2_{\text{loc},t,x}$ and, then, conclude that  the approximate solutions converge, as $n\to \infty$, to an exact solution of \eqref{Maxwell:system:*2}. 

With such strong bounds, the full justification of the stability of the approximate system follows from standard compactness arguments, which we will therefore omit. We  refer to \cite{ah} for   some details on similar arguments applied to construction of global solutions of the Euler--Maxwell system in two dimensions of space.

   Normally, we should prove the a priori estimates on the approximate system above. However, as usual, since the approximate system enjoys the same structure as the original one, we will, from now on, assume that the solutions to \eqref{Maxwell:system:*2} are smooth and we will perform all estimates directly on \eqref{Maxwell:system:*2}.

 \subsection{Weak--strong uniqueness}

The functional spaces used in Theorem  \ref{Thm:1} are  sufficient to prove   uniqueness results for \eqref{Maxwell:system:*2}. Although  this is not hard to show, we choose to provide in the following proposition a weak--strong stability result with  a self-contained  proof   which covers     the uniqueness of solutions  claimed in Theorem  \ref{Thm:1}. See also \cite[Proposition 3.11]{as} for a similar weak--strong principle for the same system.

\begin{prop}[$L^2$ weak--strong stability]
Let $c>0$ and $(u_i,E_i,B_i)_{i\in \{1,2\}}$ be two weak solutions of \eqref{Maxwell:system:*2} associated with the same initial data and satisfying the energy inequality \eqref{energy-inequa}.
Assume moreover that 
$$ u_2\in L^2_{\loc}(\mathbb{R}^+;  L^\infty), \qquad E_2  \in L^2_{\loc}(\mathbb{R}^+;  L^3), \qquad  B_2 \in L^\infty _{\loc}(\mathbb{R}^+;  L^3).$$
Then, the two solutions are equal.
\end{prop}
\begin{proof}
We define the difference of the two solutions by
$$ \widetilde{u} \bydef   u_1-u_2  ,\qquad \widetilde{ E} \bydef   E_1-E_2,\qquad \widetilde{B}   \bydef   B_1-B_2, \qquad \widetilde{j}   \bydef   j_1-j_2 , $$
and we compute, for any $t>0$, that
\begin{equation*}
	\begin{aligned}
		\int_{\mathbb{R}^3} & \left(u_1\cdot u_2+E_1\cdot E_2+ B_1\cdot B_2\right)(t)dx
		+\frac 2\sigma \int_0^t\int_{\mathbb{R}^3}j_1\cdot j_2dxd\tau + 2\nu \int_0^t \int_{\mathbb{R}^3}\nabla u_1 \cdot \nabla u_2dxd\tau  
		\\
		& =-\int_0^t\int_{\mathbb{R}^3} (j_2\times \widetilde{B})\cdot \widetilde{u} dxd\tau
		+\int_0^t\int_{\mathbb{R}^3} 
		(\widetilde{j}\times \widetilde{B})\cdot u_2 
		 dxd \tau
		 - \int_0^t\int_{\mathbb{R}^3} 	\left(u_2\otimes \widetilde{u}\right):\nabla \widetilde{u}  dx d\tau.
	\end{aligned}
\end{equation*}
Note that the computations above can be  rigorously justified  by smoothing out the two solutions and following the proof of \cite[Lemma 2.1]{GP02}, for instance.

Therefore, setting 
$$F(t)\bydef \frac12\left(\norm {\widetilde u(t)}_{L^2}^2 + \|\widetilde E(t)\|_{L^2}^2 + \|\widetilde B(t)\|_{L^2}^2\right),$$
and making use of the energy inequality \eqref{energy-inequa}, which is assumed to be satisfied by both solutions, we obtain, for any $\varepsilon>0$, that
\begin{equation*}
	\begin{aligned}
		F(t)  +\nu \norm {\nabla \widetilde{u} }_{L^2_{t,x}}^2 + \frac{1}{\sigma}\|\widetilde{j} \|_{L^2_{t,x}}^2  
		&\leq \int_0^t \norm {j_2(\tau)}_{L^3}   \|\widetilde{B}(\tau)\|_{L^2} \norm {\widetilde{u}(\tau)}_{L^6} d\tau \\
		&\quad +  \int_0^t \norm {u_2(\tau)}_{L^\infty} \left(  \|\widetilde{j}(\tau)\|_{L^2} \|\widetilde{B}(\tau)\|_{L^2} + \norm {\widetilde{u}(\tau)}_{L^2} \norm {\nabla\widetilde{ u}(\tau)}_{L^2}  \right) d\tau\\
		&\leq \varepsilon   + \frac{\nu}{4} \norm {\nabla \widetilde{u} }_{L^2_{t,x}}^2 + \frac{C}{\nu}  \int_0^t \norm {j_2(\tau)}_{L^3}^2   \|\widetilde{B}(\tau)\|_{L^2}^2   d\tau \\
		& \quad +  \frac{\nu}{4} \norm {\nabla \widetilde{u} }_{L^2_{t,x}}^2 + \frac{1}{\nu}  \int_0^t \norm {u_2(\tau)}_{L^\infty}^2   \|\widetilde{u}(\tau)\|_{L^2}^2   d\tau\\
		&\quad +  \frac{1}{2\sigma} \|\widetilde{j} \|_{L^2_{t,x}}^2+  \frac{\sigma}{2}\int_0^t \|u_2(\tau)\|_{L^\infty}^2 \|\widetilde{B}(\tau)\|_{L^2}^2  d\tau,
	\end{aligned}
\end{equation*} 
where $C>0$ is the constant from the embedding $\dot{H}^1\hookrightarrow L^6(\mathbb{R}^3).$

Hence, by further simplifying the preceding bound and applying Gr\"onwall's lemma, we arrive at the conclusion that 
\begin{equation}\label{weak-strong:uniqueness}
	\begin{aligned}
F(t) &+\frac{\nu}{2} \norm {\nabla \widetilde{u} }_{L^2_{t,x}}^2 + \frac{1}{2\sigma}\|\widetilde{j} \|_{L^2_{t,x}}^2  \leq \varepsilon    \exp\left( C_{\nu,\sigma} \int_0^t  \left(\|j_2(\tau) \|_{L^3}^2 +  \|u_2(\tau) \|_{L^\infty}^2\right) d\tau \right),
	\end{aligned}
\end{equation} 
for any $\varepsilon>0$ and some constant $C_{\nu,\sigma}>0$. 

Accordingly, by further exploiting   Ohm's law  
$$j_2 = cE_2 + P(u_2 \times B_2),$$
 and employing the additional bounds on the second solution $(u_2,E_2,B_2)$, it is  readily seen then that the right-hand side in \eqref{weak-strong:uniqueness} is finite and vanishes as $\varepsilon\to 0$. This concludes the proof of the weak--strong uniqueness.
\end{proof} 

\subsection{Existence of global solutions}\label{Section:closing the estimates}

	Here, we prove existence of global solutions of \eqref{Maxwell:system:*2}, as it is claimed in  Theorem \ref{Thm:1}. In view of the arguments laid out in Section \ref{section:compactness}, above, this proof is reduced to establishing adequate a priori global estimates for \eqref{Maxwell:system:*2}.

	The proof is split into two parts. The first part is devoted to the case of rough profiles 
	$$ c^{-1}(E_0^c,B_0^c) \in  \dot{B}^{\frac{5}{2}}_{2,1},$$
	uniformly in $c>0$. This means that the $\dot{B}^{\frac{5}{2}}_{2,1} $ norm of the initial data is allowed to blow up, as $c\to \infty,$ with a rate which is at most of order $c$.

	 In the second part of the proof below, we deal with the case of regular  profiles, i.e., 
	 $$(E_0^c,B_0^c) \in  \dot{B}^{\frac{5}{2}}_{2,1},$$
	 uniformly with respect to $c>0$.

	 \subsubsection{Rough profiles}

	 Here, we assume that 
	 $$ c^{-1}(u_0^c,E_0^c,B_0^c) \in  \dot{B}^{\frac{5}{2}}_{2,1},$$
uniformly with respect to $c\in(0,\infty)$,	 meaning that $  (u_0^c,E_0^c,B_0^c) $ can be unbounded in $  \dot{B}^{\frac{5}{2}}_{2,1}$ as $c\rightarrow \infty.$

\emph{Control of the velocity field.} The principal  control of   velocity fields is obtained in Proposition \ref{prop:almost-parabolic-ES}. For convenience, we rewrite here its main estimate:
\begin{equation} \label{velocity:A}
\begin{aligned}
& \norm{(\omega, \Omega)}_{L^\infty_t L^2 \cap L^2_t \dot{H}^1 }\\
 & \quad\quad \lesssim   \left(\norm{(\omega_0,\Omega_0)}_{L^2} +\norm { \frac{1}{c}\partial_t E}_{L^2_t \dot{H}^\frac{1}{2}  }    \norm {B}_{L^\infty_t H^2  }     + \norm {\Gamma}_{L^\infty_t L^3 }\norm {(B,\Gamma)}_{L^2_t \dot{H}^1 } \right) e^{ C\mathcal{E}^2_0 }. 
\end{aligned}
\end{equation} 
 We also notice, due to  the energy inequality \eqref{energy-inequa}, that    
 \begin{equation}\label{energy:t}
 \mathcal{E}_0 \leq \mathcal{E}_t \bydef \norm {(u , E ,B)  }_{L^\infty_t L^2} + \norm {  u }_{L^2_t\dot{H}^1}   
	+  \norm {j }_{L^2_tL^2}   \lesssim   \mathcal{E}_0 .
 \end{equation}
 Moreover, the standard embeddings
 \begin{equation*}
 \norm{u}_{L_t^{4}\dot B^{\frac{3}{2}}_{2,1}}  \lesssim  \norm{u}_{L^\infty_t   \dot{H}^1 \cap L^2_t \dot{H} ^2 } \lesssim   \norm{ \omega }_{L^\infty_t L^2 \cap L^2_t \dot{H}^1 }
 \end{equation*} 
 will be used regularly without explicit reference.
 The justification of the first inequality above is  done   by applying an abstract interpolation argument, whereas the second one   straightforwardly follows from the Biot--Savart law.

 \emph{Control of high electromagnetic frequencies.}
 The control of high electromagnetic frequencies is given by     Lemma \ref{lemma-high F}. More precisely, we obtain from \eqref{E_>:5/2}, with the specific values   $q\in\{2,\infty\}$, that
   \begin{equation}\label{high:B}
   \begin{aligned}
   c^{-1}  \norm{ (E,B) }_{\widetilde{L} ^\infty _t \dot{B}^{\frac{5}{2} }_{2,1,>}  }
   &+  \norm{ (E,B) }_{\widetilde{L} ^2  _t \dot{B}^{\frac{5}{2} }_{2,1,>}  }
   \\
   &\lesssim    c^{-1 }\norm { (E_0,B_0)}_{\dot{B}^\frac{5}{2} _{2,1,>}}    \\
	  & \; +c^{-  \frac{1}{2}   }\norm u_{L^\infty _t \dot{H} ^1 \cap L^2 _t \dot{H} ^2  }  \left(  \norm{ B }_{\widetilde{L} ^2_t\dot{B}^{ \frac{5}{2}}_{2,1,>} } +  \norm{ B }_{L ^2_t \dot{B}^{ \frac{5}{2}}_{2,1,<} }    \right).
   \end{aligned}
\end{equation}

	 \emph{Control of low electromagnetic frequencies---Part A.} The control of  low electromagnetic frequencies relies on Lemma \ref{low:freq:estimates}, above. Specifically, applying \eqref{frequencies:1b} with the values $q =2$, $ p=m=4$, $ \alpha=1$ and $s=\frac{3}{2}$ yields that
   \begin{equation*} 
   \begin{aligned}
			  \norm{B}_{L^2 _t\dot B^{ \frac{5}{2}}_{2,1,<}}
		&	 \lesssim  
			c^{-1}\norm{E_0}_{\dot B^{\frac{5}{2}}_{2,2,<}}
			+\norm{B_0}_{\dot B^{\frac{3}{2}}_{2,2,<}}
			+\norm{u}_{L_t^{4}\dot B^{\frac{3}{2}}_{2,1}} 
			\norm{B}_{L_t^{4}\dot B^{\frac{3}{2}}_{2,1} }.
			   \end{aligned}
\end{equation*}   
			Likewise, employing \eqref{frequencies:1b}, again,  with the values $p= 2 $ and $s=2$ entails that
		   \begin{equation*} 
   \begin{aligned}
			  \norm{B}_{L^2 _t\dot B^{ \frac{5}{2}}_{2,1,<}}
		&	 \lesssim  
			c^{-1}\norm{E_0}_{\dot B^{\frac{5}{2}}_{2,2,<}}
			+\norm{B_0}_{\dot B^{\frac{3}{2}}_{2,2,<}}
			+\norm{u}_{L_t^{4}\dot B^{\frac{3}{2}}_{2,1}} 
			\norm{B}_{L_t^{2}\dot B^{2}_{2,1} } .
   \end{aligned}
\end{equation*}   
		Actually, since \eqref{Maxwell:system:*} is a linear system, by  splitting   high and low frequencies of $B$ in the source term $P(u\times B)$,  one can straightforwardly adapt the proofs of the preceding estimates to obtain the more useful   control 
   \begin{equation} \label{low:A}
   \begin{aligned}
			  \norm{B}_{L^2 _t\dot B^{ \frac{5}{2}}_{2,1,<}} 
			&	 \lesssim  
			 \norm{(E_0,B_0)}_{\dot{B}^{\frac{3}{2}}_{2,2,<}}\\
		& \quad	+\norm{u}_{L_t^{4}\dot B^{\frac{3}{2}}_{2,1}} \left(\norm{\mathrm{1}_{ \{|D|< \frac{\sigma c}{2}  \}} B}_{L_t^{4}\dot B^{\frac{3}{2}}_{2,1} } +   
			\norm{\mathrm{1}_{ \{|D|\geq \frac{\sigma c}{2}  \}} B}_{L_t^{2}\dot B^{2}_{2,1}  }   \right)  \\ 
			&	 \lesssim  
			 \norm{(E_0,B_0)}_{\dot{H}^{\frac{3}{2}} } \\ 
			 & \quad
			+\norm{u}_{L_t^{\infty}\dot H^{1} \cap L_t^{2 }\dot H^{2}} \left(\norm{  B}_{L_t^{4}\dot B^{\frac{3}{2}}_{2,1,<}}    +   
		c^{ -\frac{1}{2}}	\norm{  B}_{L_t^{2}\dot B^{\frac{5}{2}}_{2,1,<} } + c^{ -\frac{1}{2}}	\norm{  B}_{\widetilde{L} _t^{2}\dot B^{\frac{5}{2}}_{2,1,>} }   \right) .
   \end{aligned}
\end{equation}   
			Before we proceed with the proof, let us establish the energy estimate for $(E,B)$ in Sobolev spaces, which will then be combined with the preceding estimate to complete the proof of the low-frequency bounds on electromagnetic fields.

			\emph{Energy estimate for electromagnetic fields.}
		Applying Lemma  \ref{lemma.reg-persistence} with the values $s=\frac{3}{2}$ and $(p,q,\varepsilon)= (2,\infty,\frac{1}{2})$, and then with the values $(p,q)= (4,4)$ at the endpoint $\varepsilon=0$, we find  that 
			\begin{equation*}
			\norm {(E,B)}_{L^\infty_t \dot{H}^\frac{3}{2}} + c \norm {E}_{L^2_t\dot{H}^\frac{3}{2}}\lesssim  \norm {(E_0,B_0)}_{ \dot{H}^\frac{3}{2}} + \norm u_{L_t^{\infty}\dot H^{1} } \norm B_{ L^2_t \dot{H}^ 2 },
			\end{equation*}
			and 
				\begin{eqnarray*}
			\norm {(E,B)}_{L^\infty_t \dot{H}^\frac{3}{2}} + c \norm {E}_{L^2_t\dot{H}^\frac{3}{2}}\lesssim  \norm {(E_0,B_0)}_{ \dot{H}^\frac{3}{2}} + \norm u_{L_t^{4}\dot B^{\frac{3}{2}}_{2,1} } \norm B_{ L^4_t \dot{H}^  \frac{3}{2} }.
			\end{eqnarray*}
			
			In fact, by spliting high and low frequencies  of $B$ in the right-hand side of the preceding estimates,  it is possible to adapt their proofs to obtain the following more useful bound
\begin{equation*}   
 \begin{aligned}
			\norm {(E,B)}_{L^\infty_t \dot{H}^\frac{3}{2}} + c \norm {E}_{L^2_t\dot{H}^\frac{3}{2}} & \lesssim  \norm {(E_0,B_0)}_{ \dot{H}^\frac{3}{2}} \\
			& \quad + \norm u_{L_t^{\infty}\dot H^{1} \cap L^2_t \dot{H}^2 } \left( \norm{  \mathrm{1}_{ \{|D|< \frac{\sigma c}{2}  \}} B}_{ L^4_t \dot{H}^ \frac{3}{2} }   +  \norm{  \mathrm{1}_{ \{|D| \geq \frac{\sigma c}{2}  \}} B}_{ L^2_t \dot{H}^ 2 } \right).
   \end{aligned}
\end{equation*}
			Therefore, similarly to    \eqref{low:A},  we infer that
	\begin{equation}   \label{energy-H32}
 \begin{aligned}  
			&\norm{(E,B)}_{L^\infty _t\dot H^{\frac{3}{2} }  }
			+ c \norm {E}_{L^2_t\dot{H}^\frac{3}{2}}  \\
			&\qquad\qquad \lesssim 
			 \norm{(E_0,B_0)}_{\dot{H}^{\frac{3}{2}}}
			\\
			&\qquad\qquad\quad+ \norm{u}_{L_t^{\infty}\dot H^{1} \cap L_t^{2 }\dot H^{2}} \left(\norm{  B}_{L_t^{4}\dot B^{\frac{3}{2}}_{2,1,<}}    +   
		c^{ -\frac{1}{2}}	\norm{  B}_{L_t^{2}\dot B^{\frac{5}{2}}_{2,1,<} } + c^{ -\frac{1}{2}}	\norm{  B}_{\widetilde{L} _t^{2}\dot B^{\frac{5}{2}}_{2,1,>} }   \right) .
 \end{aligned}
\end{equation}

\emph{Control of low electromagnetic frequencies---Part B.}
Now, we carry on with the estimates of low frequencies of   electromagnetic fields by first combining \eqref{low:A} and \eqref{energy-H32} to find that  
\begin{equation}\label{low-enery}
			\begin{aligned}
			\norm{(E,B)}_{L^\infty _t\dot H^{\frac{3}{2} }  } &+ c \norm {E}_{L^2_t\dot{H}^\frac{3}{2}}  +   \norm{B}_{L^2 _t\dot B^{ \frac{5}{2}}_{2,1,<}} \\ 
			& 	 \lesssim 
			 \norm{(E_0,B_0)}_{\dot{H}^{\frac{3}{2}}}+ c^{ -\frac{1}{2}}	 \norm{u}_{L_t^{\infty}\dot H^{1} \cap L_t^{2 }\dot H^{2}}\left(  
		 	\norm{  B}_{L_t^{2}\dot B^{\frac{5}{2}}_{2,1,<} } + \norm{  B}_{\widetilde{L} _t^{2}\dot B^{\frac{5}{2}}_{2,1,>} }   \right)     \\
			 &  
			 \quad +\norm{u}_{L_t^{\infty}\dot H^{1} \cap L_t^{2 }\dot H^{2}} \norm{  B}_{L_t^{4}\dot B^{\frac{3}{2}}_{2,1,<}}     .
			\end{aligned}
\end{equation}		

Note that   one cannot expect to observe any decay, with respect to  $c$,  in the nonlinear term 
$$ \norm{u}_{L_t^{\infty}\dot H^{1} \cap L_t^{2 }\dot H^{2}}  \norm{  B}_{L_t^{4}\dot B^{\frac{3}{2}}_{2,1,<}}.$$
Accordingly, at this stage, the bound  \eqref{low-enery}    does not seem to be helpful  in obtaining a global control without additional conditions on the size of the initial data.  
 
  Nevertheless, our key observation here is that the space $L_t^{4}\dot B^{\frac{3}{2}}_{2,1} $ can be obtained by interpolating $L^\infty_t\dot H^\frac 32$, $ L_t^{2}\dot B^{\frac{5}{2}}_{2,1 } $ and $ L^2_t \dot{H}^1$. Note that $L^\infty_t\dot H^\frac 32$ and $ L_t^{2}\dot B^{\frac{5}{2}}_{2,1 ,<} $  appear in the left-hand side of \eqref{low-enery}, whereas, $ L^2_t \dot{H}^1$ is a good space  for $B$, because, in view of \eqref{B-almost.GE}, the magnetic field is asymptotically globally bounded in that space.
   Based on this remark, we write, by interpolation, that
   \begin{equation*}
	   \begin{aligned}
		\norm{B}_{L_t^{4}\dot B^{\frac{3}{2}}_{2,1,<} } 
		&\lesssim \left( \int_0^t \norm {B(\tau)}_{\dot{H}^1} ^{\frac{4}{3}}\norm {B(\tau)}_{\dot{B}^\frac{7}{4}_{2,1,<}}^{\frac{8}{3}} d\tau \right)^\frac{1}{4}\\
		&\lesssim   \norm {B }_{L^2_t \dot{H}^1} ^{\frac{1}{3}}\norm {B }_{L^8_t\dot{B}^\frac{7}{4}_{2,1,<}}^{\frac{2}{3}}  \\
		&\lesssim   \norm {B }_{L^2_t \dot{H}^1} ^{\frac{1}{3}}\left(  \norm{B}_{L^\infty _t\dot B^{\frac{3}{2} }_{2,2,<}} + \norm{B}_{L^2 _t\dot B^{ \frac{5}{2}}_{2,1,<}}  \right)^{\frac{2}{3}},
	   \end{aligned}
\end{equation*}
which leads to
\begin{equation}\label{low:B}
 \begin{aligned}
 \norm{(E,B)}_{L^\infty _t\dot H^{\frac{3}{2} }  } &+ c \norm {E}_{L^2_t\dot{H}^\frac{3}{2}}  +   \norm{B}_{L^2 _t\dot B^{ \frac{5}{2}}_{2,1,<}} \\
			&	 \lesssim  
			 \norm{(E_0,B_0)}_{\dot{H}^{\frac{3}{2}}}
			 +c^{ -\frac{1}{2}}\norm{u}_{L_t^{\infty}\dot H^{1} \cap L_t^{2 }\dot H^{2}}   \left( 
			 		 	\norm{  B}_{L_t^{2}\dot B^{\frac{5}{2}}_{2,1,<} } + 	\norm{  B}_{\widetilde{L} _t^{2}\dot B^{\frac{5}{2}}_{2,1,>} }   \right)
			\\
			&\quad+ \norm{u}_{L_t^{\infty}\dot H^{1} \cap L_t^{2 }\dot H^{2}} \norm {B }_{L^2_t \dot{H}^1} ^{\frac{1}{3}}\left(  \norm{B}_{L^\infty _t\dot B^{\frac{3}{2} }_{2,2,<}} + \norm{B}_{L^2 _t\dot B^{ \frac{5}{2}}_{2,1,<}}\right)^{\frac{2}{3}}.
 \end{aligned}
\end{equation}

In order to complete our summary of   all relevant low-frequency bounds on electromagnetic fields,   we recall the estimate
			\begin{equation}\label{E-AAA***}
\begin{aligned}
\norm E_{L^2_t\dot{B}^\frac{5}{2}_{2,1} } &\lesssim c^{-1}  \norm {(E_0,B_0)}_{\dot{B}^\frac{5}{2}_{2,1}}  
 +  c^{-\frac{1}{2}}\norm{    u   }_{L  ^{\infty }_t\dot{H}^1  \cap L  ^{2 }_t\dot{H}^2 } \left( \norm{   B  }_{\widetilde{L}   ^{2 }_t\dot{B}^{\frac{5}{2}}_{2,1,>    } } +  \norm{   B }_{L^2  _t\dot{B}^{\frac{5}{2}}_{2,1  ,< }  } \right),
\end{aligned}
\end{equation}
which is established in \eqref{E-AAA}.

\emph{Decay of electric fields  and almost-parabolic   bounds on magnetic fields.}
 For later use, we recall and add some precision to the bounds proved in Section \ref{section:AGE}, above.
			
			  More specifically, observe first that  Proposition \ref{decay-E}, with the values $s\in\{0,\frac{1}{2}\}$, provides us with the decay estimates  
			 \begin{equation} \label{E-decay:A000}
	 \begin{aligned}
	\norm {\frac{1}{c} \partial_t E}_{L^2_tL^2}
	&\lesssim   c^{-1} \norm {\nabla \times B_0 - j_0}_{L^2}
	+c^{-1}\norm{E_0}_{\dot H^1}
	+  c^{-2} \norm {   u  }_{ L^ \infty _t  \dot{H}^1    \cap L^ 2_t \dot{H}^2  }^2    \norm {   B}_{L^\infty_t \dot{H}^\frac{3}{2}  }  \\
 &\quad+ c^{-2}(1+ \mathcal{E}_0)  \Big(       \norm {   u  }_{ L^ \infty _t \dot{H}^1    \cap L^ 2_t \dot{H} ^2  } +   \norm {   B}_{L^\infty_t \dot{H}^ \frac{3}{2}  }       \Big)   \\
 &\quad\times\Big(   \norm {   B}_{L^\infty_t \dot{H}^\frac{3}{2} }   +   c \norm {E}_{L^2_t\dot{H}^\frac{3}{2} }\Big)
 \end{aligned}
 \end{equation}
 and 
	 \begin{equation} \label{E-decay:A}
	 \begin{aligned}
	\norm {\frac{1}{c} \partial_t E}_{L^2_t\dot{H}^\frac{1}{2}} &\lesssim   c^{-1} \norm {\nabla \times B_0 - j_0}_{\dot{H}^\frac{1}{2}}
	+c^{-1}\norm{E_0}_{\dot H^\frac 32}
	+  c^{-2} \norm {   u  }_{ L^ \infty _t  \dot{H}^1    \cap L^ 2_t \dot{H}^2  }^2    \norm {   B}_{L^\infty_t \dot{H}^\frac{3}{2}  }  \\
 &\quad+ c^{-2}(1+ \mathcal{E}_0)  \Big(       \norm {   u  }_{ L^ \infty _t \dot{H}^1    \cap L^ 2_t \dot{H} ^2  } +   \norm {   B}_{L^\infty_t \dot{H}^ \frac{3}{2}  }       \Big)   \\
 &\quad\times\Big(   \norm {   B}_{L^\infty_t \dot{H}^2 }   +   c \norm {E}_{L^2_t\dot{H}^2}\Big) .
 \end{aligned} 
 \end{equation}

 As for the  almost-parabolic  estimates of $B$, we first utilize Lemma \ref{paradifferential:1} to write that
 \begin{equation*}
\begin{aligned}
c^{-1} \norm {\nabla \times B_0 - j_0}_{L^2}  & = c^{-1} \norm {\nabla \times B_0 - cE_0 - P(u_0\times B_0)}_{L^2} \\
& \lesssim  c^{-1} \left(  \norm {  B_0 }_{\dot{H}^1}   + \norm{  u_0}_{L^2}  \norm{ B_0 }_{\dot{H}^\frac{3}{2}} \right)+\norm{ E_0}_{L^2} \\
& \lesssim  c^{-1} \left(  \norm {  B_0 }_{\dot{H}^1}   +\mathcal{E}_0 \norm{ B_0 }_{\dot{H}^\frac{3}{2}} \right)+\mathcal{E}_0 .
\end{aligned}
\end{equation*}			
 Then, by substituting the preceding control in  \eqref{E-decay:A000}, and by further incorporating the resulting bound in  the estimate from Proposition \ref{prop:almost-parabolic-ES}, we find that 
 	 \begin{equation}  \label{parabolic-ES:A}
	 \begin{aligned}
\norm{ B}_{L^2_t \dot{H}^1 } &\lesssim   \mathcal{E}_0 +  c^ {-1} \big(\norm {(E_0,B_0)}_{\dot{H}^1} + \mathcal{E}_0 \norm {B_0}_{\dot{H}^\frac{3}{2}}\big)+     c^{-2} \norm {   u  }_{ L^ \infty _t  H^1    \cap L^ 2_t \dot{H} ^2  }^2    \norm {   B}_{L^\infty_t \dot{H}^\frac{3}{2}  }     \\
&\quad +c^{-2} \left(1+ \mathcal{E}_0\right)  \left(  \norm {   u  }_{L^\infty_t \dot{H} ^1 \cap  L^ 2_t \dot{H} ^2 }  +  \norm {   B}_{L^\infty_t \dot{H}^\frac{3}{2} }  \right)\left( \norm {   B}_{L^\infty_t \dot{H}^\frac{3}{2} } + \norm {cE}_{L^\infty_t \dot{H}^\frac{3}{2}} \right),
 \end{aligned}
\end{equation}
which provides an asymptotic parabolic regularity estimate on $B$.

  Finally, we emphasize that we will also make use of the similar almost-parabolic estimates on $\Gamma$ obtained in Proposition \ref{prop:almost-parabolic-ES}. More precisely, the relevant estimates on $\Gamma$ are
\begin{equation}\label{parabolic-ES:B:1}
\norm{ \Gamma}_{L^2_t\dot{H}^1}\lesssim  \norm{\Gamma_0}_{L^2}    +c^{-1} \norm {  E_0 }_{\dot{H}^2}   +  c^{-1} \left(     \norm {B}_{L^\infty _t \dot{H}^{2  }   } + c\norm { E}_{L^2_t \dot{H}^2 } \right)
  \end{equation}
  and 
\begin{equation}\label{propagation-L3:gamma}
\norm{ \Gamma}_{L^\infty_t L^3 }\lesssim  \norm { \Gamma_0}_{ L^3   }     +c^{-1} \norm {E_0}_{\dot{H}^\frac{5}{2}}   +      \norm { E}_{L^2_t \dot{H}^\frac{5}{2} }.
\end{equation}

\emph{Nonlinear energy estimate.}
Here, we gather all the bounds above to produce a nonlinear energy estimate. To that end, let us first introduce, for any $t\geq 0$, the functional $\mathcal{H}(t)$ given by
\begin{equation*}
\begin{aligned}
\mathcal{H}(t)& \bydef   \norm{(\omega,\Omega)}_{L^\infty_t L^2 \cap L^2_t \dot{H}^1 }  +  \mathcal{E}_t  + \norm \Gamma_{L^\infty_t L^3}  \\
& \quad + \norm {(E,B)}_{L^\infty_t \dot{H} ^\frac{3}{2}}
+c^{-1}\norm {(E,B)}_{L^\infty_t \dot{B}^\frac{5}{2}_{2,1}} + c \norm E_{L^2_t \dot{H}^\frac{3}{2}}     \\
&\quad+ \norm {(E,B)}_{\widetilde{L} ^2_t \dot{B}^\frac{5}{2}_{2,1,>}} +   \norm B_{L^2_t \dot{B}^\frac{5}{2}_{2,1,<}} +    \norm E_{L^2_t \dot{B}^\frac{5}{2}_{2,1}} + \norm{B}_{L^2_t\dot{H}^1}\\
&\quad+ \norm {\nabla \times  B- j}_{L^\infty_t\dot{H}^{\frac{1}{2}}},
\end{aligned}
\end{equation*}
where $\mathcal{E}_t$ is given in \eqref{energy:t} and all the time-norms are taken over the whole interval $[0,t)$. Accordingly, we conventionally define 
\begin{equation*}
\begin{aligned}
\mathcal{H}(0)& \bydef   \norm{(\omega_0,\Omega_0)}_{  L^2}  +\mathcal{E}_0 + \norm{ \Gamma_0}_{  L^3}  \\
& \quad + \norm {(E_0,B_0)}_{  \dot{H} ^\frac{3}{2}   } +c^{-1}\norm {(E_0,B_0)}_{  \dot{B}^\frac{5}{2}_{2,1}}   + \norm {\nabla \times  B_0- j_0}_{ \dot{H}^{\frac{1}{2}}}.
\end{aligned}
\end{equation*}
In particular, note, for all $t \geq 0$, that
\begin{equation}\label{H:trivial:bound}
	\mathcal{H}(0) \leq \mathcal{H}(t) . 
\end{equation}
Further observe that $\norm{\Gamma_0}_{L^3}\lesssim \norm {B_0}_{\dot H^\frac 32}$, by virtue of Lemma \ref{lemma-f/x:2}.

Now, we claim   an estimate of the form 
$$ \mathcal{H}(t) \leq C_0 + P(\mathcal{H}(t)),  $$
for $t\geq 0$, some constant $C_0>0$ depending only on the initial data, uniformly with respect to $c$, and some polynomial $P\in \mathbb{R}^+[X]$ whose coefficients vanish asymptotically as $c\rightarrow \infty.$
Owing to the latter bound above, Lemma \ref{fix-point:lem}  below will eventually allow us to deduce the desired global estimates, under adequate conditions on the data.

In order to reach such a bound, we first proceed  with the control of $\frac{1}{c} \partial_t E$ by observing that   \eqref{E-decay:A} yields that
 \begin{eqnarray*} 
	\norm {\frac{1}{c} \partial_t E}_{L^2_t\dot{H}^\frac{1}{2}} \lesssim   c^{-1} \mathcal{H}(0) +c^{-2}  \mathcal{H}^3(t) + c^{-2} (1+ \mathcal{E}_0)\mathcal{H}(t) \Big(   \norm {   B}_{L^\infty_t \dot{H}^2 }   +   c \norm {E}_{L^2_t\dot{H}^2}\Big)   .
	\end{eqnarray*} 	
	Therefore, noticing, by an interpolation argument, that 
	\begin{equation}\label{observ.1}
 \norm {B}_{L^\infty _t \dot{H}^{2  }   } + c\norm { E}_{L^2_t \dot{H}^2 }  \lesssim c^{  \frac{1}{2}}    \mathcal{H} (t), 
\end{equation}  
 it then follows that    
 \begin{eqnarray}\label{dt E: final-ES}
	\norm {\frac{1}{c} \partial_t E}_{L^2_t\dot{H}^\frac{1}{2}} \lesssim   c^{-1} \mathcal{H}(t) + c^{-\frac{3}{2} } (1+ \mathcal{E}_0)\mathcal{H}^2(t)+c^{-2}   \mathcal{H}^3(t)  .
	\end{eqnarray}
	
Next, we exploit the almost-parabolic estimate on  $B$ given in \eqref{parabolic-ES:A}. To that end, note first, by splitting the frequencies of $B_0$, that 
	$$ \begin{aligned}
	c^{-1}\norm {(E_0,B_0)}_{\dot{H}^1} &\leq c^{-1 } \left(   \norm {(E_0,B_0)}_{\dot{B}^1_{2,2,<}} + \norm {(E_0,B_0)}_{\dot{B}^1_{2,2,>}} \right) \\
	&\lesssim    \norm {(E_0,B_0)}_{L^2} +c ^{-\frac{3}{2}} \norm {(E_0,B_0)}_{\dot{H}^\frac{3}{2} } \\
	&\lesssim \mathcal{E}_0  + c ^{-\frac{3}{2}}  \mathcal{H}  (0) .
	\end{aligned}
	$$ 
	In addition to that,    using the fact that $ \mathcal{H}(0)\leq \mathcal{H}(t)$, one sees that
 	\begin{equation*}
\begin{aligned}
	c^{-1}\mathcal{E}_0 \norm {B_0}_{\dot{H}^\frac{3}{2}} &\leq   \mathcal{E}_0 + c^{-2}\mathcal{E}_0 \norm {B_0}_{\dot{H}^\frac{3}{2}}^2\\
&\leq  \mathcal{E}_0 + c^{-2}(1+\mathcal{E}_0) \mathcal{H} ^2(t).
\end{aligned}
	\end{equation*}  
	Accordingly,   \eqref{parabolic-ES:A} yields, for $c\geq 1$, that
		\begin{equation}\label{final:B-l2H1}	 
	\begin{aligned}
	\norm{ B}_{L^2_t \dot{H}^1 }  
  &  \lesssim   &     \mathcal{H}(0) +  c^{-2} (1+ \mathcal{E}_0)       \mathcal{H}^2(t)    + c^{-2}         \mathcal{H}^3(t)   
	\end{aligned}
\end{equation}	 	 
and,  since $\mathcal{H}(0)\leq \mathcal{H}(t)$, that
\begin{equation} \label{B:parabolic ES:B***}	 
	\begin{aligned}
	\norm{ B}_{L^2_t \dot{H}^1 }  
  &  \lesssim   & \mathcal{E}_0+ c^{-\frac{3}{2} }   \mathcal{H}(t) +  c^{-2} (1+ \mathcal{E}_0)       \mathcal{H}^2(t)    + c^{-2}         \mathcal{H}^3(t)   .
	\end{aligned}
\end{equation}

We turn now our attention to the bounds on $\Gamma$  given by  \eqref{parabolic-ES:B:1} and  \eqref{propagation-L3:gamma}. Observing, by a simple interpolation argument, that 
  $$ c^{-1} \norm {E_0}_{\dot{H}^2} \leq c^{- \frac{1}{2} } \left(c^{-1 }  \norm {E_0}_{\dot{H}^{\frac{5}{2}}} \right) ^ \frac{1}{2} \norm {E_0}_{\dot{H}^{\frac{3}{2}}} ^ \frac{1}{2} \leq c^{- \frac{1}{2}} \mathcal{H}(0) \leq  c^{- \frac{1}{2}} \mathcal{H}(t),$$
it follows, by incorporating  \eqref{observ.1} into \eqref{parabolic-ES:B:1}, that
  \begin{equation}\label{final:ES-GammaH1}
\norm{ \Gamma}_{L^2_t\dot{H}^1}\lesssim  \norm{\Gamma_0}_{L^2}      +  c^{-\frac{1}{2}} \mathcal{H}(t).
  \end{equation} 
As for the $L^\infty_tL^3$ bound on $\Gamma$, we begin by deducing from   \eqref{E-AAA***}    that
   \begin{equation}\label{E-AAA:2}
\norm E_{L^2_t \dot{B}^\frac{5}{2}_{2,1} } \lesssim \mathcal{H}(0) +  c^{-\frac{1}{2}}\mathcal{H}^2(t),
\end{equation}
which,
in view of  \eqref{propagation-L3:gamma}, yields that   
\begin{equation} \label{final:ES-Gamma3}
\norm{ \Gamma}_{L^\infty_t L^3 }\lesssim   \mathcal{H}(0) +  c^{-\frac{1}{2}}\mathcal{H}^2(t).
\end{equation}  

Now, we establish a control of velocity fields in terms of the functional $\mathcal{H}(t)$. To that end, employing \eqref{observ.1} and the simple fact that 
$$\mathcal{E}_0\lesssim \mathcal{E}_t \leq c^{\frac{1}{2}}\mathcal{H}(t), \quad \text{for all } c\geq 1,$$
one deduces that
\begin{equation*}
 \norm {B}_{L^\infty _t H^{2  }    }   \leq \mathcal{E}_0 +  c^{  \frac{1}{2}}    \mathcal{H} (t) \leq 2 c^{  \frac{1}{2}}    \mathcal{H} (t) .
\end{equation*}  
Hence,  \eqref{dt E: final-ES} provides us with the bound 
\begin{eqnarray} \label{v:piece1}
	\norm {\frac{1}{c} \partial_t E}_{L^2_t\dot{H}^\frac{1}{2}} \norm{B}_{L^\infty_tH^2}
	\lesssim   \left( c^{-\frac{1}{2} } \mathcal{H}^2(t) + c^{-1 } \mathcal{H}^3(t)+c^{-\frac{3}{2} }   \mathcal{H}^4(t)\right) e^{\mathcal{E}_0^2}   .
	\end{eqnarray}  
On the other hand, 
combining \eqref{B:parabolic ES:B***}, \eqref{final:ES-GammaH1} and \eqref{final:ES-Gamma3} yields, for any $c\geq 1$, that 
\begin{equation*}
\begin{aligned}
 \norm{ \Gamma}_{L^\infty_t L^3 }  &\norm{(B, \Gamma)}_{L^2_t\dot{H}^1}  \\ & \lesssim    \left(  \mathcal{H}(0) +  c^{-\frac{1}{2}}\mathcal{H}^2(t) \right) \Bigg(  \norm{\Gamma_0}_{L^2}      +  \mathcal{E}_0 +  c^{-\frac{1}{2}} \mathcal{H}(t) +    c^{-2} (1+ \mathcal{E}_0)       \mathcal{H}^2(t)   + c^{-2}         \mathcal{H}^3(t)   \Bigg)\\
  & \lesssim    \mathcal{H}   (0) \left(  \norm{\Gamma_0}_{L^2}      +  \mathcal{E}_0 \right)  + c^{-\frac{1}{2}} \mathcal{H} ^2  (t) \left(  \norm{\Gamma_0}_{L^2}      +  \mathcal{E}_0 \right)  \\
   &  \quad +   \mathcal{H}(0)  (1 + \mathcal{E}_0)   \Bigg( c^{-\frac{1}{2}} \mathcal{H}(t) +    c^{-2}     \mathcal{H}^2(t)    + c^{-2}         \mathcal{H}^3(t)   \Bigg)  \\
  & \quad + (1 + \mathcal{E}_0)
  \left(c^{-1 }  \mathcal{H}^3(t)   + c^{- \frac{5}{2} }  \mathcal{H}^4(t) + c^{-\frac{5}{2} }  \mathcal{H}^5(t)\right).
\end{aligned}
\end{equation*}
Accordingly, by employing \eqref{H:trivial:bound} and $c\geq 1$, again, we find that 
\begin{equation}\label{v:piece2}
\begin{aligned}
 \norm{ \Gamma}_{L^\infty_t L^3 } & \norm{(B, \Gamma)}_{L^2_t\dot{H}^1}
 \\
 &\lesssim   \mathcal{H}   (0) \left(  \norm{\Gamma_0}_{L^2} + 1 \right) e^{\mathcal{E}_0^2} \\
  & \quad+ \underbrace{(\norm{\Gamma_0}_{L^2}+1 )\left(  c^{-\frac{1}{2} }  \mathcal{H}^2(t)+  c^{-1 }  \mathcal{H}^3(t)   + c^{- \frac{3}{2} }  \mathcal{H}^4(t) + c^{-\frac{5}{2} }  \mathcal{H}^5(t)  \right)e^{C\mathcal{E}_0^2}}_{\bydef P_*(\mathcal{H}(t))}.
\end{aligned}
\end{equation} 
	Consequently,  by incorporating \eqref{v:piece1} and \eqref{v:piece2} into  \eqref{velocity:A}, we end up with 
 \begin{equation}\label{velocity:final}
\norm{(\omega,\Omega)}_{L^\infty_t L^2 \cap L^2_t \dot{H}^1 } \lesssim\mathcal{H}   (0) \left(  \norm{\Gamma_0}_{L^2} + 1 \right) e^{C\mathcal{E}_0^2} + P_*(\mathcal{H}(t)),
 \end{equation}
 where the value of the constant $C>0$ can be adapted to guarantee the validity of the estimate.
 
 In the next step, we gather the high and low-frequency estimates of   electromagnetic fields. To that end, on the one hand, note   that \eqref{high:B}  yields 
 \begin{equation}\label{final:high:B}
   \begin{aligned}
   c^{-1}  \norm{ (E,B) }_{\widetilde{L} ^\infty _t \dot{B}^{\frac{5}{2} }_{2,1,>}  }  +  \norm{ (E,B) }_{\widetilde{L} ^2  _t \dot{B}^{\frac{5}{2} }_{2,1,>}  }   &\lesssim     \mathcal{H}(0) + P_*(\mathcal{H}(t)).
   \end{aligned}
\end{equation}   
On the other hand, \eqref{low:B} 
  entails that
\begin{equation*} 
 \begin{aligned}
 \norm{(E,B)}_{L^\infty _t\dot H^{\frac{3}{2} }  } + c \norm {E}_{L^2_t\dot{H}^\frac{3}{2}}  &+   \norm{B}_{L^2 _t\dot B^{ \frac{5}{2}}_{2,1,<}} \\
			&	 \lesssim  
			\mathcal{H}(0)
			+c^{ -\frac{1}{2}} \mathcal{H}^2(t)
			+ \norm{ \omega }_{L_t^{\infty}L^2 \cap L_t^{2 }\dot H^{1}} \norm {B }_{L^2_t \dot{H}^1} ^{\frac{1}{3}} \mathcal{H}^{\frac{2}{3}}(t).
 \end{aligned}
\end{equation*}   
Therefore, we employ \eqref{velocity:final}   to obtain that
\begin{equation*} 
 \begin{aligned}
 \norm{(E,B)}_{L^\infty _t\dot H^{\frac{3}{2} }  } &+ c \norm {E}_{L^2_t\dot{H}^\frac{3}{2}}  +   \norm{B}_{L^2 _t\dot B^{ \frac{5}{2}}_{2,1,<}} \\
			&	 \lesssim  
			\mathcal{H}(0)
			+  P_*(\mathcal{H}(t))
			+ \left(\mathcal{H}   (0) \left(  \norm{\Gamma_0}_{L^2} + 1 \right) e^{C\mathcal{E}_0^2} + P_*(\mathcal{H}(t)) \right)  \norm {B }_{L^2_t \dot{H}^1} ^{\frac{1}{3}} \mathcal{H}^{\frac{2}{3}}(t)\\
			&	 \lesssim  
			\mathcal{H}(0)
			+  (\mathcal{H}(t) +1 )P_*(\mathcal{H}(t))
			+ \mathcal{H}   (0) \left(  \norm{\Gamma_0}_{L^2} + 1 \right) e^{C\mathcal{E}_0^2}    \norm {B }_{L^2_t \dot{H}^1} ^{\frac{1}{3}} \mathcal{H}^{\frac{2}{3}}(t),
 \end{aligned}
\end{equation*}   
which implies, for any $\lambda>0$, that
\begin{equation*} 
 \begin{aligned}
 \norm{(E,B)}_{L^\infty _t\dot H^{\frac{3}{2} }  } &+ c \norm {E}_{L^2_t\dot{H}^\frac{3}{2}}  +   \norm{B}_{L^2 _t\dot B^{ \frac{5}{2}}_{2,1,<}} \\
			&	 \lesssim  
			\mathcal{H}(0)
			+  (\mathcal{H}(t) +1 )P_*(\mathcal{H}(t))
			+  \lambda^{-2} \mathcal{H} ^3  (0) \left(  \norm{\Gamma_0}_{L^2} + 1 \right)^3 e^{C\mathcal{E}_0^2}  \norm {B }_{L^2_t \dot{H}^1}
			+  \lambda\mathcal{H}(t) .
 \end{aligned}
\end{equation*}   
After that, we employ \eqref{B:parabolic ES:B***} which leads to the control
\begin{equation*} 
 \begin{aligned}
 &\norm{(E,B)}_{L^\infty _t\dot H^{\frac{3}{2} }  } + c \norm {E}_{L^2_t\dot{H}^\frac{3}{2}}  +   \norm{B}_{L^2 _t\dot B^{ \frac{5}{2}}_{2,1,<}} \\
			&\qquad	 \lesssim  
			\mathcal{H}(0) +  (\mathcal{H}(t) +1 )P_*(\mathcal{H}(t))
			\\
			& \qquad\quad +  \lambda^{-2} \mathcal{H} ^3  (0) \left(  \norm{\Gamma_0}_{L^2} + 1 \right)^3     \left( 1+ c^{-\frac{3}{2} }   \mathcal{H}(t) +  c^{-2}   \mathcal{H}^2(t)    + c^{-2}         \mathcal{H}^3(t)    \right) e^{C\mathcal{E}_0^2}
			+  \lambda\mathcal{H}(t).
 \end{aligned}
\end{equation*}   
 At last, in view of \eqref{H:trivial:bound}, it follows that
 \begin{equation}\label{piece:three:terms}
 \begin{aligned}
 &\norm{(E,B)}_{L^\infty _t\dot H^{\frac{3}{2} }  } + c \norm {E}_{L^2_t\dot{H}^\frac{3}{2}}  +   \norm{B}_{L^2 _t\dot B^{ \frac{5}{2}}_{2,1,<}} \\
			&\qquad\qquad	 \lesssim  
			\mathcal{H}(0)
			+  (\mathcal{H}(t) +1 )P_*(\mathcal{H}(t))  
			+  \lambda^{-2} \mathcal{H} ^3  (0) \left(  \norm{\Gamma_0}_{L^2} + 1 \right)^3 e^{C\mathcal{E}_0^2}
			\\  
			&\qquad\qquad \quad +  \lambda^{-2}   \left(  \norm{\Gamma_0}_{L^2} + 1 \right)^3     \left(  c^{-\frac{3}{2} }   \mathcal{H}^4(t) +  c^{-2}   \mathcal{H}^5(t)    + c^{-2}         \mathcal{H}^6(t)    \right)e^{C\mathcal{E}_0^2}
			+  \lambda\mathcal{H}(t). 
 \end{aligned}
\end{equation}

As for the remaining piece involving $\nabla\times B-j$ which is required to construct the functional $\mathcal{H}(t)$, it is dealt with by  recasting the bound from Proposition \ref{decay-E}, for $p=\infty$ and $s=\frac{1}{2}$, combined with \eqref{observ.1}, to find that 
   \begin{equation*}
	\norm { \nabla \times B  - j }_{L^\infty  _t \dot{H}^ \frac{1}{2}}\lesssim    \mathcal{H}(0) +  c^{-\frac{1}{2}}\mathcal{H}^2(t)  +  c^{-1}  \mathcal{H}^3(t),
	\end{equation*}
	whereby, obtaining that
	\begin{equation}\label{piece:remaining}
	\norm { \nabla \times B  - j }_{L^\infty  _t \dot{H}^ \frac{1}{2}}\lesssim    \mathcal{H}(0) + P_*  (\mathcal{H}(t)).
	\end{equation}
	
	All in all, gathering the bounds
	\eqref{final:B-l2H1}, \eqref{E-AAA:2}, \eqref{final:ES-Gamma3}, \eqref{velocity:final}, \eqref{final:high:B}, \eqref{piece:three:terms} and \eqref{piece:remaining} to construct the functional $\mathcal{H}(t)$
	and choosing $\lambda$ small enough in   such a way that the term $ \lambda \mathcal{H}(t)$ can be absorbed by the left-hand side of the final estimate, we end up with the bound
	\begin{equation*}
	\mathcal{H}(t) \leq C_0 + \mathcal{P}(\mathcal{H}(t)),
\end{equation*}	 
for all $t\geq 0,$ where we set
	\begin{equation}\label{C0:def}
	C_0 \bydef  C \mathcal{H}(0) \left(1 + \mathcal{H}(0) \right)^2    \left(1+  \norm{\Gamma_0}_{L^2}       \right)^3 e^{C\mathcal{E}_0^2},  
\end{equation}	  
and
$$ \mathcal{P}(\mathcal{H}(t)) \bydef  C   \left(  \norm{\Gamma_0}_{L^2} + 1 \right)^3   \left(   c^{-\frac{1}{2} }   \mathcal{H}^2(t) +  c^{-\frac{1}{2} }   \mathcal{H}^3(t) + c^{-1 }   \mathcal{H}^4(t) +  c^{-\frac{3}{2} }   \mathcal{H}^5(t)    + c^{-2}         \mathcal{H}^6(t)    \right) e^{C\mathcal{E}_0^2}.
$$
We recall that the (possibly large) constant $C>0$ is universal. From now on, it is fixed.
	
	 The completion of the proof hinges now on a direct application of the following simple lemma.
	\begin{lem}\label{fix-point:lem}
	 Let $t\mapsto x(t) $ be a non-negative continuous  function defined for all $t\geq 0$. Consider another function $t\mapsto F(t)$
	 which is assumed to be   non-negative and increasing.
	  Further suppose that there is  $x_0 >0$ such that
	 $$ x(0)\leq x_0,  $$
	 and, for any $t\geq 0$, that
 $$x(t) \leq x_0 +F(x(t)) .$$
 If, moreover, $F$ satisfies the condition that 
\begin{equation*}
 F(2x_0) < x_0 , 
\end{equation*}  
 then, $x(t)$ enjoys the bound
  $$ x(t)< 2 x_0,$$
  for all $t\geq 0$.  
\end{lem}
\begin{proof}
Define the set 
$$ \mathcal{I}\bydef \left \{ t\geq 0: x(t) < 2x_0 \right\}.$$
This set is non-empty and open in $[0,\infty)$. In order to prove the   desired global bound, we only have to show that $\mathcal{I}$ is closed, as well.

 To that end, let $(t_n)_{n\in \mathbb{N}}$ be a sequence of elements in $\mathcal{I}$, converging to some limit point $t\in [0,\infty)$. Since $t_n \in \mathcal{I}$, we deduce, for all $n\in \mathbb{N} $,   that 
$$ x(t_n) \leq x_0 + F(x(t_n)) \leq x_0 + F(2x_0) < 2 x_0 .$$
By continuity of $t\mapsto x(t)$, taking $n \to \infty$ in the foregoing inequalities yields that $t\in \mathcal{I}$, thereby showing that $\mathcal{I}$ is closed and completing the proof of the lemma.
\end{proof}

We are now back   to the proof of Theorem \ref{Thm:1}. By applying Lemma \ref{fix-point:lem} with $F(t)= \mathcal{P}(t)$, we arrive at the conclusion, for all $t\geq 0,$ that  
	\begin{equation*}
	\mathcal{H}(t) \leq 2 C_0  ,
	\end{equation*}
	 as soon as      
	  \begin{equation}\label{CD:1**}
	    \mathcal{P}(2C_0) < C_0.  
	  \end{equation}
	In particular, recall that
	\begin{equation*}
		\lim_{c\to\infty}\mathcal{P}(2C_0)=0,
	\end{equation*}
	which implies the existence of another constant $c_0>0$, which only depends on the initial data, such that \eqref{CD:1**} is satisfied for all $c> c_0$. It then follows that the all norms involved in the construction of $\mathcal{H}(t)$ are bounded.
	
	Finally, in order to complete the justification of all uniform bounds claimed in the statement of Theorem \ref{Thm:1}, we only need to observe that the control of $j$ in $L^\infty_tL^2 \cap L_t^2\dot H^\frac 12$ follows directly from an application of Proposition \ref{decay-E}. This completes the proof of Theorem \ref{Thm:1} in the case of rough profiles.

\subsubsection{Regular profiles}

We proceed with the completion of the proof of Theorem \ref{Thm:1} in the case of regular profiles. Specifically, our task now consists in showing that if, furthermore, we assume initially that $$(E_0^c,B_0^c) \in  \dot{B}^{\frac{5}{2}}_{2,1},$$
  uniformly in $ c>0$, then the regularity of $(E,B)$ in $\dot{B}^{\frac{5}{2}}_{2,1}$ is propagated for all times $t>0$, uniformly in $c \in (c_0,\infty)$.
Note that this cannot be done  by a direct application of the energy estimates from Lemma \ref{lemma.reg-persistence}. Instead, we need to exploit the techniques from Lemmas \ref{lemma-high F} and \ref{low:freq:estimates}.

   Again, for simplicity, we will henceforth omit the index ``$c$'' referring to the dependence of the solution on the speed of light.
   Also, let us point out that all the Lebesgue spaces in time are now taken over the whole positive real-line $\mathbb{R}^+$.
    Moreover, the constant  $C_0$  defined in \eqref{C0:def}  will be allowed to change from one line to another, as long as that change only involves norms of the initial data which remain  uniformly bounded in $c$.

\emph{Preliminary bounds on magnetic fields.}
The control of   high frequencies is deduced, again,  from  Lemma \ref{lemma-high F} with a suitable choice of parameters. To see that, we   proceed with a bootstrap argument by first recasting the bound
\begin{eqnarray*}
	 \norm{ (E,B) }_{\widetilde{L} ^4_t  \dot{B}^{\frac{5}{2} }_{2,1,>}  }   \lesssim   c^{-\frac{1}{2}}\norm { (E_0,B_0)}_{\dot{B}^\frac{5}{2} _{2,1,>}}  + \norm u_{L^\infty_t \dot{H} ^1 \cap L^2_t \dot{H} ^2  }   \left(  \norm{ B }_{\widetilde{L} ^2_t \dot{B}^{ \frac{5}{2}}_{2,1,>}  } +  \norm{ B }_{L ^2_t \dot{B}^{ \frac{5}{2}}_{2,1,<} }    \right),
	\end{eqnarray*}  
from \eqref{E_>:5/2}, by setting $ q=4$ therein.
Therefore, recalling that we have already established the uniform estimates
 \begin{equation}\label{uniform bounds*}
	 \begin{aligned}
	 	u &\in L^\infty_t\dot{H}^1 \cap L^2_t \dot{H}^2 ,
		\\
		B &\in \widetilde{L}^2_{t}\dot{B}^{\frac{5}{2}}_{2,1,>} \cap L^2_{t}\dot{B}^{\frac{5}{2}}_{2,1,<} \cap L^\infty_t L^2 \cap L^2_t \dot{H}^1,
		\\
		c^{-1}B &\in L^\infty _{t}\dot{B}^{\frac{5}{2}}_{2,1,<}, 
	 \end{aligned}
 \end{equation}
one directly deduces that 
   $$B\in \widetilde{L} ^4_t  \dot{B}^{\frac{5}{2} }_{2,1,>}  ,$$
   uniformly with respect to $ c\in (c_0,\infty).$

   Now, we take care of the low frequencies by first applying the last estimate from Lemma \ref{cor:parabolic:maxwell} with values
 $$ s=2, \quad  m=4 \quad \text{and} \quad q= 1, $$ 
 to find that 
 \begin{equation*}
 	\norm {B}_{L^4_t \dot B^{\frac{5}{2}}_{2,1,<}} \lesssim \norm {(E_0,B_0)}_{\dot B^2_{2,1}} + \norm { P(u\times B)}_{L^4_t \dot B^{\frac{3}{2}}_{2,1}}.
 \end{equation*}
 Therefore, by further employing classical product laws, which are contained in Lemma \ref{paradifferential:1}, and interpolation inequalities, we obtain that 
 \begin{equation*}
 	\begin{aligned}
 		\norm {B}_{L^4_t \dot B^{\frac{5}{2}}_{2,1,<}} 
 		&\lesssim \norm {(E_0,B_0)}_{\dot B^2_{2,1}} + \norm { u}_{L^4_t \dot B^{\frac{3}{2}}_{2,1}}\norm {B}_{L^\infty _t \dot B^{\frac{3}{2}}_{2,1}}
 		\\
 		&\lesssim \norm {(E_0,B_0)}_{\dot B^2_{2,1}} + \norm { u}_{L^\infty _t \dot H^1\cap L^2 _t \dot H^2}\norm {B}_{L^\infty _t \dot B^{\frac{3}{2}}_{2,1}}
 		\\
 		&\lesssim \norm {(E_0,B_0)}_{\dot B^2_{2,1}} + \norm { u}_{L^\infty _t \dot H^1\cap L^2 _t \dot H^2}\Big( \norm {B}_{L^\infty _t \dot B^{\frac{3}{2}}_{2,1,>}} + \norm {B}_{L^\infty _t \dot B^{\frac{3}{2}}_{2,1,<}} \Big)
 		\\
 		&\lesssim \norm {(E_0,B_0)}_{\dot B^2_{2,1}} + \norm { u}_{L^\infty _t \dot H^1\cap L^2 _t \dot H^2}\Big( c^{-1}\norm {B}_{L^\infty _t \dot B^{\frac{5}{2}}_{2,1,>}} + \norm {B}_{L^\infty _t \dot B^{\frac{3}{2}}_{2,1,<}} \Big).
 	\end{aligned}
 	 \end{equation*}
 	 Thus, in view of the uniform bounds recalled in \eqref{uniform bounds*}, in order to control low frequencies, it only remains to show that
	\begin{equation*}
		B\in L^\infty_t \dot B^ \frac{3}{2}_{2,1,<}.
	\end{equation*}
	Instead, we are going to prove the slightly better bound
	\begin{equation}\label{special:bound:1}
		B\in \widetilde{L} ^\infty _t \dot B^{\frac{3}{2}}_{2,1,<}.
	\end{equation}
	To that end, we employ  the low-frequency estimates from Lemma \ref{cor:maxwell} with   the  values
 $$      q=\infty \quad \text {and} \quad  r=\widetilde{r}=\widetilde{q}= 2$$
 to deduce that 
 \begin{equation*}
 	\norm {B}_{\widetilde{L} ^\infty _t \dot B^{\frac{3}{2}}_{2,1,<}} 
 		\lesssim \norm {(E_0,B_0)}_{\dot B^\frac{3}{2} _{2,1}} + \norm {P( u\times B)}_{\widetilde{L} ^2 _t \dot B^{\frac{3}{2}}_{2,1}}.
 \end{equation*}
 Therefore, by further applying the paraproduct estimate \eqref{paraproduct:1}, we obtain that 
 \begin{equation*}
 	\begin{aligned}
 		\norm {B}_{\widetilde{L} ^\infty _t \dot B^{\frac{3}{2}}_{2,1,<}} 
 		&\lesssim \norm {(E_0,B_0)}_{\dot B^\frac{3}{2} _{2,1}} + \norm { u}_{L^\infty _t \dot B^{1}_{2, \infty}}\norm {B}_{\widetilde{L} ^2 _t \dot B^{2}_{2,1}}. 	\end{aligned}
 	 \end{equation*}
 	 Accordingly, we conclude that the bound \eqref{special:bound:1} holds uniformly, by an application of the embeddings
 	 \begin{equation*}
 	 	L^2_t \dot B^1_{2,\infty} \cap L^2_t \dot B^{\frac{5}{2}}_{2,\infty}
		\hookrightarrow
		\widetilde L^2_t \dot B^1_{2,\infty} \cap \widetilde L^2_t \dot B^{\frac{5}{2}}_{2,\infty}
		\hookrightarrow \widetilde{L} ^2 _t \dot B^{2}_{2,1}
 	 \end{equation*} 
 	 combined with the bounds \eqref{uniform bounds*}.
 	 
 	 All in all, gathering the preceding high and low-frequency estimates, we have shown  that
 	 \begin{equation}\label{intermediate:bound:0}
 	 	B\in L^4_t \dot B^{\frac{5}{2}}_{2,1},
 	 \end{equation}
	uniformly in $c$.
 	 
 	Next, we   improve the  previous bound  to obtain
	 \begin{equation}\label{B-median:full}
	 	B\in \widetilde L^4_t \dot B^{\frac{5}{2}}_{2,1}.
	 \end{equation}
	Note that the same bound on high frequencies is already established at the start of this step and, thus, we only need to take care of the corresponding estimate on the remaining frequencies.  To that end, we first apply the low-frequency estimates from Lemma \ref{cor:maxwell} with values 
 	 $$      q=4 , \quad \widetilde{q}= \frac{4}{3} \quad \text {and} \quad  r=\widetilde{r}= 2$$
 	 to obtain that 
 	 \begin{equation*}
 	\norm {B}_{\widetilde{L} ^4_t \dot B^{\frac{5}{2}}_{2,1,<}} \lesssim \norm {(E_0,B_0)}_{\dot B^2_{2,1}} + \norm { P(u\times B)}_{\widetilde{L}^4_t \dot B^{\frac{3}{2}}_{2,1}}.
 \end{equation*}
 	 Therefore, by further employing \eqref{paraproduct:1} to estimate the product above, we find that 
 	  \begin{equation*}
 	\norm {B}_{\widetilde{L} ^4_t \dot B^{\frac{5}{2}}_{2,1,<}} \lesssim \norm {(E_0,B_0)}_{\dot B^2_{2,1}} + \norm {u}_{L^\infty _t \dot B^{1}_{2,\infty}}\norm { B}_{\widetilde{L}^4_t \dot B^{2}_{2,1}}.
 \end{equation*}
 	 Thus, one deduces, thanks to the embeddings 
 	 \begin{equation*}
 	 	L^4_t \dot B^{\frac{3}{2}}_{2,1} \cap L^4_t \dot B^{\frac{5}{2}}_{2,1} \hookrightarrow \widetilde{L}^4_t \dot B^{\frac{3}{2}}_{2,\infty} \cap \widetilde{L}^4_t \dot B^{\frac{5}{2}}_{2,\infty} \hookrightarrow \widetilde{L}^4_t \dot B^{2}_{2,1}  ,
 	 \end{equation*}
 	 that 
 	 \begin{equation*}
 	\norm {B}_{\widetilde{L} ^4_t \dot B^{\frac{5}{2}}_{2,1,<}} \lesssim \norm {(E_0,B_0)}_{\dot B^2_{2,1}} + \norm {u}_{L^\infty _t \dot H^{1} }\norm { B}_{L^4_t \dot B^{\frac{3}{2}}_{2,1} \cap L^4_t \dot B^{\frac{5}{2}}_{2,1}}.
 \end{equation*}
 Hence, due to the bounds \eqref{uniform bounds*} and \eqref{intermediate:bound:0}, we see that the right-hand side in the preceding estimate is finite as soon as  
 $B\in L^4_t \dot B^{\frac{3}{2}}_{2,1}$ . This bound turns out to be a consequence of the embeddings
 \begin{equation*}
 	L^\infty_t L^2 \cap L^2_t \dot {H}^1 \cap  L^4_t \dot B^{\frac{5}{2}}_{2,1} \hookrightarrow  L^4_t \dot B^{\frac{1}{2}}_{2,1}  \cap L^4_t \dot B^{\frac{5}{2}}_{2,1}  \hookrightarrow L^4_t \dot B^{\frac{3}{2}}_{2,1},
 \end{equation*}
 	 whereby $
 	 	B\in \widetilde{L} ^4_t \dot B^{\frac{5}{2}}_{2,1,<}.$ 
 	 In conclusion, we have shown that \eqref{B-median:full} holds
 	 uniformly with respect to the speed of light $c\in (c_0,\infty)$.

\emph{Propagation of initial regularity.} We are now in a position to show the propagation of the $ \dot{B}^{\frac{5}{2}}_{2,1}$ regularity of the electromagnetic field, uniformly with respect to the speed of light. 
   
   As before, we deal first with high frequencies by    recasting the second estimate from Lemma \ref{lemma-high F}, with the values 
   $$ q=\infty,\quad p=4,$$
    to find that 
	 \begin{eqnarray*}
	     \norm{ (E,B) }_{\widetilde{L} ^\infty _t  \dot{B}^{\frac{5}{2} }_{2,1,>}  }   \lesssim   \norm { (E_0,B_0)}_{\dot{B}^\frac{5}{2} _{2,1,>}}  + \norm u_{L^\infty_t \dot{H} ^1 \cap L^2_t \dot{H} ^2  }   \norm{ B }_{\widetilde{L} ^4_t \dot{B}^{ \frac{5}{2}}_{2,1 }  }  .
	\end{eqnarray*}  
	Thus, it follows that
	\begin{equation*}
		(E,B) \in \widetilde{L} ^\infty _t  \dot{B}^{\frac{5}{2} }_{2,1,>}  ,
	\end{equation*}
	uniformly in $ c\in (c_0,\infty)$, by virtue of the bounds \eqref{uniform bounds*} and \eqref{B-median:full}.

	As for   low frequencies, we proceed by applying the corresponding estimate    from Lemma \ref{cor:maxwell} with the values 
	$$r=\tilde{r}=2,\quad q=\infty, \quad \tilde{q}=\frac{4}{3},$$
	to find that 
	\begin{equation*} 
		\begin{aligned}
		 \norm{ E}_{\widetilde{L}^\infty _t \dot{B}^{\frac{5}{2}}_{2,1,<}}
			&\lesssim
			 \norm{ (E_0,B_0)}_{\dot{B}^{\frac{5}{2}}_{2,1}}  + 
			c^{-\frac{1}{2}} \norm{ P(u\times B) }_{\widetilde{L} _t^  4 \dot{B}^{\frac{5}{2}}_{2,1,< } }   \\
			&\lesssim
			 \norm{ (E_0,B_0)}_{\dot{B}^{\frac{5}{2}}_{2,1}}  + 
			  \norm{ P(u\times B) }_{\widetilde{L} _t^  4 \dot{B}^{2}_{2,1,< } },
			\end{aligned}
	\end{equation*} 
	 and  
	\begin{equation*} 
		\begin{aligned}
	  \norm{ B}_{\widetilde{L}^\infty _t \dot{B}^{\frac{5}{2}}_{2,1,<}}
			&\lesssim
			 \norm{ (E_0,B_0)}_{\dot{B}^{\frac{5}{2}}_{2,1}}  + 
			 \norm{ P(u\times B) }_{\widetilde{L} _t^  4 \dot{B}^{2}_{2,1 } } .
		\end{aligned}
	\end{equation*}
	Therefore,     by employing \eqref{paraproduct:1} to estimate the products above,  we obtain, for any $\varepsilon\in (0,\frac{1}{2})$, that
	\begin{equation*} 
		\begin{aligned}
	  \norm{ (E,B)}_{\widetilde{L}^\infty _t \dot{B}^{\frac{5}{2}}_{2,1,<}} 
			 	&\lesssim
			 \norm{ (E_0,B_0)}_{\dot{B}^{\frac{5}{2}}_{2,1}}  + 
			 \norm{ u}_{\widetilde{L} _t^  4 \dot{B}^{\frac{3}{2}- \varepsilon}_{2,\infty } }  \norm{ B}_{\widetilde{L} _t^  \infty \dot{B}^{2+\varepsilon}_{2,1 } }   .
		\end{aligned}
	\end{equation*}
	Hence,  utilizing the   interpolation inequalities
	\begin{equation*} 
		\begin{aligned}
  \norm{ u}_{\widetilde{L} _t^  4 \dot{B}^{\frac{3}{2}- \varepsilon}_{2,\infty } }    \lesssim  \norm{ u}_{L_t^  4 \dot{B}^{\frac{3}{2}- \varepsilon}_{2,\infty } }  & \lesssim  \norm{ u}_{L_t^  4 \dot{H}^{\frac{1}{2}}  }^{\varepsilon} \norm{ u}_{L _t^  4 \dot{H}^{\frac{3}{2}}  }^{1-\varepsilon}\\
  & \lesssim  \norm{ u}_{L _t^  \infty L^2 \cap L^2_t \dot{H}^{1}  }^{\varepsilon} \norm{ u}_{L_t^  \infty \dot{H}^{1}  \cap L^2_t\dot{H}^2 }^{1-\varepsilon} 
		\end{aligned}
	\end{equation*}
	and 
	\begin{equation*} 
		\begin{aligned}
 \norm{ B}_{\widetilde{L} _t^  \infty \dot{B}^{2+\varepsilon}_{2,1 } }  & \lesssim \norm{ B}_{\widetilde{L} _t^  \infty \dot{B}^{\frac{5}{2}}_{2,1 } } ^{\frac{1}{2}+ \varepsilon} \norm{ B}_{\widetilde{L} _t^  \infty \dot{B}^{\frac{3}{2}}_{2, \infty } }^{ \frac{1}{2}-\varepsilon}\\
 & \lesssim \norm{ B}_{\widetilde{L} _t^  \infty \dot{B}^{\frac{5}{2}}_{2,1 } } ^{\frac{1}{2}+ \varepsilon} \norm{ B}_{L_t^  \infty \dot{H}^{\frac{3}{2}}  }^{ \frac{1}{2}-\varepsilon},
		\end{aligned}
	\end{equation*}
	we deduce, for any $\lambda\in(0,1)$, that
  \begin{equation*} 
		\begin{aligned}
	  \norm{  (E,B)}_{\widetilde{L}^\infty _t \dot{B}^{\frac{5}{2}}_{2,1,<}}
			&\lesssim 
			 \norm{ (E_0,B_0)}_{\dot{B}^{\frac{5}{2}}_{2,1}}  + \lambda   \norm{ B}_{\widetilde{L}^\infty _t \dot{B}^{\frac{5}{2}}_{2,1,<}} \\
			 & \quad+ C_\lambda  \left( \norm{ u}_{L _t^  \infty L^2 \cap L^2_t \dot{H}^{1}  }^{\varepsilon} \norm{ u}_{L_t^  \infty \dot{H}^{1}  \cap L^2_t\dot{H}^2 }^{1-\varepsilon}  \right)^{ \frac{1}{\frac{1}{2} - \varepsilon}}   \norm{ B}_{L_t^  \infty \dot{H}^{\frac{3}{2}}  }   .
		\end{aligned}
	\end{equation*}
	Finally, choosing $\lambda$ small enough, we arrive at the conclusion that 
	 \begin{equation*} 
		\begin{aligned}
	  \norm{  (E,B)}_{\widetilde{L}^\infty _t \dot{B}^{\frac{5}{2}}_{2,1,<}}
			&\lesssim 
			 \norm{ (E_0,B_0)}_{\dot{B}^{\frac{5}{2}}_{2,1}} +    \left(   \norm{ u}_{L _t^  \infty L^2 \cap L^2_t \dot{H}^{1}  }^{\varepsilon} \norm{ u}_{L_t^  \infty \dot{H}^{1}  \cap L^2_t\dot{H}^2 }^{1-\varepsilon}  \right)^{ \frac{1}{\frac{1}{2} - \varepsilon}}   \norm{ B}_{L_t^  \infty \dot{H}^{\frac{3}{2}}  }   .
		\end{aligned}
	\end{equation*}
	Again, due to the bounds \eqref{uniform bounds*}, the right-hand side above is finite, uniformly with respect to $c\in (c_0,\infty)$, thereby yielding the desired control for the low frequencies of $E$ and $B$.
	 This completes the proof of Theorem \ref{Thm:1}.\qed

\section{Convergence and proof of Theorem \ref{Thm:CV}}

\label{section:closing:CV}

Let  $(u^c,E^c,B^c)_{c>c_0}$ and $(u,B)$ be the solutions of \eqref{Maxwell:system:*2} and \eqref{MHD*}, given by Theorem \ref{Thm:1} and Corollary \ref{Thm:MHD}, respectively.  We further introduce the fluctuations
$$\widetilde{u} \bydef u^c-u, \qquad\widetilde{B} \bydef B^c-B,$$
with a corresponding similar notation 
$\widetilde{u}_0$, $\widetilde{B}_0 $ for their initial data, and  the  time-dependent function
$$ f(t) \bydef \norm{ \nabla\Big( u(t), u^c(t),  B(t),B^c(t)\Big)}_{L^3}, $$
for all $t>0$.
In view of Theorem \ref{Thm:1} and Corollary \ref{Thm:MHD}, observe that   $f\in L^2(\mathbb{R}^+)$ with 
\begin{equation}\label{f-integrability}
\int_0^\infty f^2(\tau) d\tau \leq C_0,
\end{equation}
uniformly in $c\in (c_0,\infty)$,
for some constant $C_0>0$ depending only on the initial data.

We proceed now in four steps:
\begin{enumerate}
	\item
	The $L^2$ energy estimate.
	\item
	An interpolation argument.
	\item
	Convergence of the velocity in the endpoint space $\dot H^1$.
	\item
	Convergence of the magnetic field in the endpoint space $\dot H^\frac 32$.
\end{enumerate}

\subsection{The $L^2$ energy estimate}

First, it is readily seen that the fluctuations $ \widetilde{u}$, $\widetilde{B}$ are solutions of the perturbed MHD equations 
\begin{equation}\label{system-CV}
	\begin{cases}
		\begin{aligned}
		 \partial_t \widetilde{u} +u \cdot\nabla \widetilde{u}    - \nu \Delta \widetilde{u}+  \nabla \widetilde{p} = - \widetilde{u}\cdot\nabla u^c + \widetilde{B}\cdot \nabla B + B^c\cdot \nabla \widetilde{B} + \frac{1}{c} \partial_t E^c \times B^c,\\
		 	 \partial_t \widetilde{B} +u \cdot\nabla \widetilde{B}    - \frac{1}{\sigma} \Delta \widetilde{B}  =- \widetilde{u}\cdot\nabla B^c + \widetilde{B}\cdot \nabla u + B^c\cdot \nabla \widetilde{u} + \nabla \times \left(\frac{1}{c} \partial_t E^c\right).
		\end{aligned}
	\end{cases}
\end{equation}
Therefore, by virtue of  the identities
$$ \int_{\mathbb{R}^3} (u \cdot \nabla \widetilde{u}) \cdot \widetilde{u}   \; dx =   \int_{\mathbb{R}^3} (u\cdot \nabla \widetilde{B}) \cdot \widetilde{B}    \; dx = 0$$
and
$$ \int_{\mathbb{R}^3} (B^c\cdot \nabla \widetilde{B}) \cdot \widetilde{u}   \; dx +  \int_{\mathbb{R}^3} (B^c\cdot \nabla \widetilde{u}) \cdot \widetilde{B}    \; dx = 0,$$
 performing an $L^2$ energy estimate    yields, for all $t>0$, that
\begin{equation*}
\begin{aligned}
\|(\widetilde{u}, \widetilde{B})(t)\|_{L^2}^2+ \int_0^t \|(\widetilde{u}, \widetilde{B}) (\tau) \|_{ \dot{H}^1}^2 d\tau & \lesssim \|(\widetilde{u}_0, \widetilde{B}_0)\|_{L^2}^2 + \int_0^t \| ( \widetilde{u},\widetilde{B})(\tau)\|_{ L^3}^2 f(\tau) d\tau \\
& \quad +  \norm {\frac{1}{c} \partial_t E^c \times B^c  }_{ L^2_t \dot{H}^{-1} } ^2+  \norm { \frac{1}{c} \partial_t E^c   }_{L^2_t L^2 }  ^2\\
& \lesssim \|(\widetilde{u}_0, \widetilde{B}_0)\|_{L^2}^2 + \int_0^t  \| ( \widetilde{u},\widetilde{B})(\tau)\|_{ L^2}  \| ( \widetilde{u},\widetilde{B})(\tau)\|_{ \dot{H}^1}   f(\tau) d\tau \\
& \quad +  \norm {\frac{1}{c} \partial_t E^c }_{L^2_t \dot{H}^{\frac{1}{2}}} ^2\norm{ B^c  }_{ L^\infty_t L^2 } ^2+  \norm { \frac{1}{c} \partial_t E^c   }_{L^2_t L^2 }  ^2.
\end{aligned}
\end{equation*}
 Thus, we find,  for any $\lambda>0$, that
 \begin{equation*}
 \begin{aligned}
\|(\widetilde{u}, \widetilde{B})(t)\|_{L^2}^2 &+ \int_0^t \|(\widetilde{u}, \widetilde{B}) (\tau) \|_{ \dot{H}^1}^2 d\tau   \\ & \lesssim \|(\widetilde{u}_0, \widetilde{B}_0)\|_{L^2}^2  +  \lambda^{-1}\int_0^t  \| ( \widetilde{u},\widetilde{B})(\tau)\|_{ L^2}   ^2   f^2(\tau) d\tau + \lambda \int_0^t \|(\widetilde{u}, \widetilde{B} )(\tau) \|_{ \dot{H}^1}^2 d\tau \\
& \quad     +  \norm {\frac{1}{c} \partial_t E^c }_{L^2_t \dot{H}^{\frac{1}{2}}} ^2\norm{ B^c  }_{ L^\infty_t L^2 } ^2+  \norm { \frac{1}{c} \partial_t E^c   }_{L^2_t L^2 }  ^2  .
\end{aligned} 
 \end{equation*}
Hence, choosing $\lambda$  small enough, utilizing the energy inequality \eqref{energy-inequa} and applying Gr\"onwall's lemma, we deduce, for some universal constant  $C>0$, that
\begin{equation*}
\begin{aligned}
\|(\widetilde{u}, \widetilde{B})(t)\|_{L^2}^2 &+  \int_0^t \|(\widetilde{u}, \widetilde{B} )(\tau) \|_{ \dot{H}^1}^2 d\tau  \\
& \qquad \lesssim  \left(  \|(\widetilde{u}_0, \widetilde{B}_0)\|_{L^2} ^2   + (\mathcal{E}_0^2  +1)  \norm {\frac{1}{c} \partial_t E^c }_{L^2_t H^{\frac{1}{2}}}^2       \right)   \exp \left(C\int_0^t   f^2(\tau) d\tau  \right).
\end{aligned}
\end{equation*}
Consequently, by virtue of  \eqref{E-decay:A000}, \eqref{E-decay:A} and \eqref{f-integrability}, we end up with 
\begin{equation}\label{final:CV:L2}
\begin{aligned}
\sup_{\tau\in [0,\infty)}\|(\widetilde{u}, \widetilde{B})(\tau)\|_{L^2}^2 &+  \int_0^ \infty \|(\widetilde{u}, \widetilde{B} )(\tau) \|_{ \dot{H}^1}^2 d\tau    \leq C_0 \underbrace{ \left(  \|(\widetilde{u}_0, \widetilde{B}_0)\|_{L^2} ^2   + c^{-2}     \right)}_{\bydef \Theta_c},
\end{aligned}
\end{equation}
where $ C_0>0$ is another constant which depends only on the size of the initial data. In the sequel, this constant will be allowed to change value as long as it remains independent of $c$. This establishes the convergence of $\widetilde u$ and $\widetilde B$ to zero in the energy space $L_t^\infty L^2\cap L_t^2\dot H^1$. 

\subsection{An interpolation argument}\label{section:interpolation}

The results in this step are standard. Indeed, since both solutions $ (u^c,B^c)$ and $(u ,B )$ belong to the space 
$$ L^\infty(\mathbb{R}^+;\dot{H}^1 \times \dot{H}^\frac{3}{2}) \cap  L^2(\mathbb{R}^+;\dot{H}^2 \times \dot{H}^\frac{5}{2}),$$
uniformly in $c$, it follows, by interpolation with \eqref{final:CV:L2}, for any $s\in [0,1]$,  that 
\begin{equation*} 
\begin{aligned}
\sup_{\tau\in [0,\infty)}\|(\widetilde{u}, \widetilde{B})(\tau)\|_{\dot{H}^s \times \dot{H}^{   \frac{3s}{2}}}^2 &+  \int_0^ \infty \|(\widetilde{u}, \widetilde{B} )(\tau) \|_{ \dot{H}^{s+1} \times \dot{H}^{ \frac{3s}{2} +1}}^2 d\tau   \leq C_0  \Theta_c^{1-s} ,
\end{aligned}
\end{equation*} 
whereby we deduce, for all $s\in [0,1)$, that 
\begin{equation*} 
\begin{aligned}
\lim_{c \rightarrow \infty}\left( \sup_{\tau\in [0,\infty)}\|(\widetilde{u}, \widetilde{B})(\tau)\|_{\dot{H}^s \times \dot{H}^{  \frac{3s}{2}}}^2  +  \int_0^ \infty \|(\widetilde{u}, \widetilde{B} )(\tau) \|_{ \dot{H}^{s+1} \times \dot{H}^{ \frac{3s}{2} +1}}^2 d\tau \right)  =0 .
\end{aligned}
\end{equation*}
We are now left   with proving the convergence of solutions in the spaces  corresponding to the endpoint case $s=1$, above.

\subsection{Convergence of the velocity in the endpoint space $\dot H^1$}

With the bounds established in the previous step, we are now going to show that the convergence of the velocity in the endpoint case $s=1$ is a straightforward consequence of the fact that the source term $\frac{1}{c}\partial_t E^c \times B^c $ in the momentum equation vanishes in $ L^2_tL^2$, as $c\to  \infty$. 

To that end, performing an $\dot{H}^1$ energy estimate for the first equation in \eqref{system-CV}, one sees that
\begin{equation*}
\begin{aligned}
 \sup_{\tau\in [0,\infty)}\| \widetilde{u} (\tau)\|_{ \dot{H}^1 }^2 + \int_0^\infty\| \widetilde{u} (\tau)\|_{ \dot{H}^2 }^2 d\tau \lesssim  \norm { \widetilde{u}_0}_{\dot{H}^1}^2 +  \norm {F }_{L^2_tL^2 }^2 + \left \|   \frac{1}{c}\partial_t E^c \times B^c \right \|_{L^2_tL^2 } ^2 ,
\end{aligned}
\end{equation*}
where we denote 
\begin{equation*}
\begin{aligned}
F  \bydef   - u \cdot \nabla \widetilde{u}    - \widetilde{u}\cdot \nabla  u^c   +  \widetilde{B}\cdot \nabla   B + B^c\cdot \nabla  \widetilde{B}   .
\end{aligned}
\end{equation*}
Moreover, by a direct application of standard paraproduct laws and interpolation inequalities, we infer that 
$$ 
\begin{aligned}
\norm {F }_{L^2_tL^2 } & \lesssim    \norm {(u,u^c)}_{L^4_t \dot{B}^{\frac{3}{2}}_{2,1}}   \norm {\widetilde{u}}_{L^4_t \dot{H}^1} +\norm { ( B,B^c)}_{ L^4_t   \dot{H}^1 }   \|\widetilde{B}\|_{ L^4 _t \dot{B}^\frac{3}{2}_{2,1} } \\
& \lesssim    \norm {(u,u^c)}_{   L^\infty _t \dot{H}^{1} \cap L^2_t \dot{H}^{2} }   \norm {\widetilde{u}}_{ L^\infty _t \dot{H}^{\frac{1}{2}} \cap L^2_t \dot{H}^{\frac{3}{2}}} +\norm { ( B,B^c)}_{  L^\infty _t \dot{H}^{\frac{1}{2}} \cap L^2_t \dot{H}^{\frac{3}{2}}}   \|\widetilde{B}\|_{  L^\infty _t \dot{H}^1 \cap L^2_t \dot{H}^{2}} ,
\end{aligned}
$$
and 
$$ 
\begin{aligned}
 \left\|   \frac{1}{c}\partial_t E^c \times B^c \right \|_{L^2_tL^2 }
 & \lesssim   \left\|   \frac{1}{c}\partial_t E^c  \right\|_{L^2_t  \dot{H}^\frac{1}{2} }    \|    B^c  \|_{L^\infty _t \dot{H}^1 } \\
  & \lesssim  \left \|   \frac{1}{c}\partial_t E^c \right \|_{L^2_t  \dot{H}^\frac{1}{2} }    \|    B^c  \|_{L^\infty _t   H^\frac{3}{2} } \\
  & \leq  C_0\left \|   \frac{1}{c}\partial_t E^c \right \|_{L^2_t  \dot{H}^\frac{1}{2} }   .
\end{aligned}   $$
Therefore, by virtue of \eqref{E-decay:A} and the bounds in Theorem \ref{Thm:1} and Corollary \ref{Thm:MHD}, we end up with 
$$  \sup_{\tau\in [0,\infty)}\| \widetilde{u} (\tau)\|_{ \dot{H}^1 }^2 + \int_0^\infty\| \widetilde{u} (\tau)\|_{ \dot{H}^2 }^2 d\tau \lesssim  \norm { \widetilde{u}_0}_{\dot{H}^1}^2 +   C_0\left(   \norm {\widetilde{u}}_{ L^\infty _t \dot{H}^{\frac{1}{2}} \cap L^2_t \dot{H}^{\frac{3}{2}}}^2 +   \|\widetilde{B}\|_{  L^\infty _t \dot{H}^1 \cap L^2_t \dot{H}^{2}}^2 + c^{-2}  \right).$$ 
Consequently, owing to the convergence results from the previous steps, we arrive at the conclusion that
\begin{equation*} 
\begin{aligned}
\lim_{c \rightarrow \infty}\left( \sup_{\tau\in [0,\infty)}\| \widetilde{u} (\tau)\|_{\dot{H}^1 }^2  +  \int_0^ \infty \| \widetilde{u} (\tau) \|_{ \dot{H}^{2} }^2 d\tau \right)  =0 ,
\end{aligned}
\end{equation*}
for we are assuming that $\widetilde u_0$ vanishes in $\dot H^1$.

\subsection{Convergence of the magnetic field in the endpoint space $\dot H^\frac 32$}
   
One could try to mimic the proof from the previous step and perform a $\dot{H}^\frac{3}{2}$ energy estimate for $\widetilde{B}$. This method would require a decay of $ \frac{1}{c} \partial_t E^c$ in $L^2_t \dot{H}^\frac{3}{2}$, which is not an available information here.

  Instead, our  proof below is inspired from a Compactness Extrapolation Lemma (see \cite[Lemma 1.4]{ah2}), which reduces the justification of the convergence in an endpoint setting to the analysis of the evanescence of some high frequencies.
   To see that,  we first fix $\varepsilon\in (0, 1 ) $ and pick any real number  $s\in [0,1-\varepsilon)$. Then,   by utilizing the results from Section \ref{section:interpolation}, we obtain that  
\begin{equation*}
\begin{aligned}
	\sup_{\tau \in [0,\infty)}  \int_{|\xi|< ( \Theta_c)^{  \frac{\varepsilon+ s-1}{3(1-s)} }} |\xi |^{ 3  }  \left| \mathcal{F}( \widetilde{B })(\tau ,\xi)\right|^2 d\xi
	& +  \int_0^\infty \int_{|\xi|< ( \Theta_c)^{  \frac{\varepsilon+ s-1}{3(1-s)} }} |\xi |^{5 }  \left| \mathcal{F}( \widetilde{B })(\tau ,\xi)\right|^2 d\xi d\tau
	\\ & \quad \leq  \Theta_c^{  \varepsilon +s-1 }  \left(  \| \widetilde {B}\|_{L^\infty_t \dot{H}^{  \frac{3s}{2}}}^2 + \| \widetilde {B}\|_{L^2_t \dot{H}^{ \frac{3s}{2} + 1}}^2 \right)\\
	&\quad \leq  C_0 \Theta_c^{   \varepsilon }.
\end{aligned}
\end{equation*}
For simplicity, we are now going to take $s =0 $ and $ \varepsilon = \frac{1}{4}$, thereby establishing that 
\begin{equation*} 
\begin{aligned}
\lim_{c \rightarrow \infty}\left( \sup_{\tau\in [0,\infty)}\| \mathds{1}_{\{|D|<\Theta_c^{-\frac{1}{4}}\}} \widetilde{B} (\tau)\|_{\dot{H}^\frac{3}{2} }^2  +  \int_0^ \infty \| \mathds{1}_{\{|D|< \Theta_c^{-\frac{1}{4}}\}} \widetilde{B} (\tau)  \|_{ \dot{H}^{\frac{5}{2}} }^2 d\tau \right)  =0 ,
\end{aligned}
\end{equation*}
which takes care of frequencies in $\{|\xi|<\Theta_c^{-\frac 14}\}$.

We now deal with  the high frequencies in $\{|\xi|\geq \Theta_c^{-\frac 14}\}$. To that end, we first recall, by Corollary \ref{Thm:MHD} and \eqref{RMK:HF}, that the bound  
$$B\in \widetilde{L}^\infty (\mathbb{R}^+;\dot{B}^\frac{3}{2}_{2,2})\cap  L^2(\mathbb{R}^+;\dot{H}^\frac{5}{2})$$ 
holds uniformly with respect to $c\in (c_0,\infty)$. Therefore, we find that
\begin{equation*} 
\begin{aligned}
\lim_{c \rightarrow \infty}&\Big( \sup_{\tau\in [0,\infty)}\|\mathds{1}_{\{|D|\geq  \Theta_c^{-\frac{1}{4}}\}} \widetilde{B} (\tau)\|_{\dot{H}^\frac{3}{2} }^2  +  \int_0^ \infty \|  \mathds{1}_{\{|D|\geq  \Theta_c^{-\frac{1}{4}}\}} \widetilde{B} (\tau) \|_{ \dot{H}^{\frac{5}{2}} }^2 d\tau \Big) \\
&  \leq \lim_{c \rightarrow \infty}\Big( \sup_{\tau\in [0,\infty)}\| \mathds{1}_{\{|D|\geq  \Theta_c^{-\frac{1}{4}}\}}  B^c(\tau)\|_{\dot{H}^\frac{3}{2} }^2   +  \int_0^ \infty \| \mathds{1}_{\{|D|\geq  \Theta_c^{-\frac{1}{4}}\}}  B^c (\tau) \|_{ \dot{H}^{\frac{5}{2}} }^2 d\tau \Big)\\
&\quad+\underbrace{ 
 \lim_{c \rightarrow \infty}\Big( \sup_{\tau\in [0,\infty)}\| \mathds{1}_{\{|D|\geq  \Theta_c^{-\frac{1}{4}}\}} B  (\tau)\|_{\dot{H}^\frac{3}{2} }^2   +  \int_0^ \infty \|  \mathds{1}_{\{|D|\geq  \Theta_c^{-\frac{1}{4}}\}} B (\tau)  \|_{ \dot{H}^{\frac{5}{2}} }^2 d\tau \Big)}_{=0},
\end{aligned}
\end{equation*}
where we exploited the principle described in \eqref{0-limit:HF}.

The treatment of   the terms involving $B^c$ in the right-hand side above requires yet another application of Lemma \ref{cor:maxwell}. To that end, first observing that
\begin{equation*}
	\begin{aligned}
		 \norm { B^c}_{L^\infty_t \dot{B}^\frac{3}{2}_{2,2,>}}& \leq  c^{-1} \norm { B^c}_{\widetilde{L} ^\infty_t \dot{B}^\frac{5}{2}_{2,1,>}} ,
		 \\
		  \norm { B^c}_{L^2_t \dot{B}^\frac{5}{2}_{2,2,>}}
		  &\leq  \norm { B^c}_{\widetilde{L} ^2_t \dot{B}^\frac{5}{2}_{2,1,>}} ,
	\end{aligned}
\end{equation*}
it is readily seen that the decay in the hyperbolic region, i.e.,   when   frequencies are localized in   $\{  |\xi| \gtrsim  \sigma c \} $, follows     from  \eqref{high:B} as soon as it is assumed  initially that 
$$ \lim _{c\rightarrow \infty} \left( c^{-1} \| (E^c_0,B^c_0) \|_{  \dot{B}^\frac{5}{2}_{2,1,>}} \right) =0. $$

Now, in order to study the vanishing of the   remaining frequencies of $B^c$, i.e., frequencies  localized in  $\{  \Theta_c^{-\frac{1}{4}} \leq  |\xi|  \lesssim  \sigma c \} $, we employ   the low-frequency estimate from Lemma \ref{cor:maxwell} with the values 
$$ r=\tilde{r}=\tilde{q}=2 \quad \text{and} \quad q\in \{2,\infty\},$$
 which leads to
\begin{equation*} 
		\begin{aligned}
	\|\mathds{1}_{\{|D|\geq  \Theta_c^{-\frac{1}{4}}\}}  B^c\|_{L^\infty _t \dot{B}^{\frac{3}{2}}_{2,2,<}} &+   \| \mathds{1}_{\{|D|\geq  \Theta_c^{-\frac{1}{4}}\}}  B^c\|_{L^2 _t \dot{B}^{\frac{5}{2}}_{2,2,<}} \\
			&\lesssim
			 \norm{  \mathds{1}_{\{|D|\geq  \Theta_c^{-\frac{1}{4}}\}} (E_0^c,B_0^c)}_{\dot{H}^{\frac{3}{2}} }  + 
			 \norm{ \mathds{1}_{\{|D|\geq  \Theta_c^{-\frac{1}{4}}\} }  P(u^c\times B^c)}_{L _t^  2 \dot{H}^{\frac{3}{2}}  } .
		\end{aligned}
	\end{equation*}
	Accordingly, due to the strong convergence of  $(E_0^c,B_0^c)_{c>0} $    in $\dot{H}^\frac{3}{2}$, it is then readily seen that the first term in the right-hand side above vanishes as $c\to  \infty$.
	
	 As for the second term in the right-hand side, it is controlled first by exploiting the localization in   high-frequencies to write that
\begin{equation*} 
		\begin{aligned}
  \norm{ \mathds{1}_{\{|D|\geq  \Theta_c^{-\frac{1}{4}}\} }  P(u^c\times B^c)}_{L _t^  2 \dot{H}^{\frac{3}{2}}  }  \leq  \Theta_c^{\frac{1}{8}} \norm{ P(u^c\times B^c) }_{L _t^  2 \dot{H}^{2  }  }.
		\end{aligned}
	\end{equation*}
Therefore, by utilizing the product law \eqref{paraproduct:1}, we arrive at the conclusion that 
\begin{equation*} 
		\begin{aligned}
  \norm{\mathds{1}_{\{|D|\geq  \Theta_c^{-\frac{1}{4}}\} }  P(u^c\times B^c)}_{L _t^  2 \dot{H}^{\frac{3}{2}}  } & \lesssim  \Theta_c^{\frac{1}{8}} \norm{ u^c  }_{L _t^  4 \dot{B}^{\frac{3}{2}  }_{2,1}  } \norm{ B^c  }_{L _t^  4 \dot{H}^{2  }  } \\
   &\lesssim  \Theta_c^{\frac{1}{8}} \norm{ u^c  }_{L _t^  4 \dot{B}^{\frac{3}{2}  }_{2,1}  } \norm{ B^c  }_{L ^  \infty_t \dot{H}^{\frac{3}{2}  } \cap L _t^  2 \dot{H}^{\frac{5}{2}  }  }.
		\end{aligned}
	\end{equation*}
	Hence, due to the bounds from Theorem \ref{Thm:1} and the fact that 
	$$\lim_{c\to \infty} \Theta_c=0,$$
	we deduce that the right-hand side above vanishes in the limit $c\to  \infty$, thereby concluding the proof Theorem \ref{Thm:CV}.
\qed

\bibliographystyle{plain} 
\bibliography{NSM_bib}

\begin{thebibliography}{10}

\bibitem{HHK11}
Hammadi Abidi, Taoufik Hmidi, and Sahbi Keraani.
\newblock On the global regularity of axisymmetric
  {N}avier-{S}tokes-{B}oussinesq system.
\newblock {\em Discrete Contin. Dyn. Syst.}, 29(3):737--756, 2011.

\bibitem{ag20}
Diogo Ars\'{e}nio and Isabelle Gallagher.
\newblock Solutions of {N}avier-{S}tokes-{M}axwell systems in large energy
  spaces.
\newblock {\em Trans. Amer. Math. Soc.}, 373(6):3853--3884, 2020.

\bibitem{ah}
Diogo Ars\'enio and Haroune Houamed.
\newblock Damped {S}trichartz estimates and the incompressible
  {E}uler--{M}axwell system.
\newblock {\em arXiv:2204.04277}, 2022.

\bibitem{ah2}
Diogo Ars\'enio and Haroune Houamed.
\newblock Stability analysis of two-dimensional ideal flows with applications
  to viscous fluids and plasmas.
\newblock {\em arXiv:2305.10148}, 2023.

\bibitem{aim15}
Diogo Ars\'{e}nio, Slim Ibrahim, and Nader Masmoudi.
\newblock A derivation of the magnetohydrodynamic system from
  {N}avier-{S}tokes-{M}axwell systems.
\newblock {\em Arch. Ration. Mech. Anal.}, 216(3):767--812, 2015.

\bibitem{as}
Diogo Ars\'{e}nio and Laure Saint-Raymond.
\newblock {\em From the {V}lasov-{M}axwell-{B}oltzmann system to incompressible
  viscous electro-magneto-hydrodynamics. {V}ol. 1}.
\newblock EMS Monographs in Mathematics. European Mathematical Society (EMS),
  Z\"{u}rich, 2019.

\bibitem{bcd11}
Hajer Bahouri, Jean-Yves Chemin, and Rapha\"{e}l Danchin.
\newblock {\em Fourier analysis and nonlinear partial differential equations},
  volume 343 of {\em Grundlehren der Mathematischen Wissenschaften [Fundamental
  Principles of Mathematical Sciences]}.
\newblock Springer, Heidelberg, 2011.

\bibitem{bl76}
J\"{o}ran Bergh and J\"{o}rgen L\"{o}fstr\"{o}m.
\newblock {\em Interpolation spaces. {A}n introduction}.
\newblock Grundlehren der Mathematischen Wissenschaften, No. 223.
  Springer-Verlag, Berlin-New York, 1976.

\bibitem{bis-book}
Dieter Biskamp.
\newblock {\em Nonlinear magnetohydrodynamics}, volume~1 of {\em Cambridge
  Monographs on Plasma Physics}.
\newblock Cambridge University Press, Cambridge, 1993.

\bibitem{YZ18}
Yuan Cai and Zhen Lei.
\newblock Global well-posedness of the incompressible magnetohydrodynamics.
\newblock {\em Arch. Ration. Mech. Anal.}, 228(3):969--993, 2018.

\bibitem{HFDZ17}
Hui Chen, Daoyuan Fang, and Ting Zhang.
\newblock Global axisymmetric solutions of three dimensional inhomogeneous
  incompressible {N}avier-{S}tokes system with nonzero swirl.
\newblock {\em Arch. Ration. Mech. Anal.}, 223(2):817--843, 2017.

\bibitem{HFDZ19}
Hui Chen, Daoyuan Fang, and Ting Zhang.
\newblock The global solutions of axisymmetric {N}avier-{S}tokes equations with
  anisotropic initial data.
\newblock {\em Z. Angew. Math. Phys.}, 70(6):Paper No. 166, 14, 2019.

\bibitem{D-book}
Peter~Alan Davidson.
\newblock {\em An introduction to magnetohydrodynamics}.
\newblock Cambridge Texts in Applied Mathematics. Cambridge University Press,
  Cambridge, 2001.

\bibitem{DS:1963}
Camillo De~Lellis and L\'{a}szl\'{o} Sz\'{e}kelyhidi, Jr.
\newblock The {E}uler equations as a differential inclusion.
\newblock {\em Ann. of Math. (2)}, 170(3):1417--1436, 2009.

\bibitem{D07}
Rapha\"el Dhanchin.
\newblock Axisymmetric incompressible flows with bounded vorticity.
\newblock {\em Uspekhi Mat. Nauk}, 62(3(375)):73--94, 2007.

\bibitem{DH21}
Pierre Dreyfuss and Haroune Houamed.
\newblock Uniqueness result for the 3-{D} {N}avier-{S}tokes-{B}oussinesq
  equations with horizontal dissipation.
\newblock {\em J. Math. Fluid Mech.}, 23(1):Paper No. 19, 24, 2021.

\bibitem{DWZ18}
Daoyuan Fang, Wenjun Le, and Ting Zhang.
\newblock Global solutions of 3{D} axisymmetric {B}oussinesq equations with
  nonzero swirl.
\newblock {\em Nonlinear Anal.}, 166:48--86, 2018.

\bibitem{GP02}
Isabelle Gallagher and Fabrice Planchon.
\newblock On global infinite energy solutions to the {N}avier-{S}tokes
  equations in two dimensions.
\newblock {\em Arch. Ration. Mech. Anal.}, 161(4):307--337, 2002.

\bibitem{gim}
Pierre Germain, Slim Ibrahim, and Nader Masmoudi.
\newblock Well-posedness of the {N}avier-{S}tokes-{M}axwell equations.
\newblock {\em Proc. Roy. Soc. Edinburgh Sect. A}, 144(1):71--86, 2014.

\bibitem{g14}
Loukas Grafakos.
\newblock {\em Classical {F}ourier analysis}, volume 249 of {\em Graduate Texts
  in Mathematics}.
\newblock Springer, New York, third edition, 2014.

\bibitem{g14:2}
Loukas Grafakos.
\newblock {\em Modern {F}ourier analysis}, volume 250 of {\em Graduate Texts in
  Mathematics}.
\newblock Springer, New York, third edition, 2014.

\bibitem{HHZ20}
Adalet Hanachi, Haroune Houamed, and Mohamed Zerguine.
\newblock On the global well-posedness of the axisymmetric viscous {B}oussinesq
  system in critical {L}ebesgue spaces.
\newblock {\em Discrete Contin. Dyn. Syst.}, 40(11):6473--6506, 2020.

\bibitem{HHZ23}
Adalet Hanachi, Haroune Houamed, and Mohamed Zerguine.
\newblock Remarks on the global well-posedness of the axisymmetric {B}oussinesq
  system with rough initial data.
\newblock {\em Commun. Pure Appl. Anal.}, 22(6):1918--1949, 2023.

\bibitem{ZH}
Zineb Hassainia.
\newblock On the global well-posedness of the 3{D} axisymmetric resistive {MHD}
  equations.
\newblock {\em Ann. Henri Poincar\'{e}}, 23(8):2877--2917, 2022.

\bibitem{HR10}
Taoufik Hmidi and Fr\'{e}d\'{e}ric Rousset.
\newblock Global well-posedness for the {N}avier-{S}tokes-{B}oussinesq system
  with axisymmetric data.
\newblock {\em Ann. Inst. H. Poincar\'{e} C Anal. Non Lin\'{e}aire},
  27(5):1227--1246, 2010.

\bibitem{HZ20}
Haroune Houamed and Mohamed Zerguine.
\newblock On the global solvability of the axisymmetric {B}oussinesq system
  with critical regularity.
\newblock {\em Nonlinear Anal.}, 200:112003, 26, 2020.

\bibitem{IK2011}
Slim Ibrahim and Sahbi Keraani.
\newblock Global small solutions for the {N}avier-{S}tokes-{M}axwell system.
\newblock {\em SIAM J. Math. Anal.}, 43(5):2275--2295, 2011.

\bibitem{QYZ17}
Quansen Jiu, Huan Yu, and Xiaoxin Zheng.
\newblock Global well-posedness for axisymmetric {MHD} system with only
  vertical viscosity.
\newblock {\em J. Differential Equations}, 263(5):2954--2990, 2017.

\bibitem{LZ15}
Zhen Lei.
\newblock On axially symmetric incompressible magnetohydrodynamics in three
  dimensions.
\newblock {\em J. Differential Equations}, 259(7):3202--3215, 2015.

\bibitem{lemarie:2016}
Pierre~Gilles Lemari\'{e}-Rieusset.
\newblock {\em The {N}avier-{S}tokes problem in the 21st century}.
\newblock CRC Press, Boca Raton, FL, 2016.

\bibitem{QY22}
Qiao Liu and Yixin Yang.
\newblock Global well-posedness of 3{D} axisymmetric {MHD}-{B}oussinesq system
  with nonzero swirl.
\newblock {\em J. Math. Fluid Mech.}, 24(3):Paper No. 72, 22, 2022.

\bibitem{LY18}
Yanlin Liu.
\newblock Global well-posedness of 3{D} axisymmetric {MHD} system with pure
  swirl magnetic field.
\newblock {\em Acta Appl. Math.}, 155:21--39, 2018.

\bibitem{YL22}
Yanlin Liu.
\newblock Solving the axisymmetric {N}avier-{S}tokes equations in critical
  spaces ({I}): {T}he case with small swirl component.
\newblock {\em J. Differential Equations}, 314:287--315, 2022.

\bibitem{LY23A}
Yanlin Liu.
\newblock Long-time asymptotics of axisymmetric {N}avier-{S}tokes equations in
  critical spaces.
\newblock {\em Calc. Var. Partial Differential Equations}, 62(3):Paper No. 97,
  34, 2023.

\bibitem{LX23}
Yanlin Liu and Li~Xu.
\newblock On the existence and structures of almost axisymmetric solutions to
  3-{D} {N}avier-{S}tokes equations.
\newblock {\em SIAM J. Math. Anal.}, 55(1):458--485, 2023.

\bibitem{MN}
Nader Masmoudi.
\newblock Global well posedness for the {M}axwell--{N}avier--{S}tokes system in
  2{D}.
\newblock {\em J. Math. Pures et Appliqu\'ees}, 93:559--571, 2010.

\bibitem{SY94}
Taira Shirota and Taku Yanagisawa.
\newblock Note on global existence for axially symmetric solutions of the
  {E}uler system.
\newblock {\em Proc. Japan Acad. Ser. A Math. Sci.}, 70(10):299--304, 1994.

\bibitem{s87}
Jacques Simon.
\newblock Compact sets in the space {$L^p(0,T;B)$}.
\newblock {\em Ann. Mat. Pura Appl. (4)}, 146:65--96, 1987.

\bibitem{WZ22}
Peng Wang and Zhengguang Guo.
\newblock Global well-posedness for axisymmetric {MHD} equations with vertical
  dissipation and vertical magnetic diffusion.
\newblock {\em Nonlinearity}, 35(5):2147--2174, 2022.

\bibitem{Yudovich1}
V.~I. Yudovich.
\newblock Nonstationary flow of an ideal incompressible liquid.
\newblock {\em Akademiya Nauk SSSR. Zhurnal Vychislitelnol Matematiki I
  Matematicheskoi Fiziki}, 3:1032--1066, 1963.

\bibitem{YG20}
Gaocheng Yue.
\newblock Global solutions to 3-{D} {N}avier-{S}tokes-{M}axwell system slowly
  varying in one direction.
\newblock {\em Nonlinear Anal. Real World Appl.}, 53:103071, 25, 2020.

\bibitem{YG22}
Gaocheng Yue.
\newblock On the well-posedness of 3-{D} {N}avier-{S}tokes-{M}axwell system
  with one slow variable.
\newblock {\em J. Math. Anal. Appl.}, 507(1):Paper No. 125747, 22, 2022.

\bibitem{YGZ16}
Gaocheng Yue and Chengkui Zhong.
\newblock On the global well-posedness to the 3-{D} {N}avier-{S}tokes-{M}axwell
  system.
\newblock {\em Discrete Contin. Dyn. Syst.}, 36(10):5817--5835, 2016.

\end{thebibliography}
\end{document}